\documentclass[11pt]{amsart}

\usepackage{epigamath}

\usepackage[english]{babel}

\numberwithin{equation}{section}

\usepackage[utf8]{inputenc}
\usepackage{amsmath}
\usepackage{amssymb}
\usepackage{amsfonts}
\usepackage{amsthm}
\usepackage{mathrsfs}
\usepackage{xcolor}
\usepackage{dsfont}
\usepackage{physics}
\usepackage{stmaryrd}
\usepackage{upgreek}

\usepackage{tikz-cd}
\usepackage{tikz}
\usepackage{tikzsymbols}
\usepackage{mathtools}
\usetikzlibrary{matrix}

\usepackage[shortlabels]{enumitem}
\setlist[enumerate,1]{label={\rm(\roman*)}, ref={\rm\roman*}} 

\newcounter{dummy}
\makeatletter
\newcommand\myitem[1][]{\item[{\rm(}#1\rm{)}]\refstepcounter{dummy}\def\@currentlabel{#1}}
\makeatother
\usepackage[matrix,arrow,cmtip]{xy} % curve, frame, graph, 2cell, all
\SelectTips{cm}{}
\newdir{(}{{}*!/-5pt/@^{(}}
\newdir{(x}{{}*!/-5pt/@_{(}}
\newdir{+}{{}*!/-9pt/{}}
\newdir{>+}{@{>}*!/-9pt/{}}
\entrymodifiers={+!!<0pt,\fontdimen22\textfont2>}
\theoremstyle{plain}
\newtheorem{theorem}{Theorem}[section]
\newtheorem{thmx}{Theorem}
\newtheorem{lemma}[theorem]{Lemma}
\newtheorem{applemma}[theorem]{Lemma}
\newtheorem{proposition}[theorem]{Proposition}
\newtheorem{corollary}[theorem]{Corollary}

\theoremstyle{definition}
\newtheorem{definition}[theorem]{Definition}

\newtheorem{example}[theorem]{Example}

\theoremstyle{remark}
\newtheorem{remark}[theorem]{Remark}

\newcommand{\NN}{\mathbb{N}} % Natural numbers
\newcommand{\ZZ}{\mathbb{Z}} % Integers
\newcommand{\QQ}{\mathbb{Q}} % Rationals
\newcommand{\RR}{\mathbb{R}} % Reals
\newcommand{\kk}{\mathds{k}} % Field
\newcommand{\one}{\mathds{1}} % Identity map or trivial line bundle or identity group

\newcommand{\q}{\; / \;}
\newcommand\inj{\hookrightarrow}
\newcommand\surj{\twoheadrightarrow}

\DeclareMathOperator{\Hom}{Hom}
\DeclareMathOperator{\Ker}{Ker} % Kernel
\DeclareMathOperator{\Coker}{Coker} % Cokernel
\DeclareMathOperator{\lcm}{lcm} % least common multiple
\DeclareMathOperator{\CHAR}{char} % Characteristic

\newcommand{\ba}{\mathbf{a}} % vector-a
\newcommand{\bb}{\mathbf{b}} % vector-b
\newcommand{\bd}{\mathbf{d}} % vector-d
\newcommand{\bB}{\mathbf{B}} % vector-B
\newcommand{\be}{\mathbf{e}} % vector-e
\newcommand{\bE}{\mathbf{E}} % exceptional-rays
\newcommand{\bw}{\mathbf{w}} % vector-w
\newcommand{\bv}{\mathbf{v}} % vector-v
\newcommand{\bu}{\mathbf{u}} % vector-u
\newcommand{\bm}{\mathbf{m}} % vector-m
\newcommand{\bX}{\mathbf{X}} % vector-X
\newcommand{\bY}{\mathbf{Y}} % vector-Y
\newcommand{\bS}{\mathbf{S}} % vector-S
\newcommand{\ua}{\mathrm{a}}
\newcommand{\ub}{\mathrm{b}}
\newcommand{\uu}{\mathrm{u}}
\newcommand{\uv}{\mathrm{v}}

\newcommand{\uB}{\mathrm{B}}

\newcommand\fa{\mathfrak{a}}

\newcommand\fc{\mathfrak{c}}
\newcommand\fI{\mathfrak{I}}
\newcommand\fj{\mathfrak{j}}
\newcommand\fM{\mathfrak{M}}
\newcommand\fx{\mathfrak{x}}

\newcommand\fs{\mathfrak{s}}

\newcommand\fY{\mathfrak{Y}}

\newcommand\sC{\mathscr{C}}

\newcommand\sF{\mathscr{F}}

\newcommand\sJ{\mathscr{J}}

\newcommand\sM{\mathscr{M}} 

\newcommand\sO{\mathscr{O}}
\newcommand\sR{\mathscr{R}}

\newcommand\sY{\mathscr{Y}}
\newcommand\sD{\mathscr{D}}

\renewcommand{\AA}{\mathbb{A}} % Affine space (NB! overwrites A with a ring which could be used in the bibliography!)

\DeclareMathOperator{\Spec}{Spec} % Spec
\DeclareMathOperator{\GSpec}{Spec} % Global Spec
\DeclareMathOperator{\bl}{\mathbf{Bl}} % Usual blow-up
\DeclareMathOperator{\tBl}{\mathfrak{Bl}} % Toroidal weighted blow-up
\DeclareMathOperator{\Bl}{\mathscr{B}l} % Nulti-weighted blow-up
\DeclareMathOperator{\Pic}{Pic} % Picard group
\DeclareMathOperator{\codim}{codim} % Codimension
\DeclareMathOperator{\ZR}{ZR} % Zariski-Riemann space
\DeclareMathOperator{\logord}{log-ord} % Log order
\DeclareMathOperator{\inv}{inv} % Log invariant
\DeclareMathOperator{\MC}{MC} % Maximal contact
\newcommand{\gp}{\mathrm{gp}} % Groupification of monoid
\newcommand{\GrpSch}{\mathrm{Grp-Sch}} 
\newcommand{\sm}{\mathrm{sm}} % smooth
\newcommand{\logsm}{\mathrm{log-sm}} % log smooth
\newcommand{\red}{\mathrm{red}} % reduced structure
\newcommand{\extd}{\mathrm{ext}} % Extended Rees algebra
\newcommand{\rig}{\mathrm{rig}} % rigidification

\newcommand{\Gm}{\mathbb{G}_m}

\DeclareMathOperator{\SProj}{\underline{\mathscr{P}roj}} % Stacky Proj
\newcommand\xtwoheadrightarrow[2][]{%
  \xrightarrow[#1]{#2}\mathrel{\mkern-14mu}\rightarrow
}

\newcommand{\ler}{\mathrm{ler}}
\newcommand{\lr}{\mathrm{lr}}
\newcommand{\clr}{\mathrm{clr}}
\newcommand{\std}{\mathrm{std}}
\newcommand{\ros}{\mathrm{ros}} % reduction of stabilizers
\newcommand{\destack}{\mathrm{destack}} % destackification

\makeatletter
\renewcommand*\env@matrix[1][*\c@MaxMatrixCols c]{%
  \hskip -\arraycolsep
  \let\@ifnextchar\new@ifnextchar
  \array{#1}}
\makeatother

\newcommand{\spref}[1]{\href{http://stacks.math.columbia.edu/tag/#1}{#1}}
\newcommand{\spcite}[1]{\cite[\spref{#1}]{stacks-project}}

\newcommand{\lra}{\longrightarrow}

\newcommand{\longinj}{\xhookrightarrow{\hphantom{aaa}}}

\def\lowsimeq{\vbox to 0pt{\vss\hbox{$\scriptstyle\simeq$}\vskip-1.5pt}}

\EpigaVolumeYear{8}{2024} \EpigaArticleNr{15} \ReceivedOn{July 15, 2022}
\InFinalFormOn{March 28, 2024}
\AcceptedOn{April 23, 2024}

\title{Logarithmic resolution via multi-weighted blow-ups}
\titlemark{Logarithmic resolution via multi-weighted blow-ups}

\author{Dan Abramovich}
\address{Department of Mathematics, Brown University, Box 1917, 151 Thayer Street, Providence, RI 02912, USA}
\email{dan\_abramovich@brown.edu}

\author{Ming Hao Quek}
\address{Department of Mathematics, Brown University, Box 1917, 151 Thayer Street, Providence, RI 02912, USA}
\email{ming\_hao\_quek@brown.edu}

\authormark{D.~Abramovich and M.\,H.~Quek}

\AbstractInEnglish{We first introduce and study the notion of multi-weighted blow-ups, which is later used to systematically construct an explicit yet efficient algorithm for functorial logarithmic resolution in characteristic zero, in the sense of Hironaka.

  Specifically, for a singular, reduced closed subscheme $X$ of a smooth scheme $Y$ over a field of characteristic zero, we resolve the singularities of $X$ by taking proper transforms $X_i \subset Y_i$ along a sequence of multi-weighted blow-ups $Y_N \to Y_{N-1} \to \dotsb \to Y_0 = Y$ which satisfies the following properties:
  \begin{enumerate}
    \item The $Y_i$ are smooth Artin stacks with simple normal crossing exceptional loci.
    \item At each step we always blow up the worst singular locus of $X_i$, and witness on $X_{i+1}$ an immediate improvement in singularities. 
    \item Finally, the singular locus of $X$ is transformed into a simple normal crossing divisor on $X_N$.
\end{enumerate}}

\MSCclass{14E15 (primary), 14A20, 14M25 (secondary)}
\KeyWords{Resolution, singularity, logarithmic geometry, stack, blow-up, fan, toric variety}

\acknowledgement{This research work was supported in part by funds from BSF grant 2018193 and NSF grant DMS-2100548. Both authors were also supported by Swedish Research Council grant No.~2016-06596, and D.~Abramovich was supported through 2021 by the Simons Fellowship No.~818406.}
\begin{document}

%%%%%%%%%%%%%%%%%%%%%%%%%%%%%%%
% Title page
%%%%%%%%%%%%%%%%%%%%%%%%%%%%%%%

\maketitle

\begin{prelims}

\DisplayAbstractInEnglish

\bigskip

\DisplayKeyWords

\medskip

\DisplayMSCclass

\end{prelims}

%%%%%%%%%%%%%%%%%%%%%
% Table of Contents
%%%%%%%%%%%%%%%%%%%%%

\newpage

\setcounter{tocdepth}{1}

\tableofcontents

%%%%%%%%%%%%%%%%%%%%%
% Content begins here
%%%%%%%%%%%%%%%%%%%%%

\section{Introduction}\label{intro}

Throughout this paper, we work over a field $\kk$ of characteristic zero.

\subsection{Recent techniques for resolution of singularities in characteristic zero}\label{1.1}

We revisit the celebrated theorem of Hironaka, see \cite{hironaka-resolution}, that one can resolve the singularities of a reduced, closed, singular subscheme $X$ of a smooth scheme $Y$ over $\kk$, in a way that is functorial with respect to smooth morphisms of such pairs $(X \subset Y)$. Over the years, the proof of this theorem has seen simplifications, for example by Bierstone--Milman,  see \cite{bierstone-milman-desingularization}, by Encinas--Villamayor, see \cite{encinas-villamayor-desingularization}, and by W{\l}odarczyk, see \cite{wlodarczyk-simple-hironaka-resolution}. Most recently, it was shown independently by Abramovich--Temkin--W{\l}odarczyk, see \cite{abramovich-temkin-wlodarczyk-weighted-resolution}, and by McQuillan, see \cite{mcquillan-resolution}, that one can do this by repeatedly blowing up the ``worst singular locus'', although one has to instead work with stack-theoretic \emph{weighted blow-ups} and admit \emph{smooth Deligne--Mumford stacks} as ambient spaces.

In addition, one typically requires, for the sake of applications, that the singular locus of $X$ is transformed under the resolution into a simple normal crossing (snc) divisor. This was a feature of Hironaka's theorem in \cite{hironaka-resolution}, although it was only recently in a different paper of Abramovich--Temkin--W{\l}odarczyk \cite{abramovich-temkin-wlodarczyk-principalization-of-ideals-on-toroidal-orbifolds} that logarithmic geometry was first accessed as a tool to account for this requirement, and one does so by encoding exceptional divisors as logarithmic structures.

We remark, however, that the resolution algorithm in \cite{abramovich-temkin-wlodarczyk-weighted-resolution} does not address the aforementioned requirement. It is thus natural to ask for an amalgamation of the two aforementioned techniques, as depicted below:

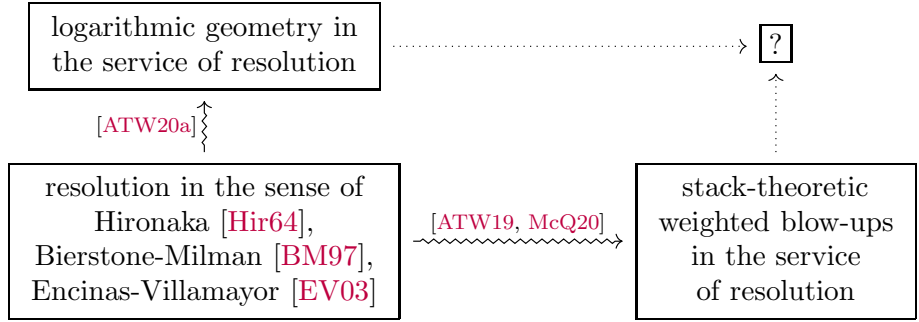
\[
    \begin{tikzcd}
    \framebox{\begin{tabular}{c}
        logarithmic geometry in \\
        the service of resolution
    \end{tabular}} \arrow[to=1-4, dotted] & & & \framebox{?}\\ 
    \framebox{\begin{tabular}{c}
        resolution in the sense of \\ Hironaka \cite{hironaka-resolution}, \\ 
        Bierstone--Milman \cite{bierstone-milman-desingularization}, \\
        Encinas--Villamayor \cite{encinas-villamayor-desingularization}
    \end{tabular}} \arrow[to=1-1, rightsquigarrow, "\textrm{\cite{abramovich-temkin-wlodarczyk-principalization-of-ideals-on-toroidal-orbifolds}}"] \arrow[to=2-4, rightsquigarrow, "\textrm{\cite{abramovich-temkin-wlodarczyk-weighted-resolution,mcquillan-resolution}}"] & & & \framebox{\begin{tabular}{c}
        stack-theoretic \\
        weighted blow-ups \\
        in the service \\
        of resolution
    \end{tabular}} \arrow[to=1-4, dotted]
    \end{tikzcd}
\]

This goal was recently realized by Quek in \cite{quek-weighted-log-resolution}, where weighted blow-ups in \cite{abramovich-temkin-wlodarczyk-weighted-resolution} are replaced by their logarithmic counterpart -- \emph{weighted toroidal blow-ups}. However, even if one takes a pair $(X \subset Y)$ from before as input for the algorithm in \cite{quek-weighted-log-resolution}, one is inevitably led to admit \emph{toroidal Deligne--Mumford stacks} as ambient spaces; \textit{i.e.}~these are logarithmically smooth over $\kk$ but not necessarily smooth over $\kk$. As a consequence, one cannot expect to resolve the singularities of $X$ solely via weighted toroidal blow-ups, and the best one can hope for at the end is \emph{toroidal singularities}, % (i.e. logarithmically smooth over $\kk$),
where the singular locus of $X$ is now transformed into a divisor with \emph{toroidal support}. Nonetheless, this is not a concern since one can then apply resolution of toroidal singularities; see \cite[Theorem 11*]{kempf-kundsen-faye-mumford-toroidal-embeddings}.
%\cite[Theorem 6.5.1]{wlodarczyk-resolution-toroidal}.

In this paper, we propose a different candidate for \framebox{?} above: namely, we use a construction of Satriano in \cite[Section 3]{satriano-canonical-artin-stacks-for-log-smooth-schemes}
to refine the weighted toroidal blow-ups in \cite{quek-weighted-log-resolution} to \emph{multi-weighted blow-ups}. This is carried out in Section~\ref{C:multi-weighted-blow-ups:canonical-aspects}, where certain multi-weighted blow-ups are realized as canonical Artin stacks over weighted toroidal blow-ups. The reader can also find, in Section~\ref{C:multi-weighted-blow-ups:local-aspects}, an account of local aspects of multi-weighted blow-ups. The key advantage of using multi-weighted blow-ups over weighted toroidal blow-ups is that we remain in the ideal realm of \emph{smooth ambient spaces}, and hence we can do without resolution of toroidal singularities at the end. However, the trade-off is that one has to work more broadly with \emph{Artin stacks} as ambient spaces. 

\begin{xpar}[Ambient spaces]\label{X:ambient-spaces}
The ambient spaces in the resolution algorithm of this paper are \emph{smooth, toroidal Artin stacks over $\kk$}. 

For the definition of a toroidal $\kk$-scheme (resp.\ more generally, a toroidal Artin stack over $\kk$), we refer the reader to \cite[Appendix B]{quek-weighted-log-resolution} (resp.\ \cite[Section 3.3]{abramovich-temkin-wlodarczyk-toroidal-destackification-kummer-blow-ups}) and the references therein. In brief, ``toroidal'' is a synonym for ``logarithmically regular'' or, equivalently in characteristic zero, ``logarithmically smooth over $\kk$''. 

With the exception of Section~\ref{4.1}, this paper mostly deals with toroidal Artin stacks over $\kk$ that are also smooth -- these are simply pairs of the form $(Y,E)$, where $Y$ is a smooth Artin stack over $\kk$ and $E \subset Y$ is either $\emptyset$ or a normal crossings (nc) divisor. Indeed, for such a pair $(Y,E)$, set $U := Y \smallsetminus E$, $j \colon U \inj Y$, and the logarithmic structure on $Y$ to be that dictated by $E$, \textit{i.e.} $\alpha_Y \colon (\sM_Y := j_\ast(\sO_U^\ast) \cap \sO_Y) \inj \sO_Y$. Then the logarithmic Artin stack $(Y,\sM_Y)$ is logarithmically smooth over $\kk$. If $E$ is a snc divisor, we say $(Y,\sM_Y)$ or $(Y,E)$ is \emph{smooth, strict toroidal}.\footnote{Here ``strict''  refers to ``strict normal crossings divisor'', a synonym for ``snc divisor''.}

Any smooth, toroidal Artin stack $Y$ over $\kk$ admits a smooth cover by a smooth $\kk$-scheme which, when endowed by the logarithmic structure given by the pull-back of $\sM_Y$, is a smooth, strict toroidal $\kk$-scheme. It is important for the reader to be aware that, since the discussions in this paper turn out to be local in the strict smooth topology, it suffices to study the case where $Y$ is a smooth, strict toroidal $\kk$-scheme; see Remark~\ref{R:case-of-artin-stacks}. 
\end{xpar}

\subsection{Main objectives}\label{1.2}

Consider a reduced, closed substack $X$ in a smooth, toroidal Artin stack $Y$ over $\kk$. We always regard $X$ as a logarithmic Artin stack over $\kk$ by pulling the logarithmic structure of $Y$ back to $X$. Such pairs $(X \subset Y)$ form the objects of a category, where a morphism between pairs $(\widetilde{X} \subset \widetilde{Y}) \to (X \subset Y)$ is a Cartesian square of stacks \[
    \begin{tikzcd}
    \widetilde{X} = X \times_Y \widetilde{Y} \arrow[to=1-2, hookrightarrow] \arrow[to=2-1] & \widetilde{Y} \arrow[to=2-2, "f"] \\
    X \arrow[to=2-2, hookrightarrow] & Y
    \end{tikzcd}
\]
for a strict, smooth, and surjective morphism $f \colon \widetilde{Y} \to Y$.\footnote{We note that this Cartesian square of stacks is also Cartesian in the category of logarithmic stacks since $f$ is strict.} {We refer to such a morphism as a \emph{strict, smooth, surjective morphism of pairs}. At times we might drop surjectivity as a condition, in which case we say $f$ is a \emph{strict, smooth morphism of pairs}. 

For such a pair $X \subset Y$, let $X^{\logsm}$ denote the toroidal (\textit{i.e.} logarithmically smooth) locus of $X$, and we typically require $X$ to be \emph{generically toroidal}; \textit{i.e.} $X^{\logsm}$ is dense in $X$. Since $Y$ is smooth and toroidal, $X^\logsm$ is contained in the smooth locus $X^\sm$ of $X$.} The primary goal of this paper is the following. 

\begin{thmx}[Logarithmic embedded resolution]\label{T:logarithmic-embedded-resolution}
  Given a reduced, generically toroidal, 
  closed substack $X$ of a smooth, toroidal Artin stack $Y$ over $\kk$, there exists a canonical sequence of multi-weighted blow-ups $\Pi \colon Y_N$ $\xrightarrow{\pi_N} Y_{N-1} \xrightarrow{\pi_{N-1}} \dotsb \xrightarrow{\pi_1} Y_0 = Y$, together with proper transforms $X_i \subset Y_i$ of $X$, such that: \begin{enumerate}
    \item\label{T:ler-1} $X_N$ is a smooth, toroidal Artin stack over $\kk$, and so is each $Y_i$;  
    \item $\Pi$ is an isomorphism over $X^{\logsm} \subset X$; 
    \item $\Pi^{-1}(X \smallsetminus X^{\logsm})$ is a snc divisor on $X_N$; 
    \item each $\pi_i$ is birational, surjective, universally closed, and factors as $Y_i \to \bY_i \to Y_{i-1}$, where $Y_i \to \bY_i$ is a good moduli space of $Y_i$ relative to $Y_{i-1}$, $\bY_i$ is normal, and $\bY_i \to Y_{i-1}$ is a schematic blow-up $($whence birational and projective$)$.
\end{enumerate}
{This procedure $(X \subset Y) \mapsto (X_N \subset Y_N)$ is functorial with respect to strict, smooth morphisms of such pairs $X \subset Y$ $($whether or not surjective$)$. Moreover, if\, $Y$ is strict toroidal, then so is $X_N$.}
\end{thmx}

If $Y$ is smooth with the trivial logarithmic structure, then $X^{\logsm} = X^{\sm}$, so in particular we recover logarithmic embedded resolution in the classical sense. We remark too that the logarithmic structure on each $Y_i$ is defined by combining ``the new exceptional divisors on $Y_i$'' with the logarithmic structure on $Y_{i-1}$.

In Section~\ref{4.3} we will prove Theorem~\ref{T:logarithmic-embedded-resolution}, as well as the other theorems in this introduction. {As anticipated in Section~\ref{1.1}, the theorems in this introduction are deduced from the analogous theorems in \cite{quek-weighted-log-resolution} because the multi-weighted blow-ups in this paper are obtainable from the weighted toroidal blow-ups in \cite{quek-weighted-log-resolution} via Satriano's construction in \cite[Section 3]{satriano-canonical-artin-stacks-for-log-smooth-schemes}. For example, Theorems~\ref{T:logarithmic-embedded-resolution} and \ref{T:invariant-drops-intro} are deduced\footnote{More precisely, in this way, one can only deduce Theorems~\ref{T:logarithmic-embedded-resolution} and \ref{T:invariant-drops-intro} under the additional hypothesis that $X \subset Y$ is of pure codimension. We supply an additional argument in Section~\ref{4.3} that does away with that hypothesis. The same argument also works to prove \cite[Theorem 1.1]{quek-weighted-log-resolution} without that hypothesis.} in this way from \cite[Theorem 1.1]{quek-weighted-log-resolution}.}

Concretely, one obtains Theorem~\ref{T:logarithmic-embedded-resolution} by taking at the $(i+1)^{\textup{th}}$ step the multi-weighted blow-up of $Y_i$ along the ``worst singular locus of $X_i$''. We give a formal statement of this procedure in Theorem~\ref{T:invariant-drops-intro} below. To give the reader a sense of what the ``worst singular locus of $X_i$'' is, we recall the following notion from \cite[Section 6.1]{quek-weighted-log-resolution}. 

For a point $p \in \abs{Y}$, one can associate an \emph{invariant} of $X \subset Y$ at $p$ (see Section~\ref{X:log-invariant}), denoted by $\inv_p(X \subset Y)$, which is, simply put, a non-decreasing finite sequence of non-negative rational numbers, whose last entry could be $\infty$. We can well-order the set of all such invariants of pairs at points in $Y$ by the lexicographic order $<$, but with a caveat: our lexicographic order considers the truncation (from the end) of a sequence to be strictly larger than the sequence itself (see Section~\ref{X:well-ordered-set}).

Let $I$ be the underlying ideal of $X \subset Y$. To give the reader a sense of our invariant, let us mention some of its properties:
\begin{enumerate}[label=(\alph*),ref=\alph*]
    \item\label{1.2a} If $X = Y$, then $\inv_p(X \subset Y) = ()$ for every $p \in \abs{Y}$.
    \myitem[a$'$]\label{1.2a'} Otherwise, for every $p \in \abs{Y}$, $\inv_p(X \subset Y) = (0)$ if and only if $p \notin \abs{X}$. If $p \in \abs{X}$, $\inv_p(X \subset Y)$ is greater than or equal to the sequence $(1,\dotsc,1)$ of length equal to $\codim_p(X \subset Y)$ ($:=$ height of $I_p$), and equality holds if and only if $X$ is smooth and toroidal at $p$; see Remark~\ref{R:log-invariant}\eqref{R:li-4}.
    \item\label{1.2b} The first entry of $\inv_p(X \subset Y)$ is the logarithmic order of $I$ at $p$. In particular, it is either a non-negative integer or $\infty$.
    \item\label{1.2c} 
      $\inv$ is upper semi-continuous on $Y$; see Remark~\ref{R:log-invariant}\eqref{R:li-2}.
    \item\label{1.2d} $\inv$ is functorial for logarithmically smooth morphisms of pairs $X \subset Y$; see Remark~\ref{R:log-invariant}\eqref{R:li-3}.
\end{enumerate}

We set $\max\inv(X \subset Y) := \max_{p \in \abs{X}}{\inv_p(X \subset Y)}$, which is functorial with respect to logarithmically smooth, surjective morphisms of pairs $X \subset Y$; see Remark~\ref{R:log-invariant}\eqref{R:li-3}. Then property \eqref{1.2a'} above suggests that the ``worst singular locus of $X$'' can be loosely interpreted as the closed substack of $X$ consisting of points $p \in \abs{X}$ such that $\inv_p(X \subset Y) = \max\inv(X \subset Y)$. For a precise definition of this notion of ``worst singular locus of $X$'', see Section~\ref{X:associated-center}, as well as the proof of Theorem~\ref{T:logarithmic-embedded-resolution}. 

We can now state our next theorem, which when \emph{suitably} iterated, gives Theorem~\ref{T:logarithmic-embedded-resolution}. 

\begin{thmx}\label{T:invariant-drops-intro}
Given a reduced, logarithmically singular, closed substack $X$ of a smooth, toroidal Artin stack $Y$ over $\kk$, there exists a canonical multi-weighted blow-up $\pi \colon Y' \to Y$, with proper transform $X' \subset Y'$ of $X$, such that: \begin{enumerate}
    \item\label{T:idi-1} $Y'$ is a smooth, toroidal Artin stack over $\kk$;
    \item\label{T:idi-2} $\max\inv(X' \subset Y') < \max\inv(X \subset Y)$;
    \item\label{T:idi-3} $\pi$ is an isomorphism away from the closed substack of $X$ consisting of points $p \in \abs{X}$ such that $\inv_p(X \subset Y) = \max\inv(X \subset Y)$;
    \item\label{T:idi-4} $\pi$ is birational, surjective, universally closed, and factors as $Y' \to \bY' \to Y$, where $Y' \to \bY'$ is a good moduli space relative to $Y$, $\bY'$ is normal, and $\bY' \to Y$ is a schematic blow-up $($whence birational and projective$)$.
\end{enumerate}
This procedure ${\sF_{\ler} \colon} (X \subset Y) \mapsto (X' \subset Y')$ is functorial with respect to strict, smooth, surjective morphisms of such pairs $X \subset Y$.
\end{thmx}

One should interpret part \eqref{T:idi-3} as saying that $\pi$ is a multi-weighted blow-up along the ``worst singular locus of $X$'' and part \eqref{T:idi-2} as saying that there is an immediate improvement of singularities after the multi-weighted blow-up. As mentioned before, we emphasize that the logarithmic structure on $Y'$ is defined by combining ``the new exceptional divisors on $Y'$'' with the logarithmic structure on $Y$. %\ChDan{(The notation $\sF_{\ler}(X \subset Y)$ is meant to suggest a ``functor of logarithmic embedded resolution".)}

In addition, these multi-weighted blow-ups in our algorithm are independent of further embeddings, in the following sense. 

\begin{thmx}[Re-embedding principle]\label{T:re-embedding-principle}
Let $X$ be reduced, logarithmically singular, closed substack of a smooth, toroidal Artin stack $Y$ over $\kk$. Let $Y_1 = Y \times \AA^1$, 
where $\AA^1 := \Spec(\kk[x_0])$ is given the trivial logarithmic structure. We embed $X \subset Y = V(x_0) \subset Y_1$. Then: \begin{enumerate}
    \item\label{T:rep-1} For every $p \in \abs{X}$, we have $\inv_p(X \subset Y_1) = (1,\inv_p(X \subset Y))$.
    \item\label{T:rep-2}  Let $\sF_{\ler}(X \subset Y) =: (X' \subset Y')$ and $\sF_{\ler}(X \subset Y_1) =: (X_1' \subset Y_1')$ {as in Theorem~\ref{T:invariant-drops-intro}}. Then $Y'$ is canonically identified with the proper transform of\, $Y  \subset Y_1$ in $Y_1'$, under which $X' = X_1'$.
\end{enumerate}
\end{thmx}

We remark that part \eqref{T:rep-2} can be checked directly, albeit in a tedious fashion. We instead deduce part \eqref{T:rep-2} from \cite[Lemma 1.3(ii)]{quek-weighted-log-resolution} -- the analogous statement for the weighted toroidal blow-ups in \cite{quek-weighted-log-resolution}.

Lastly, Theorem~\ref{T:re-embedding-principle} enables us to prove Theorem~\ref{T:logarithmic-resolution} below by {embedding $X$, locally in the smooth topology, in a smooth $\kk$-scheme $Y$} (with the trivial logarithmic structure) in pure codimension, before applying Theorem~\ref{T:logarithmic-embedded-resolution} and patching these local resolutions.

\begin{thmx}[Logarithmic resolution]\label{T:logarithmic-resolution}
Given a reduced, pure-dimensional Artin stack $X$ of finite type over $\kk$, there exists a birational, surjective, and universally closed morphism $\Pi \colon X' \to X$ such that: \begin{enumerate}
    \item $X'$ is a smooth Artin stack over $\kk$; 
    \item $\Pi$ is an isomorphism over the smooth locus $X^{\sm}$ of $X$;
    \item $\Pi^{-1}(X \smallsetminus X^{\sm})$ is a snc divisor on $X'$;
    \item $X'$ admits a good moduli space $\bX'$ relative to $X$, and $\bX' \to X$ is birational and projective.
\end{enumerate}
This procedure $\sF_\lr \colon X \mapsto X'$ is functorial with respect to smooth morphisms.
\end{thmx}

As a final note, while the steps in Theorem~\ref{T:logarithmic-embedded-resolution} are more explicit and efficient than previous resolution algorithms, we emphasize that if $X$ in Theorem~\ref{T:logarithmic-resolution} is a scheme, then $X'$ is usually not a scheme, and its good moduli space relative to $X$ is not necessarily smooth over $\kk$. 

Nevertheless, we expect Theorem~\ref{T:logarithmic-embedded-resolution} to suffice for many computations in algebraic geometry that necessitate logarithmic resolution. For example, a motivic change of variables formula for Artin stacks, that is applicable to the context of Theorem~\ref{T:logarithmic-resolution}, was very recently developed by Satriano--Usatine; see \cite{satriano-usatine-motivic-integration-Artin-stacks}.

If one still insists on returning to the world of schemes, one can apply ``canonical reduction of stabilizers'' due to Edidin--Rydh, see \cite{edidin-rydh-reduction-of-stabilizers}, followed by ``destackification'' due to Bergh--Rydh,  see \cite{bergh-rydh-destackification}, to the output $X'$ in Theorem~\ref{T:logarithmic-resolution} and its snc divisor $\Pi^{-1}(X \smallsetminus X^{\sm})$. In this way one recovers \cite[Main Theorem I]{hironaka-resolution}. For concrete statements, we refer the reader to Section~\ref{5.2}.

\subsection{Acknowledgements}

Dan Abramovich thanks the generosity of ICERM and its Combinatorial Algebraic Geometry program for its hospitality during Spring 2021, Institut Mittag-Leffler and its Moduli and Algebraic Cycles program where he was Visiting Professor during September--October 2021, the Einstein Institute of Mathematics at Jerusalem for its hospitality during November--December 2021, and the Simons Foundation for its generous Simons Fellowship through 2021. 

Ming Hao Quek thanks Institut Mittag-Leffler for accommodating him and for providing a conducive research environment during September--November 2021. He thanks John Christian Ottem, Dan Petersen, and David Rydh for organizing the Moduli and Algebraic Cycles program and is deeply grateful to them for giving him the opportunity to present a seminar talk on this paper.

We thank the referee for an insightful reading of our paper, and in particular for pointing out an overstatement of functoriality in an earlier version of this paper.

Last but not least, the authors express our deepest gratitude to David Rydh and Matthew Satriano, with whom we had the privilege and the pleasure of discussing mathematics, and whose questions and suggestions have influenced some of the approaches undertaken in this paper. 

\setcounter{tocdepth}{1}
\tableofcontents

%\newpage
\section{Multi-weighted blow-ups: Local aspects}\label{C:multi-weighted-blow-ups:local-aspects}%\addtocontents{toc}{}

\subsection{Multi-weighted blow-ups on affine spaces}\label{2.1}

We first review the notion of fantastacks in \cite{geraschenko-satriano-toric-stacks-I}. In the process, we will also fix some notation.

\begin{definition}[Fantastacks]\label{D:fantastacks}
  Given a lattice $N$ with dual lattice $N^\vee = \Hom_\ZZ(N,\ZZ)$, let $\Sigma$ be a fan in $N$ and $\beta \colon \ZZ^\kappa \to N$ be a homomorphism with finite cokernel, satisfying the following conditions:
  \begin{enumerate}[label=(\alph*),ref=\alph*]
    \item\label{D:f-1} Every ray (\textit{i.e.} $1$-dimensional cone) of $\Sigma$ contains some $\beta(\be_i)$.
    \item\label{D:f-2} Every $\beta(\be_i)$ lies in the support of $\Sigma$.
\end{enumerate}
For a cone $\sigma \in \Sigma$, set $\widehat{\sigma}$ to be the cone in $\ZZ^\kappa$ spanned by those $\be_i$ such that $\beta(\be_i) \in \sigma$. Let $\widehat{\Sigma}$ denote the fan in $\ZZ^\kappa$ generated by $\lbrace \widehat{\sigma} \colon \sigma \in \Sigma \rbrace$. The \emph{fantastack} associated to $(\Sigma,\beta)$ is \[ 
    \sF_{\Sigma,\beta} := \left[X_{\widehat{\Sigma}} \q G_\beta \right], \quad \textrm{where } G_\beta := \Ker\left(\Gm^\kappa = T_{\ZZ^\kappa} \xrightarrow{T_\beta} T_N\right).
\]
In the above expression: \begin{enumerate}
    \item $X_{\widehat{\Sigma}}$ is the toric variety associated to the fan $\widehat{\Sigma}$ on $\ZZ^\kappa$;
    \item $T_N = \Hom_\GrpSch(N^\vee,\Gm)$ (resp.\ $T_{\ZZ^\kappa}$) is the torus of $N$ (resp.\ $\ZZ^\kappa$);
    \item $T_\beta$ is the homomorphism of tori induced by $\beta$;
    \item $G_\beta$ acts on $X_{\widehat{\Sigma}}$ as a subgroup of $\Gm^\kappa = T_{\ZZ^\kappa}$.
\end{enumerate}
\end{definition}

\begin{remark}\label{R:cox-construction}
By \cite[Section 5.1]{cox-toric-varieties}, we have \[
    X_{\widehat{\Sigma}} = \AA^\kappa \smallsetminus V(J_\Sigma) = \Spec(\kk[x_1,\dotsc,x_\kappa]) \smallsetminus V(J_\Sigma) = \bigcup_{\substack{\sigma \in \Sigma \\ \textrm{\tiny maximal cone}}}{U_\sigma}, 
\]
where $J_\Sigma = \bigl(x_\sigma := \prod_{\beta(\be_i) \notin \sigma}{x_i} \colon \sigma \in \Sigma$ maximal cone$\bigr)$ is the \emph{irrelevant ideal} and $U_\sigma := \Spec(\kk[x_1,\dotsc,x_\kappa][x_\sigma^{-1}])$ for each maximal cone $\sigma \in \Sigma$. Thus, $\sF_{\Sigma,\beta}$ admits a covering by principal open substacks $D_+(\sigma) := [U_\sigma \q G_\beta]$, as $\sigma$ varies over all maximal cones of $\sigma$. We call $D_+(\sigma)$ the \emph{$x_\sigma$-chart} of $\sF_{\Sigma,\beta}$ and sometimes denote it by $D_+(x_\sigma)$.
\end{remark}

Note that $\left(\widehat{\Sigma},\beta\right)$ is a \emph{stacky fan}, see \cite[Definition 2.4]{geraschenko-satriano-toric-stacks-I}, and $\sF_{\Sigma,\beta}$ is the toric stack associated to $\left(\widehat{\Sigma},\beta\right)$. However, more is true for fantastacks. 

\begin{remark}\label{R:fantastacks-properties}
By definition, $\widehat{\Sigma}$ is a smooth fan, and the torus of $\sF_{\Sigma,\beta}$ is $\Gm^\kappa/G_\beta$, which is isomorphic to $T_N$ via $T_\beta$. In other words, fantastacks are smooth toric Artin stacks with trivial generic stabilizer.
\end{remark}

\begin{remark}\label{R:good-moduli-space-fantastacks}
By definition, the morphism $\beta$ is compatible with the fans $\widehat{\Sigma}$ and $\Sigma$, and therefore induces a toric morphism $X_{\widehat{\Sigma}} \to X_\Sigma$, which descends to the good moduli space morphism $\sF_{\Sigma,\beta} \to X_\Sigma$; see \cite[Example 6.24]{geraschenko-satriano-toric-stacks-I}.
\end{remark}

A multi-weighted blow-up on an affine space is usually carried out along the vanishing locus of a monomial ideal. It is a fantastack $\sF_{\Sigma,\beta}$, where $(\Sigma,\beta)$ is associated to the monomial ideal and $\beta \colon \ZZ^\kappa \to N$ is in particular surjective, with kernel having a general simple description. Before giving the definition, we fix the following conventions for this paper. 

\begin{xpar}[Conventions]\label{X:conventions-2.1}
Let $N = \ZZ^n$, with standard basis vectors $\be_i$. Let $M = N^\vee$ be the dual lattice, with standard dual basis vectors $\be_i^\vee$. For any fan $\Sigma$ on $N$, let $\Sigma(1)$ denote the rays in $\Sigma$, and for a ray $\rho$ of $\Sigma$,  let $\bu_\rho = (\uu_{\rho,i})_{i=1}^n$ denote the first lattice point on $\rho$.

Let $\sigma_\std$ be the cone on $N = \ZZ^n$ spanned by the standard basis vectors $\be_i$, which generates the standard fan $\Sigma_\std$ in $N$, and whose associated affine toric variety is the affine space $\AA^n := \Spec(\kk[x_1,\dotsc,x_n])$. In addition, for a set $S$, we usually write $\AA^S$ for $\Spec(\kk[x_s \colon s \in S])$ or $\Spec(\kk[x_s' \colon s \in S])$.

In this paper, $\fa$ denotes a monomial ideal of $\kk[x_1,\dotsc,x_n]$. Associated to $\fa$ are the following notions:
\begin{enumerate}
    \item\label{conv1} the submonoid $\Upgamma_\fa = \lbrace \ba \in M \colon \pmb{x}^{\ba} := x_1^{a_1}\dotsm x_n^{a_n} \in \fa \rbrace$ of $M$,
    \item\label{conv2} the Newton polyhedron $P_\fa$ of $\fa$ given by the convex hull of $\Upgamma_\fa$ in $\RR^n_{\geq 0} \subset M \otimes_\ZZ \RR =: M_\RR$,
    \item\label{conv3} and the normal fan $\Sigma_\fa$ of $P_\fa$, which is a subdivision of $\Sigma_\std$ in $N$, and hence induces a toric, proper, birational morphism $X_{\Sigma_\fa} \to \AA^n$.
\end{enumerate}
For every integer $0 \leq k \leq n$, there is an inclusion-reversing correspondence between $k$-dimensional cones $\sigma$ of $\Sigma_\fa$ and $(n-k)$-dimensional faces $\tau$ of $P_\fa$. Therefore, we introduce the following notation: \begin{enumerate}[label=(\alph*),ref=\alph*]
     \item\label{not-a} $H_\rho$ is the facet of $P_\fa$ corresponding to a ray $\rho \in \Sigma_\fa(1)$. We also let $N_\rho(\fa)$ be the natural number such that the affine span of $H_\rho$ has  equation $\sum_{i=1}^n{\uu_{\rho,i} \cdot \be_i} {= N_{\rho}(\fa)}$.
     \item\label{not-b} $\bv_\sigma = (\uv_{\sigma,i})_{i=1}^n$ is the vertex of $P_\fa$ corresponding to a maximal cone $\sigma$ of $\Sigma_\fa$.
 \end{enumerate}
Finally, we identify the subset $\Sigma_{\std}(1) \subset \Sigma_\fa(1)$ with $[1,n] := \lbrace 1,\dotsc,n \rbrace$ and denote its complement by
\[
    \bE(\fa) := \Sigma_\fa(1) \smallsetminus \Sigma_{\std}(1) = \Sigma_\fa(1) \smallsetminus [1,n].
\]
We call the rays in $[1,n]$ the \emph{standard rays} and call the rays in $\bE(\fa)$ the \emph{exceptional rays}. We also set $\bE^+(\fa) := \lbrace \rho \in \Sigma_\fa(1) \colon N_\rho(\fa) > 0 \rbrace \supset \bE(\fa)$.
\end{xpar}

\begin{definition}[Multi-weighted blow-ups on affine spaces]\label{D:multi-weighted-blow-up}
Let $\bb = (\ub_\rho)_{\rho \in \bE(\fa)} \in \NN_{>0}^{\bE(\fa)}$. Then $\bb$ yields a surjective homomorphism $\beta \colon \ZZ^{\Sigma_\fa(1)} \to N = \ZZ^n$ as follows: \[
    \be_\rho \longmapsto \begin{cases}
        \bu_\rho = \be_i \quad &\textrm{if } \rho = i \in \Sigma_{\std}(1) = [1,n], \\
        \ub_\rho \cdot \bu_\rho \quad &\textrm{if } \rho \in \bE(\fa).
    \end{cases}
\]
Therefore, the data $(\Sigma_\fa,\beta)$ yields the fantastack $\Bl_{\fa,\bb}\AA^n :=\sF_{\Sigma_\fa,\beta}$, as described in Definition~\ref{D:fantastacks}.
Moreover, the homomorphism $\beta$ induces, as explained in Remark~\ref{R:fantastacks-properties} and Section~\ref{X:conventions-2.1}\eqref{conv3}, a toric morphism \[
    \pi_{\fa,\bb} \colon \Bl_{\fa,\bb}\AA^n \xrightarrow{\textrm{good moduli space}} X_{\Sigma_\fa} \lra \AA^n.
\]
We call this composition the \emph{multi-weighted blow-up} of $\AA^n$ along $\fa$ and $\bb$. When $\bb$ is the unit vector $(1,1,\dotsc,1)$ in $\NN_{>0}^{\bE(\fa)}$, we instead write the above expression as $\pi_\fa \colon \Bl_\fa\AA^n \to \AA^n$. This should not be confused with the \emph{schematic} blow-up of $\AA^n$ along $\fa$, which we denote by $\bl_\fa\AA^n$ in this paper. (In fact, the normalization of $\bl_\fa\AA^n$ is precisely $X_{\Sigma_\fa}$; see Remark~\ref{R:cms-equals-gms}.)
\end{definition}

One anticipates that $\Bl_\fa\AA^n$ is the ``most canonical'' multi-weighted blow-up of $\AA^n$ associated to $\fa$. We will make this canonicity precise in Section~\ref{3.1}, where, more generally, we also define, with the same goal of canonicity in mind, the multi-weighted blow-up $\Bl_{\fa_\bullet}\AA^n$ of $\AA^n$ along a monomial Rees algebra $\fa_\bullet$.

\begin{xpar}[Alternative notation]\label{X:conventions-2.1-b}
To uniformize forthcoming notation, we also set $\ub_i := 1$ for every $i \in [1,n] \subset \Sigma_\fa(1)$, and we interpret $\bb$ as a vector $(\ub_\rho)_{\rho \in \Sigma_\fa(1)}$ in $\NN_{>0}^{\Sigma_\fa(1)}$. Then $\beta$ in Definition~\ref{D:multi-weighted-blow-up} is simply $\be_\rho \mapsto \ub_\rho \cdot \bu_\rho$ for $\rho \in \Sigma_\fa(1)$.
\end{xpar}

\begin{remark}\label{R:multi-weighted-blow-up-properties}
The multi-weighted blow-up $\pi_{\fa,\bb}$ is birational, as explained in Remark~\ref{R:fantastacks-properties}. By \cite[Theorem~4.16]{alper-good-moduli-spaces}, the good moduli space morphism $\Bl_{\fa,\bb}\AA^n \to X_{\Sigma_\fa}$ is universally closed and surjective. Therefore, so is $\pi_{\fa,\bb}$.
\end{remark}

\begin{remark}\label{R:newton-polyhedron-well-defined}
It is important to highlight that we have fixed coordinates $x_1,\dotsc,x_n$ on $\AA^n$ in which $\fa$ is monomial; \textit{i.e.} we fixed a toroidal
Zariski logarithmic structure, with chart $\NN^n \xrightarrow{\be_i \mapsto x_i} \kk[x_1,\dotsc,x_n]$, in which $\fa$ is a monomial ideal in the general sense of Section~\ref{X:monomiality} below. 

Then the Newton polyhedron $P_\fa$ is \emph{well defined} under the lens of logarithmic geometry; \textit{i.e.}~$P_\fa$ is \emph{independent}, up to symmetry, of any change in local coordinates at $0 \in \AA^n$ which respects the aforementioned toroidal logarithmic structure. Explicitly, any such coordinate change corresponds to a monoid automorphism of $\NN^n \oplus \sO_{\AA^n,0}^\ast$, which must map each $(\be_i,1)$ to $(\be_j,\mu)$ for some $j \in [1,n]$ and $\mu \in \sO_{\AA^n,0}^\ast$.
\end{remark}

\begin{remark}\label{R:same-newton-polyhedron-for-two-monomial-ideals}
  Two monomial ideals can possess the same normal fan, and thus yield the same multi-weighted blow-up. For example:
\begin{enumerate}
    \item\label{R:snpftmi-1} $P_\fa = P_{\overline{\fa}}$, where $\overline{\fa} = \lbrace \pmb{x}^{\ba} \colon \ba \in P_\fa \rbrace$ is the integral closure of $\fa$.
    \item\label{R:snpftmi-2} If $f_1,\dotsc,f_r$ are monomials generating $\fa$, then for any $\ell \in \NN_{>0}$, the integral closures of $\fa^\ell$ and $(f_1^\ell,\dotsc,f_r^\ell)$ coincide, so by \eqref{R:snpftmi-1}, $\fa^\ell$ and $(f_1^\ell,\dotsc,f_r^\ell)$ have the same Newton polyhedra.
    \item\label{R:snpftmi-3} {For any $\ba \in \NN^n$, $\Sigma_{\pmb{x}^\ba \cdot \fa} = \Sigma_\fa$.} However, while the multi-weighted blow-up along $\pmb{x}^\ba \cdot \fa$ is the same as that along $\fa$, there is a subtle difference in their ``exceptional'' divisors (see Remark~\ref{R:total-transform-of-monomial-ideal}). 
    \item\label{R:snpftmi-4} Lastly, as $\ell$ varies, although the Newton polyhedron of $\fa^\ell$ varies, the normal fan of $\fa^\ell$ remains the same, and so does the multi-weighted blow-up of $\fa^\ell$. 
\end{enumerate}
\end{remark}

\begin{xpar}[Explicating multi-weighted blow-ups]\label{X:explicating-multi-weighted-blow-ups}
  With the notation  in Definition~\ref{D:multi-weighted-blow-up}, the homomorphism $\beta \colon \ZZ^{\Sigma_\fa(1)} \to N = \ZZ^n$ induced by $\bb$ fits into the short exact sequence
  \[
    0 \lra \ZZ^{\bE(\fa)} \xrightarrow{\alpha \ = \ \tiny \begin{bmatrix}\bB\\-I_k\end{bmatrix}} \ZZ^{\Sigma_\fa(1)} \xrightarrow{\footnotesize \beta \ = \ \begin{bmatrix}I_k&\bB\end{bmatrix}} \ZZ^n \lra 0, 
\]
where $k := \#\bE(\fa)$ and $\bB = (\uB_{i,\rho})_{1 \leq i \leq n,\; \rho \in \bE(\fa)}$ is the matrix whose columns are $\ub_\rho \cdot \bu_\rho$ for each $\rho \in \bE(\fa)$. Unravelling the definitions, let us highlight three details:
\begin{enumerate}
    \item\label{X:emwbu-1} We have the following commutative diagram: \[
    \begin{tikzcd}
    X_{\widehat{\Sigma}_\fa}= \AA^{\Sigma_\fa(1)} \smallsetminus V(J_{\Sigma_\fa}) \arrow[to=1-2, hookrightarrow, "\textrm{\tiny{open}}"] \arrow[to=2-1, twoheadrightarrow, swap, "\textrm{\tiny{stack-theoretic quotient}}"] & \AA^{\Sigma_\fa(1)}  \arrow[to=1-3] & \AA^n\rlap{,} \\
    \Bl_{\fa,\bb}\AA^n = \left[X_{\widehat{\Sigma}_\fa} \q \Gm^{\bE(\fa)}\right] \arrow[to=1-3, swap, bend right=15, "\pi_{\fa,\bb}"] & &
    \end{tikzcd}
\]
where we follow the convention in Section~\ref{X:conventions-2.1} that \begin{equation}\label{Eq:coordinates}
    \qquad \qquad \AA^{\Sigma_\fa(1)} = \Spec\left(\kk[x_1',\dotsc,x_n']\left[x_\rho' \colon \rho \in \bE(\fa)\right]\right)
\end{equation}
so that the above morphism $\AA^{\Sigma_\fa(1)} \to \AA^n$ of affine spaces corresponds to the homomorphism $\kk[x_1,\dotsc,x_n] \to \kk[x_1',\dotsc,x_n'][x_\rho' \colon \rho \in \bE(\fa)]$ defined by
\[
    \qquad \qquad x_i \longmapsto \left(x_i' \cdot \prod_{\rho \in \bE(\fa)}{\left(x_\rho'\right)^{\uB_{i,\rho}}}\right) = \left(x_i' \cdot \prod_{\rho \in \bE(\fa)}{\left(x_\rho'\right)^{\ub_\rho \cdot \uu_{\rho,i}}}\right)
\]
for every $1 \leq i \leq n$. For $\rho \in \bE(\fa)$, the corresponding coordinate $x_\rho'$ of $\AA^{\bE(\fa)}$ will be written as $u_\rho$ during examples (\textit{e.g.}~Sections~\ref{2.2} and~\ref{5.1}).
%%%%%%%%%%%%%%%%%%%%%%%%%%%%%%%%%%%%%%%%%%%%%%%%%
\item\label{X:emwbu-2} The action of $\Gm^{\bE(\fa)}$ on $X_{\widehat{\Sigma}_\fa} \subset \AA^{\Sigma_\fa(1)} = \Spec(\kk[x_1',\dotsc,x_n'][x_\rho' \colon \rho \in \bE(\fa)])$ can be interpreted from the matrix $\alpha = \small{\begin{bmatrix}\bB \\ -I_k\end{bmatrix}}$ as follows: \begin{enumerate}
    \item\label{X:emwbu-2a} For every $1 \leq i \leq n$, $x_i'$ has $\ZZ^{\bE(\fa)}$-grading $(\uB_{i,\rho})_{\rho \in \bE(\fa)} = (\ub_\rho \cdot \uu_{\rho,i})_{\rho \in \bE(\fa)}$.
    \item\label{X:emwbu-2b} For every $\rho \in \bE(\fa)$, $x_\rho'$ has $\ZZ^{\bE(\fa)}$-grading $-\be_{\rho} = (-\delta_{\rho,\widetilde{\rho}})_{\widetilde{\rho} \in \bE(\fa)}$. 
\end{enumerate}
%%%%%%%%%%%%%%%%%%%%%%%%%%%%%%%%%%%%%%%%%%%%%%%%%
\item\label{X:emwbu-3} By Definition~\ref{D:fantastacks}\eqref{D:f-1}, $\Bl_{\fa,\bb}\AA^n$ admits an open cover by \emph{$x_\sigma'$-charts} $D_+(\sigma) = D_+(x_\sigma') := \left[U_\sigma \q \Gm^{\bE(\fa)}\right]$, where $\sigma$ varies over all maximal cones $\sigma$ of $\Sigma_\fa$. In this setting, we have \[
    \qquad \qquad U_\sigma = \Spec\left(\kk\left[x_1',\dotsc,x_n'\right]\left[x_\rho' \colon \rho \in \bE(\fa)\right]\left[\left(x_\sigma'\right)^{-1}\right]\right), 
\]
where \[
    \qquad \qquad x_\sigma' = \prod_{\substack{\rho \in \Sigma_\fa(1) \\ \rho \not\subset \sigma}}{x_\rho'} = \prod_{\substack{\rho \in \Sigma_\fa(1) \\ \bv_\sigma \notin H_\rho}}{x_\rho'} = \prod_{\substack{\rho \in \Sigma_\fa(1) \\ \bu_\rho \cdot \bv_\sigma > N_\rho(\fa)}}{x_\rho'}
\]
(see Section~\ref{X:conventions-2.1}\eqref{not-a},~\eqref{not-b}).
\end{enumerate}
\end{xpar}

\begin{remark}\label{R:generalized-multi-weighted-blow-up}
Slightly more generally, one can also consider \emph{generalized multi-weighted blow-ups} $\Bl_{\fa,\widetilde{\bb}}\AA^n$ of $\AA^n$ along $\fa$ and $\widetilde{\bb} = (\ub_\rho)_{\rho \in \Sigma_\fa(1)} \in \NN_{>0}^{\Sigma_\fa(1)}$. Such a $\widetilde{\bb}$ yields a homomorphism $\widetilde{\beta} \colon \ZZ^{\Sigma_\fa(1)} \to N = \ZZ^n$ with finite cokernel, which sends $\be_\rho$ to $\ub_\rho \cdot \bu_\rho$ for every $\rho \in \Sigma_\fa(1)$. The fantastack $\sF_{\Sigma_\fa,\widetilde{\beta}}$ associated to $(\Sigma_\fa,\widetilde{\beta})$ is then denoted by $\Bl_{\fa,\widetilde{\bb}}\AA^n$. Likewise, there is a toric morphism $\pi_{\fa,\widetilde{\bb}} \colon \Bl_{\fa,\widetilde{\bb}}\AA^n \to X_\Sigma \to \AA^n$, which is the generalized multi-weighted blow-up morphism.

In the same way as in Section~\ref{X:explicating-multi-weighted-blow-ups}, one can partially explicate generalized multi-weighted blow-ups: we have \[
    \Bl_{\fa,\widetilde{\bb}}\AA^n = \left[\left(\AA^{\Sigma_\fa(1)} \smallsetminus V\left(J_{\Sigma_\fa}\right)\right) \q D\left(\Coker\left(\widetilde{\beta}^\vee\right)\right)\right] \xrightarrow{\pi_{\fa,\widetilde{\bb}}} \AA^n, 
\]
where $\pi_{\fa,\widetilde{\bb}}$ is induced by $\kk[x_1,\dotsc,x_n] \to \kk[x_1',\dotsc,x_n']\left[x_\rho' \colon \rho \in \bE(\fa)\right]$,  which is defined by 
\[
    x_i \longmapsto \left(\left(x_i'\right)^{\ub_i} \cdot \prod_{\rho \in \bE(\fa)}{\left(x_\rho'\right)^{\ub_\rho \cdot \uu_{\rho,i}}}\right)
\]
and where $D(\Coker(\widetilde{\beta}^\vee))$ acts via the morphism of diagonalizable groups obtained from $(\ZZ^{\Sigma_\fa(1)})^\vee \surj \Coker(\widetilde{\beta}^\vee)$ by applying $D(-) := \Hom_{\GrpSch}(-,\Gm)$. This demonstrates that generalized multi-weighted blow-ups can be seen as an amalgamation of the notion of multi-weighted blow-ups in Definition~\ref{D:multi-weighted-blow-up} and the notion of \emph{root stacks} {(see \cite[Definition 2.2.1]{Cadman-root-stacks} or \cite[Appendix B]{AGV}} or \cite[Example 2.2.4]{quek-rydh-weighted-blow-up}.)

Finally, $\Bl_{\fa,\widetilde{\bb}}\AA^n$ likewise admits an open cover by $x_\sigma'$-charts $D_+(\sigma) = \left[U_\sigma \q D\left(\Coker\left(\widetilde{\beta}^\vee\right)\right)\right]$, where $U_\sigma$ has the same description in Section~\ref{X:explicating-multi-weighted-blow-ups}.
\end{remark}

\subsection{Two examples}\label{2.2}

We illustrate 
Section~\ref{X:explicating-multi-weighted-blow-ups} via two examples. The first is classical and is a special case of multi-weighted blow-ups. 

\begin{example}[Weighted blow-ups]\label{EX:weighted-blow-ups}
For $w_1,\dotsc,w_n \in \NN_{>0}$, set $\ell := \lcm(w_1,\dotsc,w_n)$. The (stack-theoretic) weighted blow-up of $\AA^n = \Spec(\kk[x_1,$ $\dotsc,x_n])$ along the center $(x_1^{1/w_1},\dotsc,x_n^{1/w_n})$, in the sense of \cite{abramovich-temkin-wlodarczyk-weighted-resolution}, is simply the multi-weighted blow-up $\Bl_{\fa,\bb}\AA^n$, where $\fa = (x_1^{\ell/w_1},\dotsc,x_n^{\ell/w_n})$ and $\bb = \gcd(w_1,\dotsc,w_n) \in \NN_{>0} = \NN_{>0}^{\bE(\fa)}$. We demonstrate this via an explicit example, which supplies no less information than the general case.

Consider $\fa = (x^2,y^3,z^3) \subset \kk[x,y,z]$ and $\bb = \gcd(2,3,3) = 1$. We draw the Newton polyhedron $P_\fa$, which has four facets:
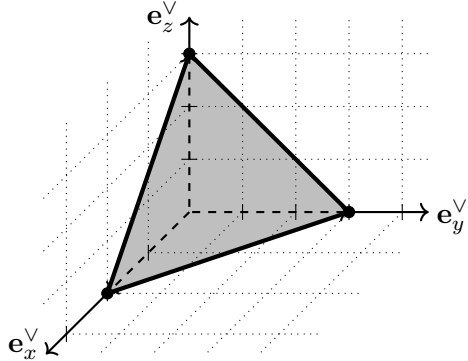
\[
	\begin{tikzpicture}[scale=0.7]
	
	%%% shade %%%
	\filldraw[lightgray] (0,0,4)--(3,0,0)--(0,3,0)--(0,0,4);
	
	%%% axes %%%	
	\draw[->, thick] (3,0,0)--(4.5,0,0); %%% y-axes
	\draw[->, thick] (0,3,0)--(0,3.7,0); %%% z-axis
	\draw[->, thick] (0,0,4)--(0,0,7); % x-axis
	\draw[->, dashed, thick] (0,0,0)--(0,0,4); % dashed x-axis
	\draw[<->, dashed, thick] (0,3,0)--(0,0,0)--(3,0,0); %%% dashed z,y-axes
	
	%%% arrows/variables %%%
	\node at (4.95,-0.1,0) {$\be_y^\vee$};
	\node at (-0.5,3.7,0) {$\be_z^\vee$};
	\node at (-0.6,0,6.5) {$\be_x^\vee$};
	
	%%% plots on axes %%%
	\foreach \x in {1,...,4} % y-axis
	\draw (\x,0.15,0) -- (\x,-0.15,0);
	
	\foreach \x in {1,...,3} % z-axis
	\draw (0.15,\x,0) -- (-0.15,\x,0);
		
	\foreach \x in {2,4,6} % x-axis
	\draw (0,0.15,\x) -- (0,-0.15,\x);
	
	%%% plot lines %%%
	
	% y-z plane
	\foreach \x in {1,...,4}
	\draw[dotted] (\x,0,0) -- (\x,3.7,0);
	\foreach \x in {1,...,3}
	\draw[dotted] (0,\x,0) -- (4.5,\x,0);
	
	% x-y plane
	\foreach \x in {1,...,4}
	\draw[dotted] (\x,0,0) -- (\x,0,7);
	\foreach \x in {2,4,6}
	\draw[dotted] (0,0,\x) -- (4.8,0,\x);
	
	% x-z plane
	\foreach \x in {2,4,6}
	\draw[dotted] (0,0,\x) -- (0,4,\x);
	\foreach \x in {1,...,3}
	\draw[dotted] (0,\x,0) -- (0,\x,7.3);
	
	%%% lines %%%
	\draw[ultra thick] (0,0,4)--(3,0,0)--(0,3,0)--(0,0,4);
	
	%%% large dots %%%
	\filldraw[black] (0,0,4) circle (3pt);
	\filldraw[black] (3,0,0) circle (3pt);
	\filldraw[black] (0,3,0) circle (3pt);
	
	\end{tikzpicture} 
\]
Taking a cross-section of the normal fan $\Sigma_\fa$, we obtain 
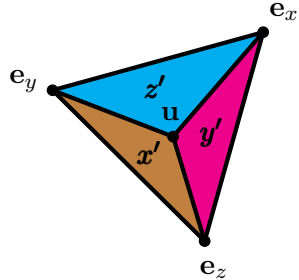
\[
    \begin{tikzpicture}
    
    %%% shade %%%
    \filldraw[magenta] (1.2,1,1)--(2,2,0)--(2,0,2)--(1.2,1,1);
    \filldraw[cyan] (1.2,1,1)--(2,2,0)--(0,2,2)--(1.2,1,1);
    \filldraw[brown] (1.2,1,1)--(2,0,2)--(0,2,2)--(1.2,1,1);
    
    %%% lines %%%
    \draw[ultra thick] (0,2,2)--(2,0,2)--(2,2,0)--(0,2,2);
    \draw[ultra thick] (2,2,0)--(1.2,1,1)--(0,2,2);
    \draw[ultra thick] (2,0,2)--(1.2,1,1);
    
    %%% large dots %%%
    \filldraw[black] (2,2,0) circle (2pt);
    \node at (2.2,2.2,-0.2) {$\be_x$};
    
    \filldraw[black] (2,0,2) circle (2pt);
    \node at (2.2,-0.3,2.2) {$\be_z$};
    
    \filldraw[black] (0,2,2) circle (2pt);
    \node at (-0.3,2.2,2.2) {$\be_y$};
    
    \filldraw[black] (1.2,1,1) circle (2pt);
    \node at (1.2,1.35,1.1) {$\bu$};
    
    %%% nodes %%%
    \node at (1.2,1.9,1.6) {$\pmb{z'}$};
    \node at (1.2,1.1,1.8) {$\pmb{x'}$};
    \node at (1.75,1.1,1.1) {$\pmb{y'}$};
    
    \end{tikzpicture}
\]
where $\bu = (3,2,2)$ is the normal vector to the shaded facet of $P_\fa$ above. The vertices $(2,0,0)$, $(0,3,0)$, and $(0,0,3)$ of $P_\fa$ correspond, respectively, to the maximal cones of $\Sigma_\fa$ represented above by brown-, magenta-, and cyan-coloured triangles, which also correspond, respectively, to the $x'$-, $y'$-, and $z'$-charts on $\Bl_\fa\AA^3 = [X_{\widehat{\Sigma}_\fa} \q \Gm]$, where $X_{\widehat{\Sigma}_\fa} = \AA^4 \smallsetminus V(J_{\Sigma_\fa}) = \AA^4 \smallsetminus V(x',y',z')$. As explained in Section~\ref{X:explicating-multi-weighted-blow-ups}, the morphism $\pi_\fa \colon \Bl_\fa\AA^3 \to \AA^3$ can be read off the following matrix: \[
    \begin{bmatrix}I_3&\bu\end{bmatrix} = \begin{bmatrix}[ccc|c]
    1&0&0&3\\
    0&1&0&2\\
    0&0&1&2
    \end{bmatrix}\qquad \rightsquigarrow \qquad 
    \begin{cases}
        x = x'u^3, \\
        y = y'u^2, \\
        z = z'u^2,
    \end{cases}
\]
and the $\Gm$-action on $X_{\widehat{\Sigma}_\fa}$ can be read off from the matrix \[
    \begin{bmatrix}\bu \\ -1\end{bmatrix} = \begin{bmatrix}
    3\\2\\2\\ -1
    \end{bmatrix} \qquad \rightsquigarrow \qquad \begin{cases}
        x' \textrm{ has $\ZZ$-weight $3$,} \\
        y' \textrm{ has $\ZZ$-weight $2$,} \\
        z' \textrm{ has $\ZZ$-weight $2$,} \\
        u \textrm{ has $\ZZ$-weight $-1$.}
    \end{cases}
\]
This is precisely the description of the (stack-theoretic) weighted blow-up of $\AA^3$ along $(x^{1/3},y^{1/2},z^{1/2})$ in \cite{abramovich-temkin-wlodarczyk-weighted-resolution}.
\end{example}

\begin{example}[A new example]\label{EX:second-example}
The Newton polyhedron $P_\fa$ of the monomial ideal $\fa = (x^2,y^2z,z^3) \subset \kk[x,y,z]$ has five facets: 
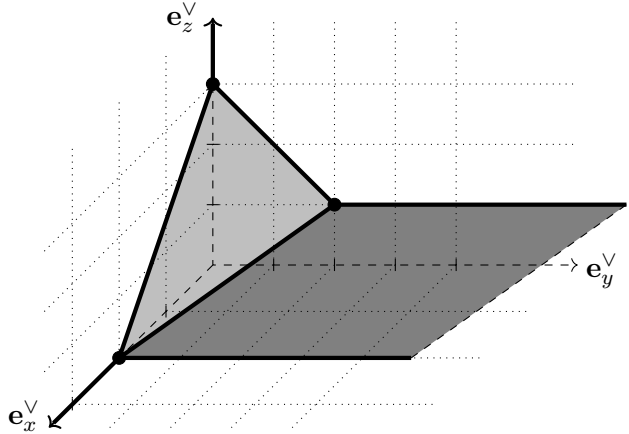
\[
	\begin{tikzpicture}[scale=0.8]
	
	%%% shade %%%
	\filldraw[gray] (4.8,0,4)--(0,0,4)--(2,1,0)--(6.8,1,0);
	\filldraw[lightgray] (0,3,0)--(0,0,4)--(2,1,0)--(0,3,0);
	
	%%% axes %%%	
	%\draw[->, thick] (3,0,0)--(4.5,0,0); %%% y-axes
	\draw[->, thick] (0,3,0)--(0,4.1,0); %%% z-axis
	\draw[->, thick] (0,0,4)--(0,0,7); % x-axis
	\draw[->, dashed] (0,0,0)--(0,0,4); % dashed x-axis
	\draw[<->, dashed] (0,3,0)--(0,0,0)--(6,0,0); %%% dashed z,y-axes
	
	%%% arrows/variables %%%
	\node at (6.4,-0.1,0) {$\be_y^\vee$};
	\node at (-0.5,4.1,0) {$\be_z^\vee$};
	\node at (-0.6,0,6.5) {$\be_x^\vee$};
	
	%%% plots on axes %%%
	\foreach \x in {1,...,4} % y-axis
	\draw (\x,0.1,0) -- (\x,-0.1,0);
	
	\foreach \x in {1,...,3} % z-axis
	\draw (0.1,\x,0) -- (-0.1,\x,0);
		
	\foreach \x in {2,4,6} % x-axis
	\draw (0,0.1,\x) -- (0,-0.1,\x);
	
	%%% plot lines %%%
	
	% y-z plane
	\foreach \x in {1,...,4}
	\draw[dotted] (\x,0,0) -- (\x,4.1,0);
	\foreach \x in {1,...,3}
	\draw[dotted] (0,\x,0) -- (6,\x,0);
	
	% x-y plane
	\foreach \x in {1,...,4}
	\draw[dotted] (\x,0,0) -- (\x,0,7);
	\foreach \x in {2,4,6}
	\draw[dotted] (0,0,\x) -- (6,0,\x);
	
	% x-z plane
	\foreach \x in {2,4,6}
	\draw[dotted] (0,0,\x) -- (0,4.3,\x);
	\foreach \x in {1,...,3}
	\draw[dotted] (0,\x,0) -- (0,\x,7.3);
	
	%%% lines %%%
	\draw[ultra thick] (0,3,0)--(0,0,4)--(2,1,0)--(0,3,0)--(0,4,0);
	\draw[ultra thick] (2,1,0)--(6.8,1,0);
	\draw[ultra thick] (0,0,4)--(4.8,0,4);
	\draw[ultra thick] (0,0,4)--(0,0,6.8);
	\draw[dashed] (4.8,0,4)--(6.8,1,0);
	
	%%% large dots %%%
	\filldraw[black] (0,0,4) circle (3pt);
	\filldraw[black] (2,1,0) circle (3pt);
	\filldraw[black] (0,3,0) circle (3pt);
	
	\end{tikzpicture} 
\]
We sketch a cross-section of the normal fan $\Sigma_\fa$: 
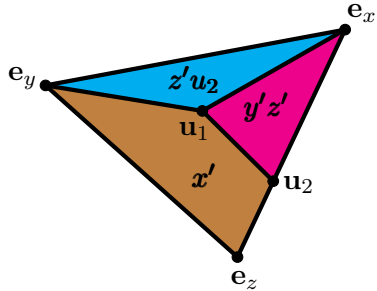
\[
    \begin{tikzpicture}[scale=0.8]
    
    %%% shade %%%
    \filldraw[magenta] (2,3,3.5)--(2.33,1,1.33)--(3,3,0);
    \filldraw[brown] (2.33,1,1.33)--(2,3,3.5)--(0,4,5)--(2,0,2);
    \filldraw[cyan] (3,3,0)--(2,3,3.5)--(0,4,5);
    
    %%% lines %%%
    \draw[ultra thick] (0,4,5)--(2,0,2)--(3,3,0)--(0,4,5);
    \draw[ultra thick] (2,3,3.5)--(2.33,1,1.33);
    \draw[ultra thick] (3,3,0)--(2,3,3.5)--(0,4,5);
    
    %%% large dots %%%
    \filldraw[black] (3,3,0) circle (2.5pt);
    \node at (3.2,3.2,-0.2) {$\be_x$};
    
    \filldraw[black] (2,0,2) circle (2.5pt);
    \node at (2.2,-0.3,2.2) {$\be_z$};
    
    \filldraw[black] (0,4,5) circle (2.5pt);
    \node at (-0.3,4.2,5.2) {$\be_y$};
    
    \filldraw[black] (2,3,3.5) circle (2.5pt);
    \node at (2.6,3.4,5.4) {$\bu_1$};
    
    \filldraw[black] (2.33,1,1.33) circle (2.5pt);
    \node at (3.1,1.3,2.25) {$\bu_2$};
    
    %%% nodes %%%
    \node at (2.6,2.6,2.3) {$\pmb{y'z'}$};
    \node at (3.5,3.4,7.3) {$\pmb{x'}$};
    \node at (2.6,4.25,5.4) {$\pmb{z'u_2}$};
    
    \end{tikzpicture}
\]
where $\bu_1 = (3,2,2)$ and $\bu_2 = (1,0,2)$ are the normal vectors to the shaded facets of $P_\fa$ above. The vertices $(2,0,0)$, $(0,2,1)$, and $(0,0,3)$ of $P_\fa$ correspond to the maximal cones of $\Sigma_\fa$ represented above by the brown-, magenta-, and cyan-coloured regions, and these also correspond to the $x'$-, $y'z'$-, and $z'u_1$-charts on $\Bl_\fa\AA^3 = [X_{\widehat{\Sigma}_\fa} \q \Gm^2]$, respectively, where $X_{\widehat{\Sigma}_\fa} = \AA^5 \smallsetminus V(J_{\Sigma_\fa}) = \AA^5\smallsetminus V(x',y'z',z'u_2)$. The morphism $\pi_\fa \colon \Bl_\fa\AA^3 \to \AA^3$ can be determined from the following matrix: \[
    \begin{bmatrix}
    I_3 & \bu_1 & \bu_2
    \end{bmatrix} = \begin{bmatrix}[ccc|cc]
    1&0&0&3&1\\
    0&1&0&2&0\\
    0&0&1&2&2
    \end{bmatrix} \qquad \rightsquigarrow \qquad \begin{cases}
        x = x'u_1^3u_2, \\
        y = y'u_1^2, \\
        z = z'u_1^2u_2^2,
    \end{cases}
\]
and the $\Gm^2$-action on $X_{\widehat{\Sigma}_\fa}$ can be determined from the matrix \[
    \begin{bmatrix}
    \bu_1 & \bu_2 \\
    -1 & 0 \\
    0 & -1 
    \end{bmatrix} = \begin{bmatrix}
    3&1\\
    2&0\\
    2&2\\
    -1&0\\
    0&-1
    \end{bmatrix} \qquad \rightsquigarrow \qquad \begin{cases}
        x' \textrm{ has $\ZZ^2$-weight $(3,1)$,} \\
        y' \textrm{ has $\ZZ^2$-weight $(2,0)$,} \\
        z' \textrm{ has $\ZZ^2$-weight $(2,2)$,} \\
        u_1 \textrm{ has $\ZZ^2$-weight $(-1,0)$,} \\
        u_2 \textrm{ has $\ZZ^2$-weight $(0,-1)$.}
    \end{cases}
\]
\end{example}

\subsection{Exceptional divisors and transforms}\label{2.3}

The presentation $\Bl_{\fa,\bb}\AA^n = \left[X_{\widehat{\Sigma}_\fa} \q \Gm^{\bE(\fa)}\right]$ induces an isomorphism
\[
    \Pic\left(\Bl_{\fa,\bb}\AA^n\right) \xrightarrow{\;\lowsimeq\;} \Pic^{\Gm^{\bE(\fa)}}\left(X_{\widehat{\Sigma}_\fa}\right), 
\]
where the right-hand side denotes the $\Gm^{\bE(\fa)}$-equivariant Picard group of $X_{\widehat{\Sigma}_\fa}$. In particular, for each $\bd \in \ZZ^{\bE(\fa)}$, there are \emph{tautological line bundles} $\sO(\bd) := \sO_{\Bl_{\fa,\bb}\AA^n}(\bd)$ on $\Bl_{\fa,\bb}\AA^n$ which correspond to the trivial line bundle $\sO_{X_{\widehat{\Sigma}_\fa}}$ on $X_{\widehat{\Sigma}_\fa}$ endowed with the $\Gm^{\bE(\fa)}$-linearization given by the ``$\bd$-shift''; \textit{i.e.} \[
     \pi_{\fa,\bb}^\ast\sO(\bd) = \sO_{X_{\widehat{\Sigma}_\fa}}(\bd) := \left(\widetilde{\kk\left[x_1',\dotsc,x_n'\right]\left[x_\rho' \colon \rho \in \bE(\fa)\right]}\right)(\bd)\ \Big|_{X_{\widehat{\Sigma}_\fa}}, 
\]
where the $\ZZ^{\bE(\fa)}$-grading of $x_i'$ and $x_\rho'$ on the right-hand side can be obtained by subtracting $\bd$ from the respective $\ZZ^{\bE(\fa)}$-gradings in Section~\ref{X:explicating-multi-weighted-blow-ups}\eqref{X:emwbu-2}.

\begin{xpar}[Exceptional divisors on a multi-weighted blow-up]\label{X:exceptional-divisors}
For every $\rho \in \bE(\fa)$, recall that $x_\rho'$ has $\ZZ^{\bE(\fa)}$-weight $-\be_\rho$ (see Section~\ref{X:explicating-multi-weighted-blow-ups}\eqref{X:emwbu-2b}), and hence there is an injection \[
    \sO(\be_\rho) \longinj \sO(\mathbf{0}) = \sO
\]
induced by multiplication by $x_\rho'$, which embeds $\sO(\be_\rho)$ as an ideal sheaf of $\sO$ cutting out the \emph{exceptional divisor} \[
    E_\rho := V\left(x_\rho'\right) \subset \Bl_{\fa,\bb}\AA^n.
\]
This justifies calling the rays $\rho \in \bE(\fa)$  exceptional rays in $\Sigma_\fa$ (see Section~\ref{X:conventions-2.1}).
\end{xpar}

\begin{remark}\label{R:root-stacks}
For $i=1,2$, let $\bb_i = (\ub_{i,\rho})_{\rho \in \bE(\fa)} \in \NN_{>0}^{\bE(\fa)}$ be such that for each $\rho \in \bE(\fa)$, \[
    \ub_{2,\rho} = c_\rho \cdot \ub_{1,\rho} \quad \textrm{for some $c_\rho \in \NN_{>0}$.}
\]
By the description in Section~\ref{X:explicating-multi-weighted-blow-ups}, observe that $\Bl_{\fa,\bb_2}\AA^n$ can be obtained from $\Bl_{\fa,\bb_1}\AA^n$ by iteratively taking $c_\rho$\textsuperscript{th} root stacks along each exceptional divisor $E_\rho = V(x_\rho')$ of $\Bl_{\fa,\bb_1}\AA^n$.
\end{remark}

The remainder of this section is devoted to the definition of transforms of ideals under a multi-weighted blow-up. We first make the following definition, which is part of the subsequent proposition.

\begin{definition}\label{D:newton-polyhedron-of-ideal}
The \emph{monomial saturation} of an ideal $I \subset \kk[x_1,\dotsc,x_n]$ is the monomial ideal $\fa_I$ generated by monomials appearing (with non-zero coefficient) in the elements of $I$. We then define the Newton polyhedron $P_I$ of $I$ to simply be the Newton polyhedron of $\fa_I$. 

For $\bu = (\uu_i)_{i=1}^n \in \NN^n$ and $m \in \NN$, we say that $P_I$ is \emph{bounded below} by the hyperplane $\sum_{i=1}^n{\uu_i \cdot \be_i} = m$ if for every $g = \sum_{\ba \in \NN^n}{c_\ba \cdot \pmb{x}^{\ba}} \in I$, we have $\bu \cdot \ba \geq m$ whenever $c_\ba \neq 0$.
\end{definition}

\begin{proposition}\label{P:total-transform-of-monomial-ideal}
  Set $\sO := \sO_{\Bl_{\fa,\bb}\AA^n}$. With the notation of Equation \eqref{Eq:coordinates}, we have the following statements: 
  \begin{enumerate}
    \item\label{P:ttomi-1} Let $I$ be an ideal of\, $\kk[x_1,\dotsc,x_n]$. For $(n_\rho)_{\rho \in \Sigma_\fa(1)} \in \NN^{\Sigma_\fa(1)}$, the ideal sheaf underlying the total transform $V(I) \times_{\AA^n,\pi_{\fa,\bb}} \Bl_{\fa,\bb}\AA^n$ satisfies the inclusion \[
        \qquad \qquad \pi_{\fa,\bb}^{-1}I \cdot \sO \subset \prod_{\rho \in \Sigma_\fa(1)}{\left(x_\rho'\right)^{\ub_\rho \cdot n_\rho}}
    \]
    if and only if for every $\rho \in \Sigma_\fa(1)$, the Newton polyhedron $P_I$ of\, $I$ is bounded below by the hyperplane $\sum_{i=1}^n{\uu_{\rho,i} \cdot \be_i} = n_\rho$.
    
    \item\label{P:ttomi-2} Moreover, the ideal sheaf underlying the total transform $V(\fa) \times_{\AA^n,\pi_{\fa,\bb}} \Bl_{\fa,\bb}\AA^n \subset \Bl_{\fa,\bb}\AA^n$ satisfies the equality \begin{equation}\label{EX:total-transform-of-monomial-ideal}
        \qquad \qquad \pi_{\fa,\bb}^{-1}\fa \cdot \sO = \prod_{\rho \in \bE^+(\fa)}{\left(x_\rho'\right)^{\ub_\rho \cdot N_\rho(\fa)}}; 
    \end{equation}
    see Section~\ref{X:conventions-2.1} for the definitions of\, $\bE^+(\fa)$, $N_i(\fa)$, and $N_\rho(\fa)$.
\end{enumerate}
\end{proposition}

Before giving the proof, let us illustrate it by revisiting Example~\ref{EX:second-example}. 

\begin{example}\label{EX:total-transform}
Let $\fa = (x^2,y^2z,z^3)$, and consider the multi-weighted blow-up $\pi_\fa \colon \Bl_\fa\AA^3 \to \AA^3$. Using the equations in Example~\ref{EX:second-example}, one computes the following: \begin{enumerate}
    \item Let $I = (x^2+y^2+z^2) \subset \kk[x,y,z]$. Its total transform $\pi_\fa^{-1}I \cdot \sO_{\Bl_\fa\AA^3}$ is $(x'^2u_1^6u_2^2+y'^2u_1^4+z'^2u_1^4u_2^4)$, which is contained in $(u_1^4)$ but is contained neither in $(u_1^i)$ for $i \geq 5$, nor in $(u_2^j)$ for $j \geq 1$. This agrees with Proposition~\ref{P:total-transform-of-monomial-ideal}\eqref{P:ttomi-1}: the Newton polyhedron of $I$ has vertices $\bv_1 = (2,0,0)$, $\bv_2 = (0,2,0)$, and $\bv_3 = (0,0,2)$, and we have $\min_{1 \leq i \leq 3}{\bu_1 \cdot \bv_i} = 4$ and $\min_{1 \leq i \leq 3}{\bu_2 \cdot \bv_i} = 0$.
    \item The total transform $\pi_\fa^{-1}\fa \cdot \sO_{\Bl_\fa\AA^3}$ of $\fa$ equals \[
        \qquad \qquad (x'^2u_1^6u_2^2,y'^2z'u_1^6u_2^2,z'^3u_1^6u_2^6) = (u_1^6u_2^2) \cdot (x'^2,y'^2z',z'^3).
    \]
    The ideal $(x'^2,y'^2z',z'^3)$ is the unit ideal on each chart of $\Bl_\fa\AA^3$, so in fact, $\pi_\fa^{-1}\fa \cdot \sO_{\Bl_\fa\AA^3} = (u_1^6u_2^2)$. This agrees with Proposition~\ref{P:total-transform-of-monomial-ideal}\eqref{P:ttomi-2}: note that $N_{\be_i}(\fa) = 0$ for $1 \leq i \leq 3$, $N_{\bu_1}(\fa) = 6$, and $N_{\bu_2}(\fa) = 2$.
\end{enumerate}
\end{example}

\begin{proof}[Proof of Proposition~\ref{P:total-transform-of-monomial-ideal}]
It suffices to compute on the smooth cover $X_{\widehat{\Sigma}_\fa}$ of $\Bl_{\fa,\bb}\AA^n$. By replacing $I$ in \eqref{P:ttomi-1} by its monomial saturation, we may assume $I$ is monomial. Recall from Section~\ref{X:explicating-multi-weighted-blow-ups}\eqref{X:emwbu-1} that for every monomial $\pmb{x}^{\ba} \in I$, we have \begin{equation}\label{E:monomial-transform}
    \pmb{x}^{\ba} = \prod_{i=1}^n{\left(x_i'\right)^{a_i}} \cdot \prod_{\rho \in \bE(\fa)}{\left(x_\rho'\right)^{\ub_\rho \cdot (\bu_\rho \cdot \ba)}} \quad \textrm{in $\pi_{\fa,\bb}^{-1}I \cdot \sO$}.
\end{equation}
Then part \eqref{P:ttomi-1} follows from \eqref{E:monomial-transform} since it says that for every $\rho \in \bE(\fa)$ (resp.\ $i \in [1,n]$) and $\ba \in P_I$, we have $\bu_\rho \cdot \ba \geq n_\rho$ (resp.\ $a_i \geq n_i$) if and only if $(x_\rho')^{\ub_\rho \cdot n_\rho}$ divides $\pmb{x}^{\ba}$ (resp.\ $(x_i')^{n_i}$ divides $\pmb{x}^\ba$).

The forward inclusion in \eqref{P:ttomi-2}  follows from \eqref{P:ttomi-1}. For the reverse inclusion, it suffices to compute locally on the open charts $U_\sigma \subset X_{\widehat{\Sigma}_\fa}$ as $\sigma$ varies over all maximal cones of $\Sigma_\fa$. Therefore, fix a maximal cone $\sigma$ of $\Sigma_\fa$, and, as in Section~\ref{X:conventions-2.1}\eqref{not-b}, let $\bv_\sigma = (\uv_{\sigma,i})_{i=1}^n$ be the corresponding vertex of $P_\fa$. 
%(\ref{X:explicating-multi-weighted-blow-ups}(iii)). 
Setting $\ba = \bv_\sigma$ in \eqref{E:monomial-transform}, we have \[
    \pmb{x}^{\bv_\sigma} = \prod_{i=1}^n{\left(x_i'\right)^{\uv_{\sigma,i}}} \cdot \prod_{\rho \in \bE(\fa)}{\left(x_\rho'\right)^{\ub_\rho \cdot (\bu_\rho \cdot \bv_\sigma)}} \quad \textrm{in $\pi_{\fa,\bb}^{-1}\fa \cdot \sO$}.
\]
Recalling that $x_\rho'$ is invertible on $U_\sigma$ for any $\rho \in \Sigma_\fa(1)$ such that $\rho \not\subset \sigma$, we obtain \[
    \pi^{-1}_{\fa,\bb}\fa \cdot \sO \supset \left(\pmb{x}^{\bv_\sigma}\right) = \prod_{\substack{i \in [1,n] \\ \RR_{\geq 0} \cdot \be_i \subset \sigma}}{\left(x_i'\right)^{\uv_{\sigma,i}}} \cdot \prod_{\substack{\rho \in \bE(\fa) \\ \rho \subset \sigma}}{\left(x_\rho'\right)^{\ub_\rho \cdot (\bu_\rho \cdot \bv_\sigma)}} \quad \textrm{on $U_\sigma$}. 
\]
To complete the proof of \eqref{P:ttomi-2}, it remains to note the following. For each $\rho \in {\bE(\fa)}$ with $\rho \subset \sigma$, we have $\bv_\sigma \in H_\rho$, which implies $\bu_\rho \cdot \bv_\sigma = N_\rho(\fa)$. Likewise, if $\rho = i \in [1,n]$ with $\RR_{\geq 0} \cdot \be_i \subset \sigma$, we have $\uv_{\sigma,i} = N_i(\fa)$.
\end{proof}

\begin{remark}\label{R:total-transform-of-monomial-ideal}
In Proposition~\ref{P:total-transform-of-monomial-ideal}\eqref{P:ttomi-2}, note that for every $i \in [1,n]$, $N_i(\fa) > 0$ if and only if $\fa \subset (x_i)$, \textit{i.e.}~$V(\fa) \supset V(x_i)$. Analogously to how the schematic blow-up of $\AA^n$ along a divisor $D$ does nothing except declare $D$ to be ``exceptional'', the multi-weighted blow-up $\pi_{\fa,\bb}$ similarly declares the divisor $V(x_i') \subset \Bl_{\fa,\bb}\AA^n$ to be ``exceptional'' for every $i \in [1,n]$ with $\fa \subset (x_i)$. In this sense, Proposition~\ref{P:total-transform-of-monomial-ideal}\eqref{P:ttomi-2} expresses the total transform as a union of ``exceptional'' divisors $V(x_\rho')$ for $\rho \in \bE^+(\fa)$ (see Section~\ref{X:conventions-2.1}), with multiplicities.
\end{remark}

Next, we explicate two classical transforms for multi-weighted blow-ups. 

\begin{definition}[Proper transform]\label{D:proper-transform}
Set $\sO := \sO_{\Bl_{\fa,\bb}\AA^n}$. The \emph{proper} (\emph{or strict}\,) \emph{transform} of an ideal $I \subset \kk[x_1,\dotsc,x_n]$ under the multi-weighted blow-up $\pi_{\fa,\bb} \colon \Bl_{\fa,\bb}\AA^n \to \AA^n$ is \[
    \text{\small$\widetilde{\pi^{-1}_{\fa,\bb}I \cdot \sO} := \left(\pi^{-1}_{\fa,\bb}I \cdot \sO : \prod_{\rho \in \bE^+(\fa)}\left(x_\rho'\right)^\infty \right) :=  \bigcup_{(n_\rho) \in \NN^{\bE^+(\fa)}}{\left(\pi^{-1}_{\fa,\bb}I \cdot \sO : \prod_{\rho \in \bE^+(\fa)}\left(x_\rho'\right)^{n_\rho}\right)}.$}
\]
Equivalently, by Proposition~\ref{P:total-transform-of-monomial-ideal}\eqref{P:ttomi-2}, the \emph{proper transform} of the closed subscheme $V(I) \subset \AA^n$ under $\pi_{\fa,\bb}$ is the schematic closure of $V(\pi^{-1}_{\fa,\bb}I \cdot \sO) \smallsetminus V(\pi^{-1}_{\fa,\bb}\fa \cdot \sO)$ in $\Bl_{\fa,\bb}\AA^n$.
\end{definition}

\begin{definition}[Weak transform]\label{D:weak-transform}
Set $\sO := \sO_{\Bl_{\fa,\bb}\AA^n}$. The \emph{weak} (\emph{or birational, or controlled}\,) \emph{transform} of an ideal $I \subset \kk[x_1,\dotsc,x_n]$ under the multi-weighted blow-up $\pi_{\fa,\bb} \colon \Bl_{\fa,\bb}\AA^n \to \AA^n$ is \[
    \left(\pi_{\fa,\bb}\right)_\ast^{-1}I := \left(\pi^{-1}_{\fa,\bb}I \cdot \sO\right) \cdot \prod_{\rho \in \bE^+(\fa)}{\left(x_\rho'\right)^{-\ub_\rho \cdot N_\rho(I)}},  
\]
where for each $\rho \in \bE^+(\fa)$, $N_\rho(I)$ is the largest natural number $n_\rho$ such that: \begin{enumerate}
    \item the fractional ideal $(\pi_{\fa,\bb}^{-1}I \cdot \sO) \cdot (x_\rho')^{-\ub_\rho \cdot n_\rho}$ is an ideal on $\Bl_{\fa,\bb}\AA^n$; 
    \item or, equivalently by Proposition~\ref{P:total-transform-of-monomial-ideal}\eqref{P:ttomi-1}, the Newton polyhedron $P_I$ of $I$ is bounded below by the hyperplane $\sum_{i=1}^n{\uu_{\rho,i} \cdot \be_i} = n_\rho$.
\end{enumerate}
\end{definition}

\begin{remark}\label{R:transforms-properties}
By definition, we always have the inclusion \[
     \left(\pi_{\fa,\bb}\right)_\ast^{-1}I \subset \widetilde{\pi^{-1}_{\fa,\bb}I \cdot \sO}
\]
with equality if $I$ is a principal ideal. Moreover, if $I$ is radical, so is the proper transform $\widetilde{\pi_{\fa,\bb}^{-1}I \cdot \sO}$. In other words, if $V(I) \subset \AA^n$ is reduced, so is its proper transform in $\Bl_{\fa,\bb}\AA^n$.
\end{remark}

It is usually a more intricate issue to identify generators of proper transforms, as opposed to generators for weak transforms. 

\begin{example}\label{EX:weak-and-proper-transforms}
 Consider the non-principal ideal $I = (x^2+y^2,z-y^2) \subset \kk[x,y,z]$ under the multi-weighted blow-up $\pi_\fa$ in Example~\ref{EX:second-example}. \begin{enumerate}
 \item Its total transform $\pi_\fa^{-1}I \cdot \sO_{\Bl_\fa\AA^3}$ is
   \[\left(x'^2u_1^6u_2^2+y'^2u_1^4,z'u_1^2u_2^2-y'^2u_1^4\right) = \left(u_1^2\right) \cdot \left(x'^2u_1^4u_2^2+y'^2u_1^2,z'u_2^2-y'^2u_1^2\right).\]
   Hence, the weak transform of $I$ under $\pi_\fa$ is $(x'^2u_1^4u_2^2+y'^2u_1^2,z'u_2^2-y'^2u_1^2)$.
    
    \item On the other hand, while we have \[
        \qquad \qquad x^2+y^2 = u_1^4 \cdot \left(x'u_1^2u_2^2+y'^2\right) \quad \text{and} \quad z-y^2 = u_1^2 \cdot \left(z'u_2^2-y'^2u_1^2\right),
    \]
    the proper transform $\widetilde{\pi_\fa^{-1}I \cdot \sO_{\Bl_\fa\AA^3}}$ is not generated by the elements $x'u_1^2u_2^2+y'^2$ and $z'u_2^2-y'^2u_1^2$. Indeed, note that $x^2+z = (x^2+y^2)+(z-y^2) \in I$ and \[
        \qquad \qquad x^2+z = x'^2u_1^6u_2^2+z'u_1^2u_2^2 = u_1^2u_2^2 \cdot \left(x'^2u_1^4+z'\right).
    \]
    Thus, $x'^2u_1^4+z' \in \widetilde{\pi_\fa^{-1}I \cdot \sO_{\Bl_\fa\AA^3}}$, but $x'^2u_1^4+z' \notin (x'u_1^2u_2^2+y'^2,z'u_2^2-y'^2u_1^2)$.
\end{enumerate} 
\end{example}

\subsection{Multi-graded Rees algebras and idealistic exponents}\label{2.4}

In this section we reinterpret some of the earlier observations and definitions in terms of Rees algebras and idealistic exponents.

\begin{xpar}[Multi-graded Rees algebras]\label{X:multi-graded-rees-algebras}
  The discussion in Section~\ref{X:explicating-multi-weighted-blow-ups} can be summarized by the compact, but notation-heavy, statement that $\Bl_{\fa,\bb}\AA^n$ equals
  \begin{equation}\label{E:presentation-of-multi-weighted-blow-up}
    \left[\left(\GSpec_{\AA^n}(\sR) \smallsetminus V\left(J_{\Sigma_\fa}\right)\right) \q_{\tiny{\begin{bmatrix}
    \bB \\ -I_k
    \end{bmatrix}}} \; \Gm^{\bE(\fa)}\right], 
\end{equation}
where \begin{equation}\label{E:multi-graded-rees-algebra}
    \sR = \sR(\fa,\bb) := \frac{\sO_{\AA^n}[x_1',\dotsc,x_n']\left[x_\rho' \colon \rho \in \bE(\fa)\right]}{\left(x_i' \cdot \prod_{\rho \in \bE(\fa)}{\left(x_\rho'\right)^{\ub_\rho \cdot \uu_{\rho,i}}} - x_i \colon 1 \leq i \leq n\right)}
\end{equation}
and the matrix $\small{\begin{bmatrix}\bB \\ -I_k\end{bmatrix}}$ records the $\ZZ^{\bE(\fa)}$-grading of each $x_i'$ and each $x_\rho'$ as in Section~\ref{X:explicating-multi-weighted-blow-ups}\eqref{X:emwbu-2} and hence describes the $\Gm^{\bE(\fa)}$-action displayed above.

We provide a reinterpretation of the $\ZZ^{\bE(\fa)}$-graded $\sO_{\AA^n}$-algebra $\sR = \sR(\fa,\bb)$ as a \emph{multi-graded Rees algebra} on $\AA^n$. Consider the homomorphism of $\ZZ^{\bE(\fa)}$-graded $\sO_{\AA^n}$-algebras $\sR \to \sO_{\AA^n}[t_\rho^{\pm} \colon \rho \in \bE(\fa)]$ defined by \begin{align}\label{E:substitution}
    \begin{split}
    x_i' &\longmapsto \left(x_i \cdot \prod_{\rho \in \bE(\fa)}{t_\rho^{\ub_\rho \cdot \uu_{\rho,i}}}\right) \quad \textrm{for } 1 \leq i \leq n, \\
    x_\rho' &\longmapsto \left(1 \cdot t_\rho^{-1}\right) \quad \qquad \qquad \; \textrm{for } \rho \in \bE(\fa).
    \end{split}
\end{align}
This is an isomorphism of $\sR$ onto its image \begin{align}\label{E:multi-graded-rees-algebra-2}\begin{split}
    \sR_\bullet &:= \sO_{\AA^n}\left[t_\rho^{-1} \colon \rho \in \bE(\fa)\right]\left[x_i \cdot \prod_{\rho \in \bE(\fa)}{t_\rho^{\ub_\rho \cdot \uu_{\rho,i}}} \colon 1 \leq i \leq n\right] \\
    &\subset \; \sO_{\AA^n}\left[t_\rho^{\pm} \colon \rho \in \bE(\fa)\right].
\end{split}
\end{align}
The image $\sR_\bullet$ is a \emph{$\ZZ^{\bE(\fa)}$-graded Rees algebra} on $\AA^n$; \textit{i.e.} it is a finitely generated, quasi-coherent $\ZZ^{\bE(\fa)}$-graded $\sO_{\AA^n}$-subalgebra \[
    \sR_\bullet = \bigoplus_{\bm \in \ZZ^{\bE(\fa)}}{\sR_{\bm}} \; \subset \; \sO_{\AA^n}\left[t_\rho^{\pm} \colon \rho \in \bE(\fa)\right]
\]
satisfying the following three conditions:
\begin{enumerate}
    \item\label{cond1} $\sR_{\mathbf{0}} = \sO_{\AA^n}$.
    \item\label{cond2} $1 \cdot t_\rho^{-1} \in \sR_\bullet$ for all $\rho \in \bE(\fa)$.
    \item\label{cond3} For every $\bm \in \ZZ^{\bE(\fa)}$, we have $\sR_\bm = \bigcap_{\rho \in \bE(\fa)}{\sR_{m_\rho \cdot \be_\rho}}$.
\end{enumerate}
Note that under \eqref{cond1}, condition~\eqref{cond2} is equivalent to
\begin{enumerate}
    \myitem[ii$'$]\label{cond2'} $\sR_{\bm+\be_\rho} \subset \sR_\bm$ for every $\bm \in \ZZ^{\bE(\fa)}$ and $\rho \in \bE(\fa)$.
\end{enumerate}
In particular, \eqref{cond2} already implies the forward inclusion in \eqref{cond3}. Moreover, note that \eqref{cond3} is redundant if $\#\bE(\fa) = 1$.

We usually make the identification $\sR = \sR_\bullet$ and hence will not make any distinction between both sides of \eqref{E:substitution}. Occasionally, we neglect the negative degrees and only work with the $\NN^{\bE(\fa)}$-graded part of $\sR_\bullet$, which is an \emph{$\NN^{\bE(\fa)}$-graded Rees algebra}\footnote{In other words, it is a finitely generated, quasi-coherent $\NN^{\bE(\fa)}$-graded $\sO_{\AA^n}$-subalgebra $\bigoplus_{\bm \in \NN^{\bE(\fa)}}{\sR_\bm} \subset \sO_{\AA^n}[t_\rho \colon \rho \in \bE(\fa)]$ satisfying \eqref{cond1}, \eqref{cond2'}, and \eqref{cond3}, where the phrase ``$\bm \in \ZZ^{\bE(\fa)}$'' in the last two conditions is replaced by ``$\bm \in \NN^{\bE(\fa)}$''.} on $\AA^n$.
\end{xpar}

\begin{remark}[Alternative description of the proper transform]\label{R:alternative-definition-proper-transform}
It is also under the interpretation in Section~\ref{X:multi-graded-rees-algebras} that the proper transform of a closed subscheme $V(I) \subset \AA^n$ under $\pi_{\fa,\bb} \colon \Bl_{\fa,\bb}\AA^n \to \AA^n$ has a natural description. Namely, it is given by the similar-looking expression \[
    \left[\left(\GSpec_{V(I)}\left(\sR|_{V(I)}\right) \smallsetminus V\left(J_{\Sigma_\fa}\right)\right) \q_{\tiny{\begin{bmatrix}
    \bB \\ -I_k
    \end{bmatrix}}} \; \Gm^{\bE(\fa)}\right].
\]
If one interprets this as the ``\emph{multi-weighted blow-up of\, $V(I)$ along $\sR|_{V(I)}$}'', this description parallels that in \cite[Corollary II.7.15]{hartshorne-algebraic-geometry}.
\end{remark}

For Section~\ref{X:idealistic-exponents} below only, let $Y$ denote a $\kk$-variety; \textit{e.g.}~$Y = \AA^n$. In \cite[Section 2.2]{quek-weighted-log-resolution} one defines a one-to-one correspondence between non-zero, integrally closed, $\NN$-graded Rees algebras on $Y$ and \emph{idealistic exponents}, see \cite[Definition 2.7]{quek-weighted-log-resolution}, over $Y$. This can be immediately promoted to a one-to-one correspondence between non-zero, integrally closed, $\NN^k$-graded Rees algebras $\sR_\bullet$ on $Y$ and $k$-tuples $\gamma = (\gamma^{[\rho]} \colon \rho \in [1,k])$ of idealistic exponents over $Y$.

\begin{xpar}[Tuples of idealistic exponents]\label{X:idealistic-exponents}
  Let $\ZR(Y)$ denote the Zariski--Riemann space of $Y$, and let $\pi_Y \colon \ZR(Y) \to Y$ denote the morphism of locally ringed spaces which maps $\nu \in \ZR(Y)$ to the center $y_\nu$ of $\nu$ on $Y$; see \cite[Appendix A]{quek-weighted-log-resolution}. Then:
  \begin{enumerate}[label=(\alph*),ref=\alph*]
    \item\label{X:ie-1} Given a $k$-tuple $\gamma := (\gamma^{[\rho]} \colon \rho \in [1,k])$ of idealistic exponents over $Y$, let $\sR_\bullet^{[\rho]}$ be the integrally closed, $\NN$-graded Rees algebra on $Y$ associated to each $\gamma^{[\rho]}$. Then the integrally closed $\NN^k$-graded Rees algebra $\sR_\bullet := \bigoplus_{\bm \in \NN^k}{\sR_\bm \cdot \pmb{t}^{\bm}}$ on $Y$ associated to $\gamma$ is defined by $\sR_\bm := \bigcap_{\rho \in [1,k]}{\sR^{[\rho]}_{m_\rho}}$ for every $\bm \in \NN^k$. In other words, for any open $U \subset Y$, \[
        \qquad \qquad \sR_\bm(U) := \left\lbrace g \in \sO_Y(U) \colon \parbox{5 cm}{$\nu(g) \geq m_\rho \cdot \left(\gamma^{[\rho]}\right)_\nu$ for every $\nu \in \pi_Y^{-1}(U)$ and $\rho \in [1,k]$} \right\rbrace.
    \]
    \item\label{X:ie-2} Conversely, to a non-zero, integrally closed, $\NN^k$-graded Rees algebra $\sR_\bullet$ on $Y$, we associate a $k$-tuple $\gamma := (\gamma^{[\rho]} \colon \rho \in [1,k])$ of idealistic exponents over $Y$, where each $\gamma^{[\rho]}$ is the idealistic exponent over $Y$ associated to the non-zero, integrally closed, $\NN$-graded Rees algebra $\sR^{[\rho]}_\bullet := \bigoplus_{m \in \NN}{\sR_{m \cdot \be_\rho} \cdot t^m}$. In other words, the stalk of $\gamma^{[\rho]}$ at each $\nu \in \ZR(Y)$ is \[
      \qquad \qquad (\gamma^{[\rho]})_\nu := \min\left\lbrace \frac{1}{m_\rho} \cdot \nu(g) \colon 0 \neq g \cdot \pmb{t}^{\bm} \in \left(\sR_\bullet\right)_{y_\nu} \textrm{ with } m_\rho \geq 1 \right\rbrace.
    \]
\end{enumerate}
Together, \eqref{X:ie-1} and \eqref{X:ie-2} give the desired one-to-one correspondence.
\end{xpar}

\begin{xpar}\label{X:associated-idealistic-exponents}
Under the above one-to-one correspondence, the $\NN^{\bE(\fa)}$-graded part of $\sR(\fa,\bb)_\bullet$ in \eqref{E:multi-graded-rees-algebra-2} then corresponds to the tuple $\gamma(\fa,\bb) := \left(\gamma(\fa,\bb)^{[\rho]} \colon \rho \in \bE(\fa)\right)$ of $\#\bE(\fa)$ idealistic exponents over $\AA^n$, where each $\gamma(\fa,\bb)^{[\rho]}$ is defined stalk-wise at each $\nu \in \ZR(\AA^n)$ by \begin{equation}\label{E:associated-idealistic-exponents}
    \left(\gamma(\fa,\bb)^{[\rho]}\right)_\nu := \min_{\substack{i \in [1,n] \\ \uu_{\rho,i} \neq 0}}\left(\frac{1}{\ub_\rho \cdot \uu_{\rho,i}} \cdot \nu(x_i)\right).
\end{equation}
Following \cite[Notation 2.12]{quek-weighted-log-resolution}, we use the suggestive notation \[
    \left(x_i^{\frac{1}{\ub_\rho \cdot \uu_{\rho,i}}} \colon i \in [1,n],\ \uu_{\rho,i} \neq 0\right)
\]
to denote, for each $\rho \in [1,k]$, the corresponding integrally closed, $\NN$-graded Rees algebra $\sR(\fa,\bb)^{[\rho]}$.
\end{xpar}

\hfill

Our next objective is to give a coordinate-free interpretation of the weak transform (see Definition~\ref{D:weak-transform}) in terms of the idealistic exponents in $\gamma(\fa,\bb)$.

\begin{xpar}[Conventions]\label{X:conventions-2.4}
For the remainder of this section, fix a monomial ideal $\fa$ on $\AA^n$ and $\bb \in \NN_{>0}^{\bE(\fa)}$. Let $\widetilde{\pi} := \widetilde{\pi}_{\fa,\bb}$ denote the composition \[
    X_{\widehat{\Sigma}_\fa} \xtwoheadrightarrow[\textrm{quotient}]{\textrm{stack-theoretic}} \Bl_{\fa,\bb}\AA^n \xrightarrow{\pi_{\fa,\bb}} \AA^n.
\]
For an ideal $I$ on $\AA^n$ (or $X_{\widehat{\Sigma}_\fa}$), let $\gamma_I$ denote the idealistic exponent over $\AA^n$ (or $X_{\widehat{\Sigma}_\fa}$) associated to $I$; see \cite[Sections 2.1--2.2]{quek-weighted-log-resolution}. Let $\widetilde{\pi}^{-1}\gamma \cdot \sO$ denote the pull-back of $\gamma$ to $X_{\widehat{\Sigma}_\fa}$ via $\widetilde{\pi}$. Unless otherwise mentioned, set $\sO := \sO_{X_{\Sigma_\fa}}$, and set $\gamma^{[\rho]} := \gamma(\fa,\bb)^{[\rho]}$ for $\rho \in \bE(\fa)$. For $i \in [1,n]$, also set $\gamma^{[i]} := \gamma_{(x_i)}$.
\end{xpar}

\hfill

To reinterpret the weak transform, we begin with the following elementary lemma. 

\begin{lemma}\label{L:inclusion-vs-inequality}
  For $(k_\rho)_{\rho \in \Sigma_\fa(1)} \in \NN^{\Sigma_\fa(1)}$, the following statements are equivalent:
  \begin{enumerate}
    \item\label{L:ivi-1} $\widetilde{\pi}^{-1}I \cdot \sO \subset \prod_{\rho \in \Sigma_\fa(1)}{(x_\rho')^{k_\rho}}$.
    \item\label{L:ivi-2} $\widetilde{\pi}^{-1}\gamma_I \cdot \sO \geq \sum_{\rho \in \Sigma_\fa(1)}{k_\rho \cdot \gamma_{(x_\rho')}}$.
    \myitem[i$'$]\label{L:ivi-1'} $\widetilde{\pi}^{-1}I \cdot \sO \subset (x_\rho')^{k_\rho}$ for every $\rho \in \Sigma_\fa(1)$.
    \myitem[ii$'$]\label{L:ivi-2'} $\widetilde{\pi}^{-1}\gamma_I \cdot \sO \geq k_\rho \cdot \gamma_{(x_\rho')}$ for every $\rho \in \Sigma_\fa(1)$.
\end{enumerate}
\end{lemma}

\begin{proof}
The equivalence \eqref{L:ivi-1} $\Leftrightarrow$ \eqref{L:ivi-1'} is evident. The equivalences \eqref{L:ivi-1} $\Leftrightarrow$ \eqref{L:ivi-2} and \eqref{L:ivi-1'} $\Leftrightarrow$ \eqref{L:ivi-2'} follow from the fact that if $X$ is a normal variety, $I$ is an ideal on $X$, and $D$ is a divisor on $X$ with underlying ideal $I_D$, then $\gamma_D \leq \gamma_I$ if and only if $I_D \supset I$. Indeed, by \cite[Lemma A.1]{quek-weighted-log-resolution}, both statements are equivalent to $I \cdot I_D^{-1} \subset \sO_X$. 
\end{proof}

Our next goal is to provide a recharacterization of statement  \eqref{L:ivi-2'} in Lemma~\ref{L:inclusion-vs-inequality} in terms of idealistic exponents over $\AA^n$. Before doing that, we need the following lemma. 

\begin{lemma}\label{L:inequality-idealistic-exponents}
For each $\rho \in \Sigma_\fa(1)$, we have \[
    \widetilde{\pi}^{-1}\gamma^{[\rho]} \cdot \sO \geq \gamma_{(x_\rho')}.
\]
\end{lemma}

\begin{proof}
Let $\nu \in \ZR(X_{\widehat{\Sigma}_\fa})$ be arbitrary. If $\rho = i \in [1,n]$, then the lemma  follows from \[
    \gamma_{(x_i'),\nu} = \nu(x_i') = \nu(x_i) - \sum_{\rho \in \bE(\fa)}{\left(\ub_\rho \cdot \uu_{\rho,i}\right) \cdot \nu\left(x_\rho'\right)} \leq \nu(x_i) = \left(\widetilde{\pi}^{-1}\gamma^{[i]} \cdot \sO\right)_\nu.
\]
If instead $\rho \in \bE(\fa)$, we have, for every $1 \leq i \leq n$ such that $\uu_{\rho,i} \neq 0$,   \begin{align*}
    \gamma_{(x_\rho'),\nu} = \nu\left(x_\rho'\right) &= \frac{1}{\ub_\rho \cdot \uu_{\rho,i}} \cdot \left(\nu(x_i) - \nu(x_i') - \sum_{\widetilde{\rho} \in \bE(\fa) \smallsetminus \lbrace \rho \rbrace}{\left(\ub_{\widetilde{\rho}} \cdot \uu_{\widetilde{\rho},i}\right) \cdot \nu\left(x_{\widetilde{\rho}}'\right)}\right) \\
    &\leq \frac{1}{\ub_\rho \cdot \uu_{\rho,i}} \cdot \nu(x_i).
\end{align*}
Taking the minimum over all such $1 \leq i \leq n$, we obtain the lemma.
\end{proof}

\begin{xpar}[More conventions]\label{X:more-conventions-2.4}
Let $\ZR(\widetilde{\pi}) \colon \ZR(X_{\widehat{\Sigma}_\fa}) \to \ZR(\AA^n)$ denote the morphism of Zariski--Riemann spaces induced by $\widetilde{\pi}$; see \cite[Appendix A.3]{quek-weighted-log-resolution}. For each $\rho \in \Sigma_\fa(1)$, let $\nu_\rho'$ be the $V(x_\rho')$-divisorial valuation on $X_{\widehat{\Sigma}_\fa}$, and let $\nu_\rho = \ZR(\widetilde{\pi})(\nu_\rho')$. By Section~\ref{X:explicating-multi-weighted-blow-ups}\eqref{X:emwbu-1}, $\nu_\rho(x_i) = \ub_\rho \cdot \uu_{\rho,i}$ for every $\rho \in \Sigma_\fa(1)$ and $1 \leq i \leq n$.
\end{xpar}

\begin{proposition}\label{P:admissibility-criterion}
  For $\rho \in \Sigma_\fa(1)$ and $k \in \QQ_{>0}$, the following statements are equivalent:
  \begin{enumerate}
    \item\label{P:ac-1} $\gamma_I \geq k \cdot \gamma^{[\rho]}$.
    \item\label{P:ac-2} $\widetilde{\pi}^{-1}\gamma_I \cdot \sO \geq k \cdot \gamma_{(x_\rho')}$.
    \item\label{P:ac-3} $\gamma_{I,\nu_\rho} \geq k$ $\left(= k \cdot (\gamma^{[\rho]})_{\nu_\rho} = k \cdot \gamma_{(x_\rho'),\nu_\rho'}\right)$.
\end{enumerate}
\end{proposition}

\begin{proof}
The equivalence \eqref{P:ac-1} $\Rightarrow$ \eqref{P:ac-2} follows from Lemma~\ref{L:inequality-idealistic-exponents}. For \eqref{P:ac-2} $\Rightarrow$ \eqref{P:ac-3}, we localize the inequality in \eqref{P:ac-2} at $\nu := \nu_\rho'$ to obtain $\gamma_{I,\nu_\rho} = (\widetilde{\pi}^{-1}\gamma_I \cdot \sO)_\nu \geq k \cdot \gamma_{E_\rho,\nu} = k$. For \eqref{P:ac-3} $\Rightarrow$ \eqref{P:ac-1}, note that \eqref{P:ac-3} says that for $f = \sum_{\ba}{c_{\ba} \cdot \pmb{x}^{\ba}} \in I$, we have \begin{equation}\label{E:inequality-from-iii}
    \min_{c_{\ba} \neq 0}\left\lbrace\sum_{i=1}^n{a_i \cdot \left(\ub_\rho \cdot \uu_{\rho,i}\right)}\right\rbrace = \nu_\rho(f) \geq \gamma_{I,\nu_\rho} \geq k.
\end{equation}
Then for arbitrary $\nu \in \ZR(\AA^n)$ and $f = \sum_{\ba}{c_\ba \cdot \pmb{x}^{\ba}} \in I$, we have \begingroup
\allowdisplaybreaks
\begin{align*}
    \nu(f) \geq \min_{c_\ba \neq 0}\left\lbrace\sum_{i=1}^n{a_i \cdot \nu(x_i)}\right\rbrace &\geq \min_{c_{\ba} \neq 0}\left\lbrace\sum_{\substack{i=1 \\ \uu_{\rho,i} \neq 0}}^n{a_i\left(\ub_\rho \cdot \uu_{\rho,i}\right)\left(\frac{1}{\ub_\rho \cdot \uu_{\rho,i}} \cdot \nu(x_i) \right)}\right\rbrace \\
    &\geq (\gamma^{[\rho]})_\nu \cdot \min_{c_{\ba} \neq 0}\left\lbrace\sum_{i=1}^n{a_i \left(\ub_\rho \cdot \uu_{\rho,i}\right)}\right\rbrace \geq k \cdot (\gamma^{[\rho]})_\nu, 
\end{align*}
\endgroup
where the last inequality follows from \eqref{E:inequality-from-iii}. Therefore, $\gamma_{I,\nu} = \min\lbrace \nu(f) \colon f \in I \rbrace \geq k \cdot (\gamma^{[\rho]})_\nu$. This proves \eqref{P:ac-1}.
\end{proof}

\begin{remark}
The proof of Proposition~\ref{P:admissibility-criterion} suggests that we should interpret $\gamma_I$ as the ``Newton polyhedron $P_I$ of $I$'' and $\gamma^{[\rho]}$ as the ``hyperplane $\sum_{i=1}^n{(\ub_\rho \cdot \uu_{\rho,i}) \cdot \be_i}$ $= k$''. Then Proposition~\ref{P:admissibility-criterion}\eqref{P:ac-1} translates to the statement that ``$P_I$ is bounded below by the hyperplane $\sum_{i=1}^n{(\ub_\rho \cdot \uu_{\rho,i}) \cdot \be_i} = k$'' (see Definition~\ref{D:newton-polyhedron-of-ideal}). Combining Lemma~\ref{L:inclusion-vs-inequality} and Proposition~\ref{P:admissibility-criterion}, we get a reinterpretation of Proposition~\ref{P:total-transform-of-monomial-ideal}\eqref{P:ttomi-1} in terms of idealistic exponents. This reinterpretation is justified by  \eqref{E:inequality-from-iii}, which says that for every $\ba \in P_I$, we have $(\ub_\rho \cdot \bu_\rho) \cdot \ba \geq k$.
\end{remark}

\begin{remark}\label{R:admissibility-criterion-for-monomial-ideals}
Let us apply Proposition~\ref{P:admissibility-criterion} to $I = \fa$. For every $\rho \in \Sigma_\fa(1)$, one can compute that $\gamma_{\fa,\nu_\rho} = \ub_\rho \cdot N_\rho(\fa)$, so the proposition says $\gamma_\fa \geq (\ub_\rho \cdot N_\rho(\fa)) \cdot \gamma^{[\rho]}$. In fact, $\sup\lbrace k \in \QQ_{>0} \colon \gamma_\fa \geq k \cdot \gamma^{[\rho]} \rbrace = \ub_\rho \cdot N_\rho(\fa)$ because whenever $\gamma_\fa \geq k \cdot \gamma^{[\rho]}$, we have $k \leq \gamma_{\fa,\nu_\rho} = \ub_\rho \cdot N_\rho(\fa)$. By Lemma~\ref{L:inclusion-vs-inequality}, we therefore have $\widetilde{\pi}^{-1}\gamma_\fa \cdot \sO \geq \sum_{\rho \in \bE^+(\fa)}{(\ub_\rho \cdot N_\rho(\fa)) \cdot \gamma_{(x_\rho')}}$. In fact, Proposition~\ref{P:total-transform-of-monomial-ideal}\eqref{P:ttomi-2} says more: this inequality is an equality!
\end{remark}

By the equivalences in Lemma~\ref{L:inclusion-vs-inequality} and Proposition~\ref{P:admissibility-criterion}, the following definition of the weak transform is equivalent to Definition~\ref{D:weak-transform}.

\begin{definition}[Weak transform, revisited]\label{D:weak-transform-revisited}
Set $\sO := \sO_{\Bl_{\fa,\bb}\AA^n}$. The weak transform of an ideal $I \subset \kk[x_1,\dotsc,x_n]$ under the multi-weighted blow-up $\pi_{\fa,\bb} \colon \Bl_{\fa,\bb}\AA^n \to \AA^n$ is \[
    (\pi_{\fa,\bb})_\ast^{-1}I :=  \left(\pi^{-1}_{\fa,\bb}I \cdot \sO\right) \cdot \prod_{\rho \in \bE^+(\fa)}{\left(x_\rho'\right)}^{-K_\rho(I)}, 
\]
where for each $\rho \in \bE^+(\fa)$, $K_\rho(I)$ is the largest natural number $k_\rho$ such that $\gamma_I \geq k_\rho \cdot \gamma^{[\rho]}$ (or equivalently, $\gamma_{I,\nu_\rho} \geq k_\rho$).
\end{definition}

%\newpage
\section{Multi-weighted blow-ups: Canonical aspects}\label{C:multi-weighted-blow-ups:canonical-aspects}

\subsection{Canonicity of multi-weighted blow-ups, I}\label{3.1}

In this section we continue to follow the conventions in Section~\ref{X:conventions-2.1}, and we endow $\AA^n$ with the logarithmic structure $\NN^n \xrightarrow{\be_i \mapsto x_i} \kk[x_1,\dotsc,x_n]$. Let $\fa_\bullet$ be an ($\NN$-graded) \emph{monomial Rees algebra on $\AA^n$}, \textit{i.e.}~a finitely generated, $\NN$-graded $\sO_{\AA^n}$-subalgebra $\fa_\bullet = \bigoplus_{m \in \NN}{\fa_m \cdot t^m} \subset \sO_{\AA^n}[t]$ such that $\fa_0 = \sO_{\AA^n}$, $\fa_m \supset \fa_{m+1}$ for every $m \in \NN$ and each $\fa_m$ is a monomial ideal of $\kk[x_1,\dotsc,x_n]$. 

We give a definition of $\Bl_{\fa_\bullet}\AA^n$, which generalizes the notion of $\Bl_\fa\AA^n$ (see Definition~\ref{D:multi-weighted-blow-up}) for a monomial ideal $\fa$ on $\AA^n$, before demonstrating that this notion is canonically associated to $\fa_\bullet$.

\begin{definition}[Multi-weighted blow-ups along monomial Rees algebras]\label{D:multi-weighted-blow-up-along-monomial-rees-algebras}
Fix a sufficiently large $\ell \in N_{>0}$ such that the $\ell$\textsuperscript{th} Veronese subalgebra $\fa_{\ell\bullet}$ of $\fa_\bullet$ is generated in degree $1$. The multi-weighted blow-up of $\AA^n$ along $\fa_\bullet$ is then defined as \[
    \pi_{\fa_\bullet} \colon \Bl_{\fa_\bullet}\AA^n := \Bl_{\fa_\ell,\widetilde{\bb}}\AA^n \xrightarrow{\pi_{\fa_\ell,\widetilde{\bb}}} \AA^n, 
\]
where \[
    \widetilde{\bb} := \left(\frac{\ell}{\gcd\left(\ell,N_\rho(\fa_\ell)\right)} \colon \rho \in \Sigma_{\fa_\ell}(1) \right) \in \NN_{>0}^{\Sigma_{\fa_\ell}(1)}
\]
(see Remark~\ref{R:generalized-multi-weighted-blow-up}). We endow $\Bl_{\fa_\bullet}\AA^n$ with the toroidal logarithmic structure ``dictated by that of $\AA^n = \Spec(\NN^n \xrightarrow{\be_i \mapsto x_i} \kk[x_1,\dotsc,x_n])$ and its exceptional divisors''. Namely, it is obtained by descent from the following toroidal logarithmic structure on $\AA^{\Sigma_\fa(1)} \smallsetminus V(J_{\Sigma_\fa})$: 
\[
    \NN^{\Sigma_\fa(1)} \to \kk[x_1',\dotsc,x_n']\left[x_\rho' \colon \rho \in \bE(\fa)\right],
\]
which sends $\be_\rho$ to $x_\rho'$ for every $\rho \in \Sigma_\fa(1)$.
\end{definition}

Note that if $\fa_\bullet$ is generated in degree $1$, then $\Bl_{\fa_\bullet}\AA^n$ equals $\Bl_{\fa_1}\AA^n$ in Definition~\ref{D:multi-weighted-blow-up}. It is also simple but essential to verify the following. 

\begin{lemma}
The definition of $\Bl_{\fa_\bullet}\AA^n$ does not depend on the choice of $\ell \in \NN_{>0}$ such that $\fa_{\ell\bullet}$ is generated in degree $1$.
\end{lemma}

\begin{proof}
Let $\ell,L \in \NN_{>0}$ be such that both $\fa_{\ell\bullet}$ and $\fa_{L\bullet}$ are generated in degree $1$. By comparing $\fa_{\ell\bullet}$ and $\fa_{L\bullet}$ with $\fa_{\ell L \bullet}$, we reduce to the case where $L = r\ell$ for some $r \in \NN_{>0}$. Then $\fa_L = (\fa_\ell)^r$, and thus the normal fans of $\fa_\ell$ and of $\fa_L$ are identical. In particular, $\Sigma_{\fa_\ell}(1) = \Sigma_{\fa_L}(1)$. Lastly, note that $N_\rho(\fa_L) = r \cdot N_\rho(\fa_\ell)$ for every $\rho \in \Sigma_{\fa_\ell}(1) = \Sigma_{\fa_L}(1)$, so $\frac{\ell}{\gcd(\ell,N_\rho(\fa_\ell))} = \frac{L}{\gcd(L,N_\rho(\fa_L))}$.
\end{proof}

\begin{remark}\label{R:reduction-to-monomial-Q-ideals}
Let $\overline{\fa}_\bullet$ denote the integral closure of $\fa_\bullet$ in $\sO_{\AA^n}[t]$. By Remark~\ref{R:same-newton-polyhedron-for-two-monomial-ideals}, we then have $\Bl_{\fa_\bullet}\AA^n = \Bl_{\overline{\fa}_\bullet}\AA^n$.

In particular, for any $\ell \in \NN_{>0}$ such that $\fa_{\ell\bullet}$ is generated in degree $1$, let $\fa_\ell^{1/\ell}$ be the integral closure in $\sO_{\AA^n[t]}$ of the $\sO_{\AA^n}$-subalgebra generated by $\fa_\ell \cdot t^\ell := \lbrace \pmb{x}^\ba \colon \pmb{x}^\ba \in \fa_\ell \rbrace$; see \cite[Notation 2.12]{quek-weighted-log-resolution}. Then \[
    \fa_\ell^{1/\ell} = \overline{\fa}_\bullet
\]
since they are both integrally closed and their $\ell$\textsuperscript{th} Veronese subalgebras coincide. Hence, $\Bl_{\fa_\bullet}\AA^n = \Bl_{\fa_\ell^{1/\ell}}\AA^n$.
\end{remark}

In \cite[Section 4]{quek-weighted-log-resolution} one also associates to $\fa_\bullet$ the following blow-up. 

\begin{definition}\label{D:weighted-monoidal-blow-up}
The \emph{weighted toroidal blow-up} of $\AA^n$ along $\fa_\bullet$ is the following stack-theoretic Proj {(see \cite[Section 10.2]{olsson-algebraic-stacks}} or \cite[Section 2]{quek-rydh-weighted-blow-up}) over $\AA^n$: \[
    \tBl_{\fa_\bullet}\AA^n := \SProj_{\AA^n}(\overline{\fa}_\bullet) := \left[\GSpec_{\AA^n}(\overline{\fa}_\bullet) \smallsetminus V(\overline{\fa}_+) \q \Gm\right] \xrightarrow{\varpi_{\fa_\bullet}} \AA^n, 
\]
where $\overline{\fa}_+$ is the ideal generated by the positive degrees of $\overline{\fa}_\bullet$ and the $\Gm$-action is induced by the $\ZZ$-grading on $\overline{\fa}_\bullet$. This is a Deligne--Mumford stack. 
\end{definition}

\begin{xpar}[Charts]\label{X:charts-for-weighted-monoidal-blow-up}
Write $\overline{\fa}_\bullet = \fa^{1/\ell}$ for a monomial ideal $\fa$ on $\AA^n$ and $\ell \in \NN_{>0}$ (see Remark~\ref{R:reduction-to-monomial-Q-ideals}). As $\bv$ varies over the vertices of $P_\fa$, the open substacks \[
    D_+\left(\pmb{x}^{\bv} \cdot t^\ell\right) := \left[\GSpec_{\AA^n}\left(\overline{\fa}_\bullet\left[\left(\pmb{x}^{\bv} \cdot t^\ell\right)^{-1}\right]\right) \q \Gm\right]
\]
cover $\tBl_{\fa_\bullet}\AA^n$; see {\eqref{E:integral-closure-of-monomial-ideal} below.} Here, $D_+(\pmb{x}^\bv \cdot t^\ell)$ is known as the $(\pmb{x}^\bv \cdot t^\ell)$-chart of $\tBl_{\fa_\bullet}\AA^n$.

Note too that $t^{-\ell} = \pmb{x}^{\bv} \cdot (\pmb{x}^{\bv} \cdot t^\ell)^{-1} \in \overline{\fa}_\bullet[(\pmb{x}^{\bv} \cdot t^\ell)^{-1}]$. Since $\overline{\fa}_\bullet[(\pmb{x}^{\bv} \cdot t^\ell)^{-1}]$ is integrally closed in $\sO_{\AA^n}[t^\pm]$, we therefore have $t^{-1} \in \overline{\fa}_\bullet[(\pmb{x}^{\bv} \cdot t^\ell)^{-1}]$. The exceptional divisor of $\tBl_{\fa_\bullet}\AA^n$ on $D_+(\pmb{x}^\bw \cdot t^r)$ is then given by $V(t^{-1})$. 
\end{xpar}

\begin{remark}\label{R:extended-rees-algebra}
  For various subtle reasons, we prefer to work instead with a ``global'' version of $\fa_\bullet$, namely the \emph{extended Rees algebra} $\overline{\fa}_\bullet^\extd$ associated to $\overline{\fa}_\bullet$. {(In fact, such extended Rees algebras were simply referred to as Rees rings or algebras; see \cite{rees-problem-of-zariski} or \cite[Section 15, Section 4]{matsumura-commutative-ring-theory}.)} The extended Rees algebra
  $\overline{\fa}_\bullet^\extd$ is defined as the $\ZZ$-graded $\sO_{\AA^n}$-subalgebra $\bigoplus_{m \in \ZZ}{\overline{\fa}_m^\extd \cdot t^m}$ of $\sO_{\AA^n}[t^\pm]$ obtained by extending $\overline{\fa}_\bullet$ trivially in negative degrees: \[
    \overline{\fa}_m^\extd = \begin{cases}
        \overline{\fa}_m \quad &\textrm{if } m \geq 0, \\
        \sO_{\AA^n} \quad &\textrm{if } m < 0.
    \end{cases}
\]
Since $\overline{\fa}_\bullet[(\pmb{x}^\bv \cdot t^\ell)^{-1}] = \overline{\fa}_\bullet^\extd[(\pmb{x}^\bv \cdot t^\ell)^{-1}]$ for every vertex $\bv$ of $P_\fa$, Definition~\ref{D:weighted-toroidal-blow-up} continues to hold after replacing $\overline{\fa}_\bullet$ with $\overline{\fa}_\bullet^\extd$: %(see \cite[Remark 3.24]{quek-rydh-weighted-blow-up}): 
\[
    \tBl_{\fa_\bullet}\AA^{n;r} := \SProj_{\AA^n}\left(\overline{\fa}_\bullet^\extd\right) := \left[\GSpec_{\AA^n}\left(\overline{\fa}_\bullet^\extd\right) \smallsetminus V\left(\overline{\fa}_+^\extd\right) \q \Gm\right].
\]
In the above expression, the exceptional divisor of $\tBl_{\fa_\bullet}\AA^{n;r}$ can now be seen directly as the vanishing locus of $(t^{-1}) \subset \overline{\fa}_\bullet^\extd$, without having to pass to charts; see Section~\ref{X:charts-for-weighted-monoidal-blow-up}. This is one way in which $\overline{\fa}_\bullet^\extd$ ``globalizes''~$\overline{\fa}_\bullet$.
\end{remark}

\begin{xpar}[Logarithmic structure]\label{X:logarithmic-structure-of-weighted-monoidal-blow-up}
  The object
  $\tBl_{\fa_\bullet}\AA^n$ carries the toroidal logarithmic structure dictated by that of $\AA^n = \Spec(\NN^n \xrightarrow{\be_i \mapsto x_i} \kk[x_1,\dotsc,x_n])$ and the exceptional divisor, as demonstrated in \cite[Section~4]{quek-weighted-log-resolution}. 

To expound on this, write $\overline{\fa}_\bullet = \fa^{1/\ell}$ for a monomial ideal $\fa$ on $\AA^n$ and $\ell \in \NN_{>0}$ (see Remark~\ref{R:reduction-to-monomial-Q-ideals}). Let $\Upgamma$ denote the saturation of the submonoid of $\NN^n \oplus \ZZ$ generated by $\NN^{n+1}$ and $(\bv,-\ell)$ for vertices $\bv$ of $P_\fa$. Recall from Remark~\ref{R:same-newton-polyhedron-for-two-monomial-ideals}\eqref{R:snpftmi-1} that \begin{equation}\label{E:integral-closure-of-monomial-ideal}
    \overline{\fa} = \overline{(\pmb{x}^\bv \colon \bv \textrm{ vertex of } P_\fa)} = \overline{(\pmb{x}^{\bv_\sigma} \colon \sigma \textrm{ maximal cone of } \Sigma_\fa)}.
\end{equation}
Therefore, the assignment \[
    \be_i \mapsto \begin{cases}
        x_i \quad &\textrm{for } 1 \leq i \leq n, \\
        t^{-1} \quad &\textrm{for } i = n+1
    \end{cases}
\]
defines the following isomorphisms: \begin{equation}\label{E:logarithmic-structure-of-weighted-monoidal-blow-up}
    \begin{tikzcd}
    \kk[\NN^n \oplus \ZZ] \arrow[to=1-2, "\simeq"] & \sO_{\AA^n}[t^\pm] \\
    \kk[\Upgamma] \arrow[to=1-1, hookrightarrow] \arrow[to=2-2, "\simeq"] & \overline{\fa}_\bullet^\extd\rlap{.} \arrow[to=1-2, hookrightarrow]
    \end{tikzcd}
\end{equation}
The bottom row defines an isomorphism $\Spec(\kk[\Upgamma]) \xrightarrow{\lowsimeq} \GSpec_{\AA^n}(\overline{\fa}_\bullet^\extd)$. The former has the toroidal logarithmic structure $\Upgamma \inj \kk[\Upgamma]$ and hence defines a toroidal logarithmic structure on $\GSpec_{\AA^n}(\overline{\fa}_\bullet^\extd) \smallsetminus V(\overline{\fa}_+^\extd) \subset \GSpec_{\AA^n}(\overline{\fa}_\bullet^\extd)$. The toroidal logarithmic structure on $\tBl_{\fa_\bullet}\AA^n$ is then obtained by descent.
\end{xpar}

%\hfill

We can now state the main goal of this section. 

\begin{proposition}\label{P:canonicity-of-multi-weighted-blow-up}
  The object  $\Bl_{\fa_\bullet}\AA^n$ is the canonical smooth, toroidal Artin stack over $\tBl_{\fa_\bullet}\AA^n$.
\end{proposition}

The canonicity asserted in Proposition~\ref{P:canonicity-of-multi-weighted-blow-up} is in the sense of Satriano in \cite{satriano-canonical-artin-stacks-for-log-smooth-schemes}, which we will now recall in detail.

\begin{xpar}\label{X:satriano-construction}
  Given a toroidal $\kk$-scheme $Y$, Satriano demonstrates in \cite[Section 3]{satriano-canonical-artin-stacks-for-log-smooth-schemes} that there is a smooth, toroidal Artin stack $\sY$ over $Y$, which satisfies the following universal property. Any \emph{sliced resolution}, see \cite[Definition 2.6]{satriano-canonical-artin-stacks-for-log-smooth-schemes}, from a fine and saturated (fs)
  logarithmic scheme $(T,\sM_T)$ to $(Y,\sM_Y)$ factors uniquely as a strict morphism $(T,\sM_T) \to (\sY,\sM_{\sY})$ followed by $(\sY,\sM_{\sY}) \to (Y,\sM_Y)$. We call $\sY \to Y$ the \emph{canonical smooth, toroidal Artin stack over $Y$.}

In \cite[Proposition 3.1]{satriano-canonical-artin-stacks-for-log-smooth-schemes} Satriano gives the following local description of $\sY \to Y$. Let $Y = \Spec(\Upgamma \inj \kk[\Upgamma])$ for a sharp, toric monoid $\Upgamma$. Let $C(\Upgamma)$ denote the rational cone generated by $\Upgamma$ in $M_\RR := \Upgamma^\gp \otimes_\ZZ \RR$, and $C(\Upgamma)^\vee$ be the dual cone in $N_\RR := M_\RR^\vee$. For an extremal ray $\rho$ of $C(\Upgamma)^\vee$, we denote by $\bu_\rho$ the first lattice point on $\rho$. 

Let $F$ denote the free monoid on the set $S$ of extremal rays $\rho$ of $C(\Upgamma)^\vee$, and consider $\iota \colon \Upgamma \inj F$ defined by 
\begin{equation}\label{E:minimal-free-resolution}
    \bv \longmapsto (\bu_\rho \cdot \bv)_{\rho \in S} \quad \textrm{for } \bv \in \Upgamma \subset C(\Upgamma).
\end{equation}
Then $\iota$ is a \emph{minimal free resolution}, in the sense of \cite[Definition 2.3]{satriano-canonical-artin-stacks-for-log-smooth-schemes}. Setting $D(-) := \Hom_{\GrpSch}(-,\Gm)$, $\iota$ induces the morphism \[
    \left[\Spec(F \longinj \kk[F]) \q D(F^\gp / \Upgamma^\gp)\right] \lra \Spec(\Upgamma \inj \kk[\Upgamma]), 
\]
which is $\sY \to Y$.
\end{xpar}

\begin{remark}\label{R:satriano-construction}
By descent, Satriano's demonstration immediately generalizes for a toroidal Artin stack $Y$ over~$\kk$.  {We may appeal to descent because Satriano's construction \emph{commutes with strict, smooth morphisms}}. More precisely, given a strict morphism $f \colon \widetilde{Y} \to Y$ between toroidal Artin stacks, the canonical smooth, toroidal Artin stack over $\widetilde{Y}$ is the Cartesian product $\widetilde{Y} \times_Y \sY$ in the category of fs logarithmic Artin stacks, where $\sY \to Y$ is the canonical smooth, toroidal Artin stack over $Y$. This can be seen using the universal property in Section~\ref{X:satriano-construction}. {Indeed, it suffices to note that given any sliced resolution $g \colon T \to \widetilde{Y}$, the composition $g \circ f \colon T \to Y$ is still a sliced resolution because the induced morphism $\overline{\sM}_{Y,f(p)} \to \overline{\sM}_{\widetilde{Y},p}$ is an isomorphism for all geometric points $p$ of $\widetilde{Y}$. Thus, $g \circ f$ factors uniquely as $T \xrightarrow{\textrm{strict}} \sY \to Y$, and hence $g$ factors uniquely as $T \xrightarrow{\textrm{strict}} \sY \times_Y \widetilde{Y} \to \widetilde{Y}$.}
\end{remark}

In particular, Satriano's construction can be explicated for a toric Artin stack $Y$ arising from a \emph{stacky cone} $(\sigma,\beta)$; see \cite[Definition 2.4]{geraschenko-satriano-toric-stacks-I}. Here, $\beta$ is a homomorphism of lattices $N \to L$ with finite cokernel, and we assume $\sigma$ is a strongly convex, rational cone in $N_\RR := N \otimes_\ZZ \RR$. The dual morphism $\beta^\vee \colon L^\vee \to N^\vee$ is injective, and the dual cone $\sigma^\vee$ in $N_\RR^\vee$ yields the sharp, toric monoid $\Upgamma := \sigma^\vee \cap N^\vee$ and hence gives rise to the affine toric variety $\Spec(\Upgamma \inj \kk[\Upgamma])$. We then get the toric stack $Y := [\Spec(\Upgamma \inj \kk[\Upgamma]) \q G]$, where $G := D(\Coker(\beta^\vee))$ acts as a subgroup of the torus $T_N := D(N)$ (with $D(-) := \Hom_{\GrpSch}(-,\Gm)$). 

\begin{xpar}\label{X:satriano-construction-for-toric-stacks}
With the above notation, the canonical smooth, toroidal Artin stack $\sY$ over $Y$ can be constructed as follows. Let $F$ denote the monoid on the set $S$ of extremal rays $\rho$ of $\sigma = C(\Upgamma)^\vee$, and set $N_F^\vee := F^\gp$. The same rule as in \eqref{E:minimal-free-resolution} defines an embedding of lattices $\eta^\vee \colon N^\vee \inj N_F^\vee$, which restricts to a minimal free resolution $\iota \colon \Upgamma \inj F$ and fits in the commutative diagram \begin{equation}\label{E:satriano-construction-for-toric-stacks}
    \begin{tikzcd}
    0 \arrow[to=1-2] & L^\vee \arrow[to=2-2, equal] \arrow[to=1-3, hookrightarrow, "\beta_F^\vee"] & N_F^\vee \arrow[to=1-4] & \Coker\left(\beta_F^\vee\right) \arrow[to=1-5] & 0 \\
    0 \arrow[to=2-2] & L^\vee \arrow[to=2-3, hookrightarrow, "\beta^\vee"] & N^\vee \arrow[to=1-3, swap, hookrightarrow, "\eta^\vee"] \arrow[to=2-4] & \Coker(\beta^\vee) \arrow[to=1-4, hookrightarrow] \arrow[to=2-5] & 0\rlap{.}
    \end{tikzcd}
\end{equation}
The stacky cone $(\sigma_{\std},\beta_F)$, where $\sigma_{\std}$ is the standard cone on $N_F$ and $\beta_F \colon N_F \to L$ is the dual of $\beta_F^\vee$, then induces the corresponding smooth toric stack $\left[\Spec(F \inj \kk[F]) \q G_F\right]$, where $G_F := D(\Coker(\beta_F))$ acts as a subgroup of the torus $T_{N_F} := D(N_F)$ (with $D(-) := \Hom_{\GrpSch}(-,\Gm)$). Finally, the above commutative diagram induces the toric morphism \[
    \left[\Spec(F \longinj \kk[F]) \q G_F\right] \lra \left[\Spec(\Upgamma \longinj \kk[\Upgamma]) \q G\right], 
\]
which is $\sY \to Y$.
\end{xpar}

\begin{proof}[Proof of Proposition~\ref{P:canonicity-of-multi-weighted-blow-up}]
Without loss of generality (see Remark~\ref{R:reduction-to-monomial-Q-ideals}), we may replace $\fa_\bullet$ by $\overline{\fa}_\bullet$. Write $\fa_\bullet = \fa^{1/\ell}$ for a monomial ideal $\fa$ on $\AA^n$ and $\ell \in \NN_{>0}$. Our approach is to first carry out the construction in Section~\ref{X:satriano-construction-for-toric-stacks} for the toric Artin stack \[
    \fM := \left[\GSpec_{\AA^n}\left(\Upgamma \longinj \kk[\Upgamma] \xrightarrow[\eqref{E:logarithmic-structure-of-weighted-monoidal-blow-up}]{\lowsimeq} \fa_\bullet^\extd\right) \q \Gm\right]
\]
before doing the same for $\tBl_{\fa_\bullet}\AA^n$, which is a strict, open substack of $\fM$.

\begin{xpar}[Step 1]\label{X:step-1}
By definition, the toric Artin stack $\fM$ arises from the stacky cone $(\sigma,\beta)$, whose dual is given by \[
    \beta^\vee = (\one_{\ZZ^n},0) \colon \ZZ^n \longinj \ZZ^{n+1} \quad \textrm{and} \quad \sigma^\vee \cap \ZZ^{n+1} = \Upgamma; 
\]
\textit{i.e.} $\sigma := C(\Upgamma)^\vee \subset \RR^{n+1}$, and $\beta \colon \ZZ^{n+1} \to \ZZ^n$ is the projection onto the first $n$ factors. Next, the extremal rays of $\sigma$ are the normal rays to the facets of $C(\Upgamma)$, and so the set $\bS$ of their first lattice points is the disjoint union of \begin{enumerate}
    \item\label{X:s1-1} $\lbrace \be_i \colon 1 \leq i \leq n \rbrace$ and
    \item\label{X:s1-2} $\left\lbrace \widetilde{\bu}_\rho := \left(\frac{\ell}{\gcd(\ell,N_\rho(\fa))} \cdot \bu_\rho,\frac{N_\rho(\fa)}{\gcd(\ell,N_\rho(\fa))}\right) \colon \rho \in \Sigma_\fa(1) \textrm{ with } N_\rho(\fa) > 0 \right\rbrace$.
\end{enumerate}
Indeed, \eqref{X:s1-1} is evident since the coordinate hyperplanes $\be_i = 0$ ($1 \leq i \leq n$) intersect $C(\Upgamma)$ in facets. For \eqref{X:s1-2}, note that the intersection of $C(\Upgamma)$ with the hyperplane $\be_{n+1} = -\ell$ is canonically identified with $P_\fa$, and the (non-empty) intersection of every other facet of $C(\Upgamma)$ with this hyperplane $\be_{n+1} = -\ell$ corresponds to a unique facet $H_\rho$ of $P_\fa$ satisfying $N_\rho(\fa) > 0$.
\end{xpar}

\begin{xpar}[Step 2]\label{X:step-2}
Let us rewrite $\bS$ as the  disjoint union of the following sets: \begin{enumerate}
    \item $\bS_1 := \lbrace \be_i \colon 1 \leq i \leq n$ with $N_i(\fa) > 0 \rbrace$ and
    \item $\bS_2 := \left\lbrace \widetilde{\bu}_\rho := \left(\frac{\ell}{\gcd(\ell,N_\rho(\fa))} \cdot \bu_\rho,\frac{N_\rho(\fa)}{\gcd(\ell,N_\rho(\fa))}\right) \colon \rho \in \Sigma_\fa(1) \right\rbrace$.
\end{enumerate}
We take the indexing set of $\bS_2$ to be $\Sigma_\fa(1)$ and  denote the indexing set of $\bS_1$ by $I := \lbrace 1 \leq i \leq n \colon N_i(\fa) > 0 \rbrace$. 

By Section~\ref{X:satriano-construction-for-toric-stacks}, the canonical smooth, toroidal Artin stack $\sM$ over $\fM$ arises from the stacky cone $(\sigma_{\std},\beta_F)$, where $\sigma_{\std}$ is the standard cone on $\ZZ^I \oplus \ZZ^{\Sigma_\fa(1)}$, and the dual of $\beta_F$ fits in the following commutative diagram: \begin{equation}\label{E:step-2}
    \begin{tikzcd}
    0 \arrow[to=1-2] & \ZZ^n \arrow[to=2-2, equal] \arrow[to=1-3, hookrightarrow, "\beta_F^\vee"] & \ZZ^I \oplus \ZZ^{\Sigma_\fa(1)} \arrow[to=1-4] & \Coker\left(\beta_F^\vee\right) \arrow[to=1-5] & 0 \\
    0 \arrow[to=2-2] & \ZZ^n \arrow[to=2-3, hookrightarrow, "\beta^\vee"] & \ZZ^{n+1} \arrow[to=1-3, hookrightarrow, swap, "\eta^\vee"] \arrow[to=2-4] & \ZZ \arrow[to=2-5] \arrow[to=1-4, hookrightarrow] & 0\rlap{.}
    \end{tikzcd}
\end{equation}
Here the matrix of $\eta^\vee$ has rows given by $\be_i$ for $i \in I$ and $\widetilde{\bu}_\rho$ for $\rho \in \Sigma_\fa(1)$, and the matrix of $\beta_F^\vee$ is obtained by deleting the last column of the matrix of $\eta^\vee$. Recall that $\eta^\vee$ restricts to a minimal free resolution $\iota \colon \Upgamma \inj \NN^I \oplus \NN^{\Sigma_\fa(1)}$. Explicitly,  \[
    \sM = \left[\Spec\left(\NN^I \oplus \NN^{\Sigma_\fa(1)} \longinj \kk\left[\chi_i \colon i \in I\right]\left[x_\rho' \colon \rho \in \Sigma_\fa(1)\right]\right) \q D\left(\Coker\left(\beta_F^\vee\right)\right)\right]
\]
and $\sM \to \fM$ is induced by \begin{equation}\label{E:sM-to-fM}
    x_i \longmapsto \chi_i \cdot \prod_{\rho \in \Sigma_\fa(1)}{\left(x_\rho'\right)^{\frac{\ell}{\gcd(\ell,N_\rho(\fa))} \cdot \uu_{\rho,i}}} \quad \textrm{and} \quad t^{-1} \longmapsto \prod_{\rho \in \Sigma_\fa(1)}{\left(x_\rho'\right)^{\frac{N_\rho(\fa)}{\gcd(\ell,N_\rho(\fa))}}}, 
\end{equation}
where $\chi_i := 1$ whenever $i \in [1,n] \smallsetminus I$.
\end{xpar}

\begin{xpar}[Step 3]\label{X:step-3}
We show, in this step, the following strengthening of Proposition~\ref{P:canonicity-of-multi-weighted-blow-up}. 

\begin{proposition}\label{P:canonicity-of-multi-weighted-blow-up-revisited}
  Let $\fa_\bullet = \fa^{1/\ell}$ for some monomial ideal $\fa$ on $\AA^n$ and $\ell \in \NN_{>0}$. For every maximal cone $\sigma$ of $\Sigma_\fa$, the chart 
  $D_+(\sigma) \subset \Bl_{\fa_\bullet}\AA^n$ is the canonical smooth, toroidal Artin stack over $D_+(\pmb{x}^{\bv_\sigma} \cdot t^\ell) \subset \tBl_{\fa_\bullet}\AA^n$.
\end{proposition}

\begin{proof}[Proof of Proposition~\ref{P:canonicity-of-multi-weighted-blow-up-revisited}]
  Since $D_+(\pmb{x}^{\bv_\sigma} \cdot t^\ell)$ is a strict, open substack of $\fM$, Remark~\ref{R:satriano-construction} says that the canonical smooth, toroidal Artin stack over $D_+(\pmb{x}^{\bv_\sigma} \cdot t^\ell)$ is $\sM_\sigma := D_+(\pmb{x}^\bv_\sigma \cdot t^\ell) \times_\fM \sM \xhookrightarrow{\textrm{\tiny strict, open}} \sM$. To explicate $\sM_\sigma$, note that by \eqref{E:sM-to-fM}, we have 
  \[
    \pmb{x}^{\bv_\sigma} \cdot t^\ell \;\; \longmapsto \;\;\; \prod_{i \in I}{\chi_i^{\uv_{\sigma,i}}} \cdot \prod_{\rho \in \Sigma_\fa(1)}{\left(x_\rho'\right)^{\frac{\ell}{\gcd(\ell,N_\rho(\sigma))} \cdot \left(\sum_{i=1}^n{\uv_i \cdot \uu_{\rho,i}} - N_\rho(\fa)\right)}}, 
\]
where $\uv_{\sigma,i} \geq N_i(\fa) > 0$ for all $i \in I$ and $\sum_{i=1}^n{\uv_i \cdot \uu_{\rho,i}} - N_\rho(\fa) > 0$ if and only if $\bv_\sigma \notin H_\rho$, \textit{i.e.} $\rho \not\subset \sigma$ (see Section~\ref{X:conventions-2.1}). Therefore, $\sM_\sigma$ is equal to \[
    \left[\Spec\left(\NN^I \oplus \NN^{\Sigma_\fa(1)} \lra \kk\left[\chi_i^{\pm} \colon i \in I\right]\left[x_\rho' \colon \rho \in \Sigma_\fa(1)\right]\left[(x_\sigma')^{-1}\right]\right) \q D\left(\Coker\left(\beta_F^\vee\right)\right)\right].
\]
For every $i \in I$, note that the image of $\be_i$ in $\Coker(\beta_F^\vee)$ ($=$ the weight of $\chi_i$ under the $\Coker(\beta_F^\vee)$-grading) has infinite order. Therefore, by {Lemma~\ref{L:slicing}}, we have \[
    \sM_\sigma = \left[\Spec\left(\NN^{\Sigma_\fa(1)} \lra \kk\left[x_\rho' \colon \rho \in \Sigma_\fa(1)\right]\left[(x_\sigma')^{-1}\right]\right) \q D\left(\Coker\left(\widetilde{\beta}^\vee\right)\right)\right], 
\]
where $\widetilde{\beta}^\vee$ is the composition $\ZZ^n \xhookrightarrow{\beta_F^\vee} \ZZ^I \oplus \ZZ^{\Sigma_\fa(1)} \xtwoheadrightarrow{\textrm{projection}} \ZZ^{\Sigma_\fa(1)}$ and $D(\Coker(\widetilde{\beta}^\vee))$ acts as a subgroup of the torus $D(\ZZ^{\Sigma_\fa(1)}) = T_{\ZZ^{\Sigma_\fa(1)}}$. 

Since the matrix of $\widetilde{\beta}^\vee$ has rows given by $\frac{\ell}{\gcd(\ell,N_\rho(\fa))} \cdot \bu_\rho$ for $\rho \in \Sigma_\fa(1)$, it follows by definition that $\sM_\sigma = D_+(\sigma) \subset \Bl_{\fa_\bullet}\AA^n$ (this also means that their logarithmic structures coincide).
\end{proof}
\end{xpar}

Finally, as $\sigma$ varies over all maximal cones of $\Sigma_\fa$, the charts $D_+(\sigma)$ cover $\tBl_{\fa_\bullet}\AA^n$ and the charts $D_+(\pmb{x}^{\bv_\sigma} \cdot t^\ell)$ cover $\tBl_{\fa_\bullet}\AA^n$. Since Satriano's construction is canonical, this completes the proof of Proposition~\ref{P:canonicity-of-multi-weighted-blow-up}.
\end{proof}

\begin{remark}\label{R:factored-morphism-formula}
Let $\fa_\bullet = \fa^{1/\ell}$ be as before. Then the morphism \[
    \begin{tikzcd}
    \Bl_{\fa_\bullet}\AA^n = \left[\Spec\left(\kk[x_1',\dotsc,x_n']\left[x_\rho' \colon \rho \in \bE(\fa)\right]\right) \smallsetminus V\left(J_{\Sigma_\fa}\right)\q D\left(\Coker\left(\widetilde{\beta}^\vee\right)\right)\right] \arrow[to=2-1] \\
    \tBl_{\fa_\bullet}\AA^n = \left[\Spec\left(\fa_\bullet^\extd\right) \smallsetminus V\left(\fa_+^\extd\right) \q \Gm\right]
    \end{tikzcd}
\]
is induced by
\begin{align}\label{E:factored-morphism-formula}
    \begin{split}
    x_i &\longmapsto (x_i')^{\frac{\ell}{\gcd(\ell,N_\rho(\fa))}} \cdot \prod_{\rho \in \bE(\fa)}{(x_\rho')^{\frac{\ell}{\gcd(\ell,N_\rho(\fa))} \cdot \uu_{\rho,i}}} \quad \textrm{for } 1 \leq i \leq n,  \\
    t^{-1} &\longmapsto \prod_{\rho \in \Sigma_\fa(1)}{\left(x_\rho'\right)^{\frac{N_\rho(\fa)}{\gcd(\ell,N_\rho(\fa))}}}.
    \end{split}
\end{align}
\end{remark}

\begin{remark}\label{R:properties-of-factored-morphism}
The morphism $\Bl_{\fa_\bullet}\AA^n \to \tBl_{\fa_\bullet}\AA^n$ is evidently toric (in particular, logarithmically smooth) and birational. Since $\pi_{\fa_\bullet} \colon \Bl_{\fa_\bullet}\AA^n \to \AA^n$ is universally closed (see Remark~\ref{R:multi-weighted-blow-up-properties}) and $\varpi_{\fa_\bullet} \colon \tBl_{\fa_\bullet}\AA^n \to \AA^n$ is proper {(since its coarse moduli space is proper over $\AA^n$)}, we deduce that $\Bl_{\fa_\bullet}\AA^n \to \tBl_{\fa_\bullet}\AA^n$ is universally closed. Therefore, it is also surjective since it is dominant and closed. Finally, as a birational morphism, it is \emph{small}; \textit{i.e.} it has no exceptional divisors. This can be seen from Remark~\ref{R:factored-morphism-formula} or directly from the fact that $\tBl_{\fa_\bullet}\AA^n$ is normal, whence smooth in codimension $1$.
\end{remark}

\begin{remark}\label{R:cms-equals-gms}
If $\fa_\bullet$ is generated in degree $\ell$, the coarse moduli space of $\tBl_{\fa_\bullet}\AA^n$ is the schematic blow-up $\bl_{\overline{\fa}_\ell}\AA^n$ of $\AA^n$ along $\overline{\fa}_\ell$. %see \cite[Proposition 1.6.1(iii)]{quek-rydh-weighted-blow-up}. 
We claim that this coincides with the good moduli space $X_{\Sigma_{\fa_\ell}}$ of $\Bl_{\fa_\bullet}\AA^n$. The reader can check this computationally, but we propose a more direct approach. By \cite[Theorem 6.6]{alper-good-moduli-spaces}, there exists a unique morphism $\iota \colon X_{\Sigma_{\fa_\ell}} \to \bl_{\overline{\fa}_\ell}\AA^n$ making the following diagram commute: \[
    \begin{tikzcd}
    \Bl_{\fa_\bullet}\AA^n \arrow[to=2-1] \arrow[to=1-4, "\textrm{good moduli space}"] & & & X_{\Sigma_{\fa_\ell}} \arrow[to=2-4, dotted, "\iota", swap] \\
    \tBl_{\fa_\bullet}\AA^n \arrow[to=2-4, "\textrm{coarse moduli space}"] & & & \bl_{\overline{\fa}_\ell}\AA^n\rlap{.}
    \end{tikzcd}
\]
It remains to note that $\iota$ is a birational and integral morphism between normal schemes and hence an isomorphism; see \spcite{0AB1}. Indeed, $\iota$ is birational because $\Bl_{\fa_\bullet}\AA^n \to \tBl_{\fa_\bullet}\AA^n$ is an isomorphism above $\AA^n \smallsetminus V(\fa_1) \subset \AA^n$. To see that $\iota$ is integral, it suffices, by~\spcite{01WM}, to observe that: \begin{enumerate}
    \item $\iota$ is affine; indeed, for every vertex $\bv$ of $P_{\fa_\ell}$, the preimage of the coarse space of $D_+(\pmb{x}^\bv \cdot t^\ell) \subset \tBl_{\fa_\bullet}\AA^n$ is the good moduli space of $D_+(\bv) \subset \Bl_{\fa_\bullet}\AA^n$ (see Proposition~\ref{P:canonicity-of-multi-weighted-blow-up-revisited}), and they are both affine; 
    \item $\iota$ is universally closed since $\Bl_{\fa_\bullet}\AA^n \to \bl_{\overline{\fa}_\ell}\AA^n$ is universally closed by Remark~\ref{R:properties-of-factored-morphism} and \cite[Theorem 11.1.2(ii)]{olsson-algebraic-stacks}, and good moduli spaces remain surjective after any schematic base change; see \cite[Propositions 4.7(i) and 4.16(i)]{alper-good-moduli-spaces}.
\end{enumerate}
\end{remark}

\subsection{Canonicity of multi-weighted blow-ups, II}\label{3.2}

We consider in this section a slightly more general setting than in Section~\ref{3.1}. Given $0 \leq r \leq n$, we instead endow $\AA^n$ with the following logarithmic structure: \begin{align*}
    \AA^{n;r} := \Spec(\NN^r &\lra \kk[x_1,\dotsc,x_{n-r},\underline{x}_{n-r+1},\dotsc,\underline{x}_n]), \\
    \be_i &\longmapsto \underline{x}_{n-r+i}
\end{align*}
where we underline $\underline{x}_i$ for $i > n-r$ to emphasize that they are \emph{monomial coordinates} on $\AA^{n;r}$ (see Remark~\ref{R:log-coordinates}) and differentiate them from the \emph{ordinary coordinates} $x_1,\dotsc,x_{n-r}$. Note that the case $r=n$ was considered in Section~\ref{3.1}.

\begin{xpar}[Conventions]\label{X:conventions-3.2}
On $\AA^{n;r}$ we consider an ideal of the form \[
    \fj = \left(x_1^{a_1},\dotsc,x_k^{a_k},\fa\right), 
\]
where $0 \leq k \leq n-r$, $a_i \in \NN_{>0}$, and $\fa$ is a monomial ideal on $\AA^{n;r}$ in the sense of Section~\ref{X:monomiality}; \textit{i.e.} $\fa$ is generated by monomials in $\underline{x}_{n-r+1},\dotsc,\underline{x}_n$.

We set $\ell := \lcm(a_1,\dotsc,a_k)$ (this is $1$ if $k = 0$),
set $w_i := \frac{\ell}{a_i}$ for $1 \leq i \leq k$, and set $\underline{\fj}_\bullet := \fj^{1/\ell} := (x_1^{1/w_1},\dotsc,x_k^{1/w_k},\fa^{1/\ell})$; \textit{i.e.} the integral closure in $\sO_{\AA^n}[t]$ of the $\sO_{\AA^n}$-subalgebra generated by $\lbrace x_i \cdot t^{w_i} \colon 1 \leq i \leq k \rbrace$ and $\fa \cdot t^\ell$; see \cite[Notation 2.12]{quek-weighted-log-resolution}. Then $\underline{\fj}_\bullet$ is an integrally closed ($\NN$-graded) Rees algebra on $\AA^n$. Analogously to Definition~\ref{D:weighted-monoidal-blow-up}, we have the following. 
\end{xpar}

\begin{definition}\label{D:weighted-toroidal-blow-up}
The \emph{weighted toroidal blow-up} of $\AA^{n;r}$ along $\underline{\fj}_\bullet$ is the stack-theoretic Proj over $\AA^{n;r}$: \[
    \tBl_{\underline{\fj}_\bullet}\AA^{n;r} := \SProj_{\AA^n}\left(\underline{\fj}_\bullet\right) := \left[\GSpec_{\AA^n}\left(\underline{\fj}_\bullet\right) \smallsetminus V\left(\underline{\fj}_+\right) \q \Gm\right] \xrightarrow{\varpi_{\underline{\fj}_\bullet}} \AA^{n;r}, 
\]
where $\underline{\fj}_+$ is the ideal generated by the positive degrees of $\underline{\fj}_\bullet$ and the $\Gm$-action is induced by the $\ZZ$-grading on $\underline{\fj}_\bullet$. Analogously to Section~\ref{X:logarithmic-structure-of-weighted-monoidal-blow-up}, this is a toroidal Deligne--Mumford stack under the logarithmic structure dictated by that of $\AA^{n;r}$ and the exceptional divisor; see Section~\ref{X:logarithmic-structure-of-weighted-toroidal-blow-up} or \cite[Section 4]{quek-weighted-log-resolution}.
\end{definition}

\begin{xpar}[Charts of $\tBl_{\underline{\fj}_\bullet}\AA^{n;r}$]\label{X:charts-for-weighted-toroidal-blow-up}
By the same explanation as in Section~\ref{X:charts-for-weighted-monoidal-blow-up}, $\tBl_{\underline{\fj}_\bullet}\AA^{n;r}$ is covered by the following charts: \begin{enumerate}\itemsep.5em 
    \item $D_+\left(x_i \cdot t^{w_i}\right) = \left[\GSpec_{\AA^n}\left(\underline{\fj}_\bullet\left[(x_i \cdot t^{w_i})^{-1}\right]\right) \q \Gm\right]$ for $1 \leq i \leq k$ and
    \item $D_+\left(\underline{\pmb{x}}^{\bv} \cdot t^\ell\right) = \left[\GSpec_{\AA^n}\left(\underline{\fj}_\bullet\left[\left(\underline{\pmb{x}}^\bv \cdot t^\ell\right)^{-1}\right]\right) \q \Gm\right]$ for vertices $\bv$ of $P_\fa$.
\end{enumerate}
Similarly to before, $V(t^{-1})$ is the exceptional divisor on each chart.
\end{xpar}

\begin{xpar}[Logarithmic structure on $\tBl_{\underline{\fj}_\bullet}\AA^{n;r}$]\label{X:logarithmic-structure-of-weighted-toroidal-blow-up}
Let $\Upgamma$ denote the saturation of the submonoid of $\NN^n \oplus \ZZ$ generated by $\NN^{n+1}$, $(\be_i,-w_i)$ for $1 \leq i \leq k$, and $(\bv,-\ell)$ for vertices $\bv$ of $P_\fa$. Then the assignment \[
    \be_i \longmapsto \begin{cases}
       x_i \quad &\textrm{for } 1 \leq i \leq n, \\
        t^{-1} \quad &\textrm{for } i = r+1
    \end{cases}
\]
defines the following isomorphisms: \begin{equation}\label{E:logarithmic-structure-of-weighted-toroidal-blow-up}
    \begin{tikzcd}
    \kk[\NN^n \oplus \ZZ] \arrow[to=1-2, "\simeq"] & \sO_{\AA^n}[t^\pm] \\
    \kk[\Upgamma] \arrow[to=1-1, hookrightarrow] \arrow[to=2-2, "\simeq"] & \underline{\fj}_\bullet^\extd\rlap{,} \arrow[to=1-2, hookrightarrow]
    \end{tikzcd}
\end{equation}
where $\underline{\fj}_\bullet^\extd$ denotes the extended Rees algebra associated to $\underline{\fj}_\bullet$; see Remark~\ref{R:extended-rees-algebra}. Next, consider the following submonoids of $\Upgamma$: \begin{enumerate}
    \item $\Upgamma_1$ denotes the submonoid generated by $(\be_i,-w_i)$ for $1 \leq i \leq k$ and $(\be_i,0)$ for $k < i \leq n-r$.
    \item $\Upgamma_2$ denotes the saturation of the submonoid generated by $\NN^{r+1}$ and $(\bv,-\ell)$ for every vertex $\bv$ of $P_\fa$.
\end{enumerate}
Then $\Upgamma_2 \inj \kk[\Upgamma]$ defines, via \eqref{E:logarithmic-structure-of-weighted-toroidal-blow-up}, a logarithmic structure on $\GSpec_{\AA^n}(\underline{\fj}_\bullet^\extd) \smallsetminus V(\underline{\fj}_+^\extd) \subset \GSpec_{\AA^n}(\underline{\fj}_\bullet^\extd)$, which descends to the logarithmic structure on $\tBl_{\underline{\fj}_\bullet}\AA^{n;r} = \left[\GSpec_{\AA^n}(\underline{\fj}_\bullet^\extd) \smallsetminus V(\underline{\fj}_+^\extd) \smallsetminus \Gm\right]$. To see that $\tBl_{\underline{\fj}_\bullet}\AA^{n;r}$ is toroidal, it suffices to observe that $\Upgamma = \Upgamma_1 \oplus \Upgamma_2$ and moreover that $\Upgamma_1$ is a free monoid of rank $n-r$ whose free generators form part of a basis of $\Upgamma^\gp = \ZZ^{n+1}$.
\end{xpar}

\begin{remark}
Let us remark on our choice to work with $\underline{\fj}_\bullet^\extd$ instead of $\underline{\fj}_\bullet$ in Section~\ref{X:logarithmic-structure-of-weighted-toroidal-blow-up}. One can still define, in a similar fashion, a logarithmic structure on $\GSpec_{\AA^n}(\underline{\fj}_\bullet) \smallsetminus V(\underline{\fj}_+) \subset \GSpec_{\AA^n}(\underline{\fj}_\bullet)$ which descends to the same logarithmic structure on $\tBl_{\underline{\fj}_\bullet}\AA^{n;r}$. However, it is not evident, without passing to charts, that the logarithmic structure on $\GSpec_{\AA^n}(\underline{\fj}_\bullet) \smallsetminus V(\underline{\fj}_+)$ is toroidal. This is a second way in which $\underline{\fj}_\bullet^\extd$ ``globalizes'' $\underline{\fj}_\bullet$; see Remark~\ref{R:extended-rees-algebra}.
\end{remark}

\begin{xpar}[Logarithmic structure on $\Bl_{\fj}\AA^n$]\label{X:logarithmic-structure-of-multi-weighted-blow-up}
On the other hand, we may also consider $\Bl_{\fj}\AA^n$ as defined in Section~\ref{1.1}. For the next proposition, we endow $\Bl_{\fj}\AA^n$ with the toroidal logarithmic structure obtained by descent from the following toroidal logarithmic structure on $\AA^{\Sigma_\fj(1)} \smallsetminus V(J_{\Sigma_\fj})$: \[
    \NN^r \oplus \NN^{\bE(\fj)} \lra \kk[x_1',\dotsc,x_{n-r}',\underline{x}_{n-r+1}',\dotsc,\underline{x}_n']\left[\underline{x}_\rho' \colon \rho \in \bE(\fj)\right], 
\]
which sends $\be_i$ to $\underline{x}_{n-r+i}'$ for $1 \leq i \leq r$, and $\be_\rho$ to $\underline{x}_\rho'$ for $\rho \in \bE(\fj)$. We denote by $\Bl_{\fj}\AA^{n;r}$ the resulting logarithmic Artin stack. Using the language introduced in Section~\ref{X:satriano-construction}, we can now state the main objective of this section. 
\end{xpar}

\begin{proposition}\label{P:canonicity-of-multi-weighted-blow-up-II}
  The object   $\Bl_{\fj}\AA^{n;r}$ is the canonical smooth, toroidal Artin stack over $\tBl_{\underline{\fj}_\bullet}\AA^{n;r}$.
\end{proposition}

We prove this via Proposition~\ref{P:canonicity-of-multi-weighted-blow-up} and the following digression. 

\begin{xpar}\label{X:satriano-construction-for-toroidal-stacks}
We return to the discussion in Section~\ref{X:satriano-construction-for-toric-stacks}. Adopting the notation there, we suppose further that $N^\vee = N_1^\vee \oplus N_2^\vee$ for sublattices $N_i^\vee \subset N^\vee$, and hence $\Upgamma = \Upgamma_1 \oplus \Upgamma_2$ for the submonoids $\Upgamma_i := \Upgamma \cap N_i^\vee \subset \Upgamma$ such that $\Upgamma_1$ is a free monoid of finite rank satisfying $\Upgamma_1^\gp = N_1^\vee$; \textit{i.e.} its free generators form a basis of $N_1^\vee$. We redefine $Y$ as \[
    Y := \left[\Spec(\Upgamma_2 \longinj \kk[\Upgamma]) \q G\right], 
\]
which is a toroidal Artin stack by hypothesis. As before, our goal here is to explicate the canonical smooth, toroidal Artin stack $\sY$ over $Y$. 

Recall that in Section~\ref{X:satriano-construction-for-toric-stacks}, $F$ denotes the free monoid on the set $S$ of extremal rays $\rho$ of $\sigma$, $N_F^\vee$ denotes $F^\gp$, and we defined an embedding of lattices $\eta^\vee \colon N^\vee \inj N_F^\vee$ which restricts to a minimal free resolution $\iota \colon \Upgamma \inj F$. We \emph{\underline{claim}} that there exist free submonoids $F_i \subset F$ such that the following hold:  \begin{enumerate}
    \item\label{claim-1} $F = F_1 \oplus F_2$.
    \item\label{claim-2} Set $N_{F_i}^\vee := F_i^\gp$. Then $\eta^\vee \colon N^\vee \xrightarrow{\lowsimeq} N_F^\vee$ decomposes as $\eta_1^\vee \oplus \eta_2^\vee$, where $\eta_1^\vee = \eta^\vee|_{N_1^\vee} \colon N_1^\vee \xrightarrow{\lowsimeq} N_{F_1}^\vee$, which restricts to $\iota_1 \colon \Upgamma_1 \xrightarrow{\lowsimeq} F_1$, and $\eta_2^\vee = \eta^\vee|_{N_2^\vee} \colon N_2^\vee \inj N_{F_2}^\vee$, which restricts to a minimal free resolution $\iota_2 \colon \Upgamma_2 \inj F_2$. Moreover, $\iota = \iota_1 \oplus \iota_2$.
\end{enumerate}
Combining this claim with \eqref{E:satriano-construction-for-toric-stacks} yields the logarithmically smooth morphism \[
    \left[\Spec(F_2 \longinj \kk[F]) \q G_F\right] \lra \left[\Spec(\Upgamma_2 \longinj \kk[\Upgamma]) \q G\right] 
\]
and moreover shows that it is $\sY \to Y$. 

\begin{proof}[Proof of the \underline{claim}]
For $i=1,2$, let $\sigma_i$ denote the dual cone in $N$ of $C(\Upgamma_i) \subset N_i^\vee \subset N^\vee$, and let $\sigma_i'$ denote the dual cone in $N_i$ of $C(\Upgamma_i) \subset N_i^\vee$. Since $\Upgamma = \Upgamma_1 \oplus \Upgamma_2 \subset N_1^\vee \oplus N_2^\vee = N^\vee$, we have $\sigma = \sigma_1 \cap \sigma_2$ with \[
    \sigma_1 = \sigma_1' \oplus N_2^\vee \quad \textrm{and} \quad \sigma_2 = N_1^\vee \oplus \sigma_2'.
\]
Thus, we may decompose $S = S_1 \sqcup S_2$, where $S_i$ is the set of extremal rays of $\sigma_i$. For $i=1,2$, let $F_i$ denote the free monoid on $S_i$. Then part \eqref{claim-1} is immediate, while part \eqref{claim-2}  follows from the definition of $\eta^\vee$ (in Section~\ref{X:satriano-construction-for-toric-stacks}), together with the following pair of observations: \begin{enumerate}
    \item $\lbrace \bu_\rho \colon \rho \in S_1 \rbrace = \lbrace (\bu_{\overline{\rho}},0) \colon \overline{\rho}$ extremal ray of $\sigma_1' \rbrace$,
    \item $\lbrace \bu_\rho \colon \rho \in S_2 \rbrace = \lbrace (0,\bu_{\overline{\rho}}) \colon \overline{\rho}$ extremal ray of $\sigma_2' \rbrace$,
\end{enumerate}
where $\bu_{\overline{\rho}}$ denotes the first lattice point on $\overline{\rho}$.
\end{proof}
\end{xpar}

\begin{proof}[Proof of Proposition~\ref{P:canonicity-of-multi-weighted-blow-up-II}]
It suffices to assume $k \geq 1$, or else we are in the situation of Proposition~\ref{P:canonicity-of-multi-weighted-blow-up}. We may also assume $\fa \neq 0$, or else $\tBl_{\underline{\fj}_\bullet}\AA^n$ is already smooth over $\kk$ and is equal to $\Bl_\fj\AA^n$; see Example~\ref{EX:weighted-blow-ups}. Our approach is to first revisit Section~\ref{X:step-2} for the toric Artin stack \[
    \fM := \left[\GSpec_{\AA^n}\left(\Upgamma \longinj \kk[\Upgamma] \xrightarrow[\eqref{E:logarithmic-structure-of-weighted-toroidal-blow-up}]{\lowsimeq} \underline{\fj}_\bullet^\extd\right) \q \Gm\right]
\]
before using Section~\ref{X:satriano-construction-for-toroidal-stacks} to deduce the canonical smooth, toroidal Artin stack over \[
    \fM' := \left[\GSpec_{\AA^n}\left(\Upgamma_2 \longinj \kk[\Upgamma] \xrightarrow[\eqref{E:logarithmic-structure-of-weighted-toroidal-blow-up}]{\lowsimeq} \underline{\fj}_\bullet^\extd\right) \q \Gm\right], 
\]
which contains $\tBl_{\underline{\fj}_\bullet}\AA^{n;r}$ as a strict, open substack.

\begin{xpar}[Step 1]\label{X:II-step-1}
Before revisiting Section~\ref{X:step-2} for $\fM$, let us first establish the following lemma. 

\begin{lemma}\label{L:numerical-data-for-j}
Assume $k \geq 1$ and $\fa \neq 0$. For any $\rho \in \bE(\fj)$,  \begin{enumerate}
    \item\label{L:ndfj-1} $\ell$ divides $N_\rho(\fj)$; 
    \item\label{L:ndfj-2} the corresponding facet $H_\rho$ of\, $P_\fj$ contains the vertices $\lbrace a_i \cdot \be_i^\vee \colon 1 \leq i \leq k \rbrace$; in other words, $a_i \cdot \uu_{\rho,i} = N_\rho(\fj)$ for every $1 \leq i \leq k$; 
    \item\label{L:ndfj-3} $\uu_{\rho,i} = 0$ for every $k < i \leq n-r$.
\end{enumerate}
\end{lemma}

\begin{proof}
Let $\rho \in \bE(\fj)$. Let $H_\rho$ denote the corresponding facet of $P_\fj$, whose affine span is given by $\sum_{i=1}^n{\uu_{\rho,i} \cdot \be_i} = N_\rho(\fj)$. 

On one hand, note that $P_\fj \cap \lbrace \be_{n-r+1} = \dotsb = \be_n = 0 \rbrace$ is the Newton polyhedron $P_\fx$ of $\mathfrak{x} := (x_1^{a_1},\dotsc,x_k^{a_k}) \subset \kk[x_1,\dotsc,x_{n-r}]$ and that there is only one ray $\widetilde{\rho} \in \bE^+(\fx)$, whose corresponding facet $H_{\widetilde{\rho}}$ of $P_\fx$ has the affine span $\sum_{i=1}^k{\frac{\ell}{a_i} \cdot \be_i} = \ell$. 

On the other hand, $H_\rho \cap \lbrace \be_{n-r+1} = \dotsb = \be_n = 0 \rbrace$ is a facet of $P_\fx$ whose affine span is $\sum_{i=1}^{n-r}{\uu_{\rho,i} \cdot \be_i} = N_\rho(\fj)$. Since $\rho \in \bE(\fj)$, we must have $N_\rho(\fj) > 0$, so the facet of $P_\fx$ in the preceding sentence must be $H_{\widetilde{\rho}}$ in the preceding paragraph. By comparing equations and noting that $\gcd\left(\frac{\ell}{a_i} \colon 1 \leq i \leq k \right) { = 1}$, part~\eqref{L:ndfj-1} follows. Parts~\eqref{L:ndfj-2} and~\eqref{L:ndfj-3} are now also immediate.
\end{proof}

Let us note in addition that since $k \geq 1$ and $\fa \neq 0$, then $\bE^+(\fj) = \bE(\fj)$; \textit{i.e.}~$N_i(\fj) = 0$ for all $i \in [1,n]$. Therefore, combining Section~\ref{X:step-2} with the above observations, we see that the canonical smooth, toroidal Artin stack $\sM$ over $\fM$ arises from the stacky cone $(\sigma_{\std},\beta_F)$, where $\sigma_{\std}$ is the standard cone on $\ZZ^{\Sigma_\fj(1)}$ and the dual of $\beta_F$ fits in the following commutative diagram: \begin{equation}\label{E:II-step-1}
    \begin{tikzcd}
    0 \arrow[to=1-2] & \ZZ^n \arrow[to=2-2, equal] \arrow[to=1-3, hookrightarrow, "\beta_F^\vee"] & \ZZ^{\Sigma_\fj(1)} \arrow[to=1-4] & \Coker\left(\beta_F^\vee\right) \arrow[to=1-5] & 0 \\
    0 \arrow[to=2-2] & \ZZ^n \arrow[to=2-3, hookrightarrow, "\beta^\vee"] & \ZZ^{n+1} \arrow[to=1-3, hookrightarrow, swap, "\eta^\vee"] \arrow[to=2-4] & \ZZ \arrow[to=2-5] \arrow[to=1-4, hookrightarrow] & 0\rlap{.}
    \end{tikzcd}
\end{equation}
Here, the matrix of $\eta^\vee$ has rows given by $\left(\bu_\rho,\frac{N_\rho(\fj)}{\ell}\right)$ for $\rho \in \Sigma_\fj(1)$, and the matrix of $\beta_F^\vee$ has rows given by $\bu_\rho$ for $\rho \in \Sigma_\fj(1)$. Explicitly,  \[
    \sM = \left[\Spec\left(\NN^{\Sigma_\fj(1)} \lra \kk\left[\underline{x}_\rho' \colon \rho \in \Sigma_\fj(1)\right]\right) \q D\left(\Coker\left(\beta_F^\vee\right)\right)\right], 
\]
and $\sM \to \fM$ is induced by \begin{equation}\label{E:II-fM-to-sM}
    x_i \longmapsto x_i \cdot \prod_{\rho \in \bE(\fj)}{\left(x_\rho'\right)^{\uu_{\rho,i}}} \quad \text{and} \quad t^{-1} \longmapsto \prod_{\rho \in \bE(\fj)}{\left(x_\rho'\right)^{\frac{N_\rho(\fj)}{\ell}}}.
\end{equation}
\end{xpar}

\begin{xpar}[Step 2]\label{X:II-step-2}
By Lemma~\ref{L:numerical-data-for-j}\eqref{L:ndfj-2}, we have, for every $1 \leq i \leq k$,  \[
    \eta^\vee(\be_i,-w_i) \; = \;\be_i + \sum_{\rho \in \bE(\fj)}{\left(\uu_{\rho,i} - \frac{N_\rho(\fj)}{a_i}\right) \cdot \be_\rho} \; = \; \be_i;  
\]
\textit{i.e.} \eqref{E:II-fM-to-sM} maps $x_i \cdot t^{w_i}$ to $x_i'$. By part \eqref{L:ndfj-3} of the same lemma, we have, for every $k < i \leq n-r$, $\eta^\vee(\be_i,0) = \be_i$; \textit{i.e.} \eqref{E:II-fM-to-sM} maps $x_i$ to $x_i'$. Therefore, $\eta^\vee$ maps $\Upgamma_1$ isomorphically onto $\NN^{[1,n-r]} \subset \ZZ^{\Sigma_\fj(1)}$. 

On the other hand, it is plain that $\eta^\vee$ maps $\Upgamma_2$ into $\NN^{[n-r+1,n]} \oplus \NN^{\bE(\fj)} \subset \ZZ^{\Sigma_\fj(1)}$. By Section~\ref{X:satriano-construction-for-toroidal-stacks}, we know $\eta^\vee|_{\Upgamma_2}$ is a minimal free resolution of $\Upgamma_2$, and the canonical smooth, toroidal Artin stack $\sM'$ over $\fM'$ is the stack quotient of \[
\Spec\left(\NN^{[n-r+1,n]} \oplus \NN^{\bE(\fj)} \lra \kk[x_1',\dotsc,x_{n-r}',\underline{x}_{n-r+1}',\dotsc,\underline{x}_n']\left[\underline{x}_\rho' \colon \rho \in \bE(\fj)\right]\right) 
\]
by the action of $D(\Coker(\beta)^\vee) \subset D(\ZZ^{\Sigma_\fj(1)}) = T_{\ZZ^{\Sigma_\fj(1)}}$.
\end{xpar}

\begin{xpar}[Step 3]\label{X:II-step-3}
By Remark~\ref{R:satriano-construction}, the canonical smooth, toroidal Artin stack over $\tBl_{\underline{\fj}_\bullet}\AA^{n;r}$ is \[
    \tBl_{\underline{\fj}_\bullet}\AA^{n;r} \times_{\fM'} \sM' \xhookrightarrow{\textrm{\tiny strict, open}} \sM', 
\]
which is schematically identical\footnote{Schematically (\textit{i.e.} forgetting any logarithmic structures), the morphisms $\sM \to \fM$ and $\sM' \to \fM'$ are identical, and so are $\tBl_{\underline{\fj}_\bullet}\AA^{n;n}$ and $\tBl_{\underline{\fj}_\bullet}\AA^{n;r}$.} to $\tBl_{\underline{\fj}_\bullet}\AA^{n;n} \times_{\fM} \sM$, \textit{i.e.} the canonical smooth, toroidal Artin stack over $\tBl_{\underline{\fj}_\bullet}\AA^{n;n}$. By Proposition~\ref{P:canonicity-of-multi-weighted-blow-up} and Lemma~\ref{L:numerical-data-for-j}\eqref{L:ndfj-1}, the latter is $\Bl_\fj\AA^{n;n}$. Therefore, the former is $\Bl_\fj\AA^n$ with the logarithmic structure induced from that of $\sM'$, \textit{i.e.} $\Bl_\fj\AA^{n;r}$.
\end{xpar}
\end{proof}

\begin{remark}\label{R:II-properties-of-factored-morphism}
The morphism $\Bl_{\fj}\AA^{n;r} \to \tBl_{\underline{\fj}_\bullet}\AA^{n;r}$ is logarithmically smooth. Of course, it also satisfies all the schematic properties listed in Remarks~\ref{R:properties-of-factored-morphism} and~\ref{R:cms-equals-gms}.
\end{remark}

\begin{remark}\label{R:numerical-data-for-j}
It is possible to say more in Lemma~\ref{L:numerical-data-for-j}. First we note that $P_\fj \cap \lbrace \be_1 = \dotsb = \be_{n-r} = 0 \rbrace$ is the Newton polyhedron $P_\fa$ of $\fa \subset \kk[x_{n-r+1},\dotsc,x_n]$. Correspondingly, for $\rho \in \bE(\fj)$, $H_\rho \cap \lbrace \be_1 = \dotsb = \be_{n-r} = 0 \rbrace$ must be a facet $H_{\overline{\rho}}$ of $P_\fa$ (for some $\overline{\rho} \in \Sigma_\fa(1)$). Moreover, since $N_\rho(\fj) > 0$, we must have $N_{\overline{\rho}}(\fa) > 0$. Then $\rho \mapsto \overline{\rho}$ sets up a one-to-one correspondence $\bE(\fj) = \bE^+(\fj) \xrightarrow{\lowsimeq} \bE^+(\fa)$. Through this correspondence, Lemma~\ref{L:numerical-data-for-j} can be supplemented as follows: \begin{align*}
        \uu_{\rho,i} &= \frac{\ell}{\gcd(\ell,N_{\overline{\rho}}(\fa))} \cdot \uu_{\overline{\rho},i} \quad \textrm{for } n-r < i \leq n,  \\
        N_\rho(\fj) &= \frac{\ell N_{\overline{\rho}}(\fj)}{\gcd(\ell,N_{\overline{\rho}}(\fa))}.
\end{align*}
In particular, the number $\frac{N_\rho(\fj)}{\ell}$ in \eqref{E:II-fM-to-sM} is equal to $\frac{N_{\overline{\rho}}(\fa)}{\gcd(\ell,N_{\overline{\rho}}(\fa))}$.
\end{remark}

\begin{corollary}\label{C:proper-transform-of-first-coordinate-hyperplane}
Suppose that $a_1$ divides $\lcm(a_2,\dotsc,a_k)$ $($this is $1$ if\, $k=1)$. Let $\AA^{n-1;r} = V(x_1) \subset \AA^{n;r}$, and set $\fj_1 := \fj|_{V(x_1)} = (x_2^{a_2},\dotsc,x_k^{a_k},\fa) \subset \kk[x_2,\dotsc,x_{n-r},\underline{x}_{n-r+1},\underline{x}_n]$. Then the proper transform $V(x_1')$ of\, $V(x_1) \subset \AA^{n;r}$ under $\pi_\fj \colon \Bl_\fj\AA^{n;r} \to \AA^{n;r}$ is canonically identified with $\pi_{\fj_1} \colon \Bl_{\fj_1}\AA^{n-1;r}$ $\to \AA^{n-1;r}$.
\end{corollary}

\begin{proof}
We saw at the start of Section~\ref{X:II-step-2} that $V(x_1 \cdot t^{w_1}) \times_{\tBl_{\underline{\fj}_\bullet}\AA^{n;r}} \Bl_\fj\AA^{n;r} = V(x_1')$. Thus, by Remark~\ref{R:satriano-construction}, $V(x_1')$ is the canonical smooth, toroidal Artin stack over $V(x_1 \cdot t^{w_1})$. 

On the other hand, $V(x_1 \cdot t^{w_1})$, being the proper transform of $\AA^{n-1;r} = V(x_1) \subset \AA^{n;r}$ under $\tBl_{\underline{\fj}_\bullet}\AA^{n;r} \to \AA^{n;r}$, is equal to $\tBl_{\underline{\fj}_{1\bullet}}\AA^{n-1;r}$, where $\underline{\fj}_{1\bullet} = \underline{\fj}_\bullet|_{V(x_1)} = \fj_1^{1/\ell}$; see \cite[Lemma 4.6]{quek-weighted-log-resolution}. By hypothesis, $\ell = \lcm(a_2,\dotsc,a_k)$, whence Proposition~\ref{P:canonicity-of-multi-weighted-blow-up-II} implies that the canonical smooth, toroidal Artin stack over $V(x_1 \cdot t^{w_1})$ is also $\Bl_{\fj_1}\AA^{n-1;r}$. Combining this with the preceding paragraph, we deduce the corollary.
\end{proof}

To lift the hypothesis in Corollary~\ref{C:proper-transform-of-first-coordinate-hyperplane}, one needs to consider a natural extension of the discussion in this section, which we do not need for the purposes of this paper. Nevertheless we treat this briefly below. 

\begin{xpar}\label{X:more-multi-weighted-blow-ups}
Slightly more generally, for any $c \in \NN_{>0}$, we may consider $\underline{\fj}_\bullet^{1/c} := \fj^{1/c\ell} := (x_1^{1/cw_1},\dotsc,x_k^{1/cw_k},\fa^{1/c\ell})$, \textit{i.e.} the integral closure in $\sO_{\AA^n}[t]$ of the $\sO_{\AA^n}$-subalgebra generated by $\lbrace x_i \cdot t^{cw_i} \colon 1 \leq i \leq k \rbrace$ and $\fa \cdot t^{c\ell}$. Schematically, the \emph{multi-weighted blow-up of $\AA^{n;r}$ along $\underline{\fj}_\bullet^{1/c}$} is defined as \[
    \pi_{\underline{\fj}_\bullet^{1/c}} \colon \Bl_{\underline{\fj}_\bullet^{1/c}}\AA^n := \Bl_{\fj,\bb}\AA^n \xrightarrow{\pi_{\fj,\bb}} \AA^n, 
\]
where \[
    \bb := \left(\frac{c\ell}{\gcd\left(c\ell,N_\rho(\fj)\right)} \colon \rho \in \bE(\fj) \right) = \left(\frac{c}{\gcd\left(c,\frac{N_\rho(\fj)}{\ell}\right)} \colon \rho \in \bE(\fj) \right).
\]
The same toroidal logarithmic structure on $\AA^{\Sigma_\fj(1)} \smallsetminus V(J_{\Sigma_\fj})$ as in Section~\ref{X:logarithmic-structure-of-multi-weighted-blow-up} descends to a toroidal logarithmic structure on $\Bl_{\underline{\fj}_\bullet^{1/c}}\AA^n$. We denote by $\Bl_{\underline{\fj}_\bullet^{1/c}}\AA^{n;r}$ the resulting logarithmic Artin stack. If $c=1$, note that $\Bl_{\underline{\fj}_\bullet}\AA^{n;r} = \Bl_\fj\AA^{n;r}$. Then: \begin{enumerate}
    \item By the same method of proof as for Proposition~\ref{P:canonicity-of-multi-weighted-blow-up-II}, it is the canonical smooth, toroidal Artin stack over $\tBl_{\underline{\fj}_\bullet^{1/c}}\AA^{n;r} := \SProj_{\AA^n}\left(\underline{\fj}_\bullet^{1/c}\right)$.
    \item Corollary~\ref{C:proper-transform-of-first-coordinate-hyperplane} has the following natural generalization, with the same proof. Let $\AA^{n-1;r} = V(x_1) \subset \AA^{n;r}$ and $\fj_1 = \fj|_{V(x_1)}$, so that $\underline{\fj}_{1\bullet} = \fj_1^{1/\ell_1}$, where $\ell_1 := \lcm(a_2,\dotsc,a_k)$ ($:= 1$ if $k = 1$). Then the proper transform $V(x_1')$ of $V(x_1) \subset \AA^{n;r}$ under $\pi_{\underline{\fj}_\bullet^{1/c}} \colon \Bl_{\underline{\fj}_\bullet^{1/c}}\AA^{n;r} \to \AA^{n;r}$ is $\pi_{\underline{\fj}_{1\bullet}^{1/cc'}} \colon \Bl_{\underline{\fj}_{1\bullet}^{1/cc'}}\AA^{n-1;r} \to \AA^{n-1;r}$, where $c' := \frac{\ell}{\ell_1}$.
\end{enumerate}
\end{xpar}

\section{Logarithmic resolution via multi-weighted blow-ups}\label{C:log-resolution-via-multi-weighted-blow-ups}

\subsection{Recollections}\label{4.1}
We begin by briefly reviewing the relevant aspects from \cite{quek-weighted-log-resolution}. Unless otherwise stated, $Y$ denotes a strict toroidal $\kk$-scheme, see \cite[Definition B.6]{quek-weighted-log-resolution}, with logarithmic structure $\sM_Y \to \sO_Y$. Let $I \subset \sO_Y$ be a (coherent) ideal on $Y$. 

In Sections~\ref{X:log-stratification}--\ref{X:log-coordinates} below we recall some important notions associated to a strict toroidal $\kk$-scheme; see \cite[Appendix B]{quek-weighted-log-resolution} for details and references.

\begin{xpar}[Logarithmic stratification]\label{X:log-stratification}
For every $p \in Y$, set $\overline{\sM}_{Y,p} := \sM_{Y,p}/\sO_{Y,p}^\ast$. Then $\overline{\sM}_{Y,p}^\gp$ is a free abelian group of finite rank. We define the \emph{rank} of $\sM_Y$ at $p$ to be $r(p) := \rank(\overline{\sM}_{Y,p}^\gp)$. Then $r(p)$ is upper semi-continuous on $Y$; \textit{i.e.} more precisely, there exists a \emph{logarithmic stratification} $Y := \bigsqcup_{r \in \NN}{Y(r)}$, where for each $r \in \NN$,  \begin{enumerate}
    \item $Y(r)$ is a locally closed, pure-dimensional, smooth subscheme of $Y$, consisting precisely of points $p \in Y(r)$ such that $r(p) = r$; 
    \item the schematic closure $\overline{Y(r)}$  is $\bigsqcup_{r' \geq r}{Y(r')}$.
\end{enumerate}
Consider the open immersion $j \colon U := Y(0) \inj Y$. Then the logarithmic structure $\sM_Y \to \sO_Y$ can be recovered as the inclusion $\sM_Y := j_\ast(\sO_U^\ast) \cap \sO_Y \inj \sO_Y$. In particular, the logarithmic structure $\sM_Y \to \sO_Y$ is always injective. 

We define the \emph{toroidal divisor} of $Y$ to be the complement $D := Y \smallsetminus Y(0)$ with the reduced closed subscheme structure. Finally, for every $p \in Y$, the \emph{logarithmic stratum} at $p$ is defined as $\fs_p := Y(r(p))$.
\end{xpar}

\begin{xpar}[Logarithmic derivations]\label{X:log-derivations}
The $\sO_Y$-module $\Theta_Y^{\log}$ of \emph{logarithmic derivations} on $Y$ is the $\sO_Y$-submodule of $\Theta_Y$ ($:= \sO_Y$-module of usual derivations on $Y$) whose sections are derivations that preserve the ideal of $D$. 
\end{xpar}

\begin{xpar}[Logarithmic coordinates]\label{X:log-coordinates}
At every point $p \in Y$, there exists a local system of \emph{logarithmic coordinates} at $p$, which comprises of: \begin{enumerate}
    \item\label{X:lc-1} a local system of \emph{ordinary coordinates} $(x_1,x_2,\dotsc,x_{n-r})$ at $p$; \textit{i.e.} these descend to a local system of coordinates of $\fs_p$ at $p$;
    \item\label{X:lc-2} and a local \emph{chart} $\fc \colon M \to \sM_Y(U)$ which is \emph{neat} at $p$; \textit{i.e.} the composition $M \xrightarrow{\fc} \sM_Y(U) \to \sM_{Y,p} \to \overline{\sM}_{Y,p}$ is an isomorphism. 
\end{enumerate}
Moreover, setting $s := \codim_{\fs_p}\overline{\lbrace p \rbrace}$, it is customary to arrange for the first $s$ elements in the ordered tuple $(x_1,x_2,\dotsc,x_{n-r})$ to descend to a regular system of parameters of $\sO_{\fs_p,p}$, in which case $(x_1,x_2,\dotsc,x_s)$ is called a local system of \emph{ordinary parameters} at $p$. Formally, \eqref{X:lc-1} and \eqref{X:lc-2} induces an isomorphism \[
    \kappa(p)\left\llbracket x_1,x_2,\dotsc,x_s,M\right\rrbracket \xrightarrow{\;\lowsimeq\;} \widehat{\sO}_{Y,p}, 
\]
which sends $x_i$ to $x_i$ and $m \in M$ to the image of $\fc(m)$ under $\sM_Y \to \sO_Y$. 
\end{xpar}

\begin{remark}\label{R:log-coordinates}
Let us now explicate the above notions in the event that $Y$ is in addition a smooth $\kk$-scheme. (We refer the reader to \cite[Appendix B]{quek-weighted-log-resolution} for the general case.) In this case, the toroidal divisor $D$ is a simple normal crossings (snc) divisor. 

At every $p \in Y$, let $D$ be given locally on a neighbourhood $U$ of $p$ by $\prod_{i > n-r}{x_i} = 0$ for a local system of coordinates $(x_1,x_2,\dotsc,x_n)$ at $p$. Then $(x_1,x_2,\dotsc,x_{n-r})$ is a local system of ordinary coordinates at $p$, and $(x_{n-r+1},\dotsc,x_n)$ induces a local chart $\fc \colon \NN^r \to \sM_Y(U)$ that is neat at $p$. As in Section~\ref{3.2}, we usually underline the parameters $(\underline{x}_{n-r+1},\dotsc,\underline{x}_n)$ to emphasize that they are \emph{monomial parameters} at $p$. Finally, the stalk $\Theta_{Y,p}^{\log}$ is generated as a free $\sO_{Y,p}$-module by $\frac{\partial}{\partial x_i}$ for $1 \leq i \leq n-r$ and $x_i\frac{\partial}{\partial x_i}$ for $n-r < i \leq n$.
\end{remark}

\begin{xpar}[Logarithmic differential operators]\label{X:logarithmic-differential-operators}
For each $m \in \NN$, let $\sD_{Y,\log}^{\leq m}$ be the sheaf of \emph{logarithmic differential operators} on $Y$ of order at most $m$; \textit{i.e.} in characteristic zero, this is simply the $\sO_Y$-submodule of the total sheaf $\sD_Y = \sD_Y^\infty$ of differential operators generated by $\sO_Y$ and the images of $(\Theta_Y^{\log})^{\otimes i}$ for $1 \leq i \leq m$. The total sheaf of logarithmic differential operators on $Y$ is \[
    \sD_Y^{\log} = \sD_{Y,\log}^\infty := \bigcup_{m \in \NN}{\sD^{\leq m}_{Y,\log}}.
\]
We write $\sD_{Y,\log}^{\leq m}(I)$ (resp.\ $\sD_{Y,\log}^\infty(I)$) to denote the ideal on $Y$ generated by the image of $I \subset \sO_Y$ under the logarithmic differential operators in $\sD_{Y,\log}^{\leq m}$ (resp.\ $\sD_{Y,\log}^\infty$).
\end{xpar}

\begin{xpar}[Logarithmic order]\label{X:logarithmic-order}
The \emph{logarithmic order} of $I \subset \sO_Y$ at a point $p \in Y$ is \[
    \logord_p(I) := \min\left\lbrace n \in \NN \colon \sD^{\leq n}_{Y,\log}(I)_p = \sO_{Y,p} \right\rbrace \in \NN \sqcup \lbrace \infty \rbrace, 
\]
where we take $\min(\emptyset) := \infty$. We set $\max\logord(I) := \max_{p \in Y}{\logord_p(I)}$.
\end{xpar}

\begin{xpar}[Monomiality]\label{X:monomiality}
We say an ideal $Q \subset \sO_Y$ is \emph{monomial} if it is generated by the image of an ideal of $\sM_Y$ under $\sM_Y \to \sO_Y$. Equivalently, $Q$ is monomial if and only if $\sD_{Y,\log}^{\leq 1}(Q) = Q$. 

As in Section~\ref{3.1}, a \emph{monomial Rees algebra} $Q_\bullet$ is a finitely generated, quasi-coherent graded $\sO_Y$-subalgebra $\bigoplus_{m \in \NN}{Q_m \cdot t^m} \subset \sO_Y[t]$ such that each $Q_m \subset \sO_Y$ is a monomial ideal. 

The \emph{monomial saturation} $\sM_Y(I)$ of an ideal $I \subset \sO_Y$ is defined as the intersection of all monomial ideals containing $I$. Equivalently, $\sM_Y(I) = \sD_{Y,\log}^\infty(I)$. Note that $\sM_Y(I)$ is itself a monomial ideal on $Y$. Lastly, for $p \in Y$, $\logord_p(I) = \infty$ if and only if $p \in V(\sM_Y(I))$; \textit{i.e.} $\logord_p(I) < \infty$ if and only if $\sM_Y(I)_p = \sO_{Y,p}$.
\end{xpar}

\begin{xpar}[Maximal contact element]\label{X:maximal-contact}
If $1 \leq a := \max\logord(I) < \infty$, the (logarithmic) \emph{maximal contact ideal} of $I$ is $\MC(I) = \sD_{Y,\log}^{\leq a-1}(I)$. For a point $p \in Y$ such that $\logord_p(I) = a$, a (logarithmic) \emph{maximal contact element} of $I$ at $p$ is a section in $\MC(I)_p$ that is part of a system of ordinary parameters at $p$ or, equivalently, that has logarithmic order $1$ at $p$. Note that since $\CHAR(\kk) = 0$, we have $\sO_{Y,p} = \sD_{Y,\log}^{\leq a}(I)_p = \sD_{Y,\log}^{\leq 1}(\MC(I))_p$, so maximal contact elements always exist locally.
\end{xpar}

\begin{xpar}[Coefficient ideal]\label{X:coefficient-ideal}
If $1 \leq a := \max\logord(I) < \infty$, the (logarithmic) \emph{coefficient ideal} of an ideal $I \subset \sO_Y$ is the ideal \[
    \sC(I,a) := \left(\prod_{j=0}^{a-1}{\sD_{Y,\log}^{\leq j}(I)^{c_j}} \colon c_j \in \NN,\; \sum_{j=0}^{a-1}{(a-j)c_j \geq a!}\right) \subset \sO_Y.
\]
Note that $\MC(I)^{a!} \subset \sC(I,a)$, so for any point $p \in Y$ with $\logord_p(I) = a$ and any maximal contact element $x$ of $I$ at $p$, we have $x^{a!} \in \sC(I,a)_p$. See \cite[Section 5]{quek-weighted-log-resolution} and the references therein for the motivation and properties of this construction.
\end{xpar}

\begin{xpar}[A well-ordered set]\label{X:well-ordered-set}
For $k \geq 1$, we define \begin{align*}
    \NN_{>0}^{k,!} &:= \text{\small$\left\lbrace (b_i)_{i=1}^k \in \NN_{>0}^k \colon b_1 \leq \frac{b_2}{(b_1-1)!} \leq \frac{b_3}{\prod_{j=1}^2{(b_j-1)!}} \leq \dotsb \leq \frac{b_k}{\prod_{j=1}^{k-1}{(b_j-1)!}} \right\rbrace$}, \\
    \NN_\infty^{k,!} &:= \left(\NN_{>0}^{k-1,!} \times \lbrace \infty \rbrace\right) \sqcup \NN_{>0}^{k,!},
\end{align*}
and for $s \geq 1$, we set \begin{align*}
    \NN_{>0}^{\leq s,!} &:= \left\lbrace (0) \right\rbrace \sqcup \left(\bigsqcup_{k=0}^s{\NN_{>0}^{k,!}}\right), \\
    \NN_\infty^{\leq s,!} &:= \left\lbrace (0) \right\rbrace \sqcup \left(\bigsqcup_{k=0}^s{\NN_\infty^{k,!}}\right).
\end{align*}
We well-order the set $\NN_\infty^{\leq s,!}$ by the lexicographic order $<$, with a caveat: our lexicographic order considers truncations of sequences to be strictly larger. For example, in $\NN_\infty^{\leq 3,!}$, we have \[
    (0) < (1,2,8) < (1,3,6) < (1,3) < (1,4,24) < (1,\infty) < (1) < (\infty) < ().
\]
\end{xpar}

We can now recall the key construction in \cite[Section 6.1]{quek-weighted-log-resolution}. 

\begin{xpar}[Invariant of an ideal at a point]\label{X:log-invariant}
  Let $p \in Y$, and set $s := \codim_{\fs_p}\overline{\lbrace p \rbrace}$ (see Section~\ref{X:log-coordinates}). In Section~\ref{X:associated-data} below, we will associate to the pair $(I,p)$ the following:
  \begin{enumerate}[label=(\alph*),ref=\alph*]
    \item\label{4.1.10.a} a sequence of ordinary parameters $(x_1,x_2,\dotsc,x_k)$ at $p$,
    \item\label{4.1.10.b} a monomial ideal $Q \subset \sO_{Y,p}$,
    \item\label{4.1.10.c} a sequence $(b_1,b_2,\dotsc,b_k) \in \NN_{>0}^{\leq s,!}$.
\end{enumerate}
Then the \emph{invariant} of $I$ at $p$ is a non-decreasing sequence of length at most $\min\{s+1,\dim(Y)\}$: \[
    \inv_p(I) := \begin{cases}
        \left(b_1,\frac{b_2}{(b_1-1)!},\frac{b_3}{\prod_{j=1}^2{(b_j-1)!}},\dotsc,\frac{b_k}{\prod_{j=1}^{k-1}{(b_j-1)!}}\right) \quad &\textrm{if } Q = 0, \\
        \left(b_1,\frac{b_2}{(b_1-1)!},\frac{b_3}{\prod_{j=1}^2{(b_j-1)!}},\dotsc,\frac{b_k}{\prod_{j=1}^{k-1}{(b_j-1)!}},\infty\right) \quad &\textrm{if } Q \neq 0.
    \end{cases}
\]
We denote the finite entries of $\inv_p(I)$ by $a_1,a_2,a_3,\dotsc,a_k$. Note that the set of all possible invariants of ideals at points in $Y$ can be well ordered by the same lexicographic order as in Section~\ref{X:well-ordered-set} since it is order-isomorphic to $\NN^{\leq \dim(Y),!}_\infty$. We set $\max\inv(I) := \max_{p \in Y}{\inv_p(I)}$.
\end{xpar}

\begin{xpar}\label{X:associated-data}
The data in Section~\ref{X:log-invariant}\eqref{4.1.10.a}--\eqref{4.1.10.c} associated to $(I,p)$ is defined inductively as follows:
\begin{enumerate}
    \item\label{X:ad-1} \emph{Base case}: If $\logord_p(I) = 0$, then $p \notin V(I)$; we set $k=1$, with $(b_1,\dotsc,b_k) := (0)$, $(x_1,\dotsc,x_k) := (x_1)$ for any ordinary parameter $x_1$ at $p$, and $Q = 0$. If $\logord_p(I) = \infty$, set $k=0$, with $(b_1,\dotsc,b_k) := ()$, $(x_1,\dotsc,x_k) := ()$, and $Q := \sM_Y(I)_p$. 
    
    If $\logord_p(I)$ is neither, set $b_1 := \logord_{p}(I) \in \NN_{>0}$, let $x_1$ be a maximal contact element of $I$ at $p$, set $I[1] := I$, and proceed to the $1^{\textrm{st}}$ \emph{inductive step}.
    
    \item\label{X:ad-2} $\ell^{\textrm{th}}$ \emph{inductive step}: Suppose by the induction hypothesis that the sequences $(b_1,\dotsc,b_\ell)$ and $(x_1,\dotsc,x_\ell)$, as well as the ideal $I[\ell] \subset \sO_{V(x_1,\dotsc,x_{\ell-1})}$, have been defined. Set \[
        \qquad \qquad I[\ell+1] := \sC(I[\ell],b_\ell)|_{V(x_1,\dotsc,x_\ell)} \subset \sO_{V(x_1,\dotsc,x_\ell)}.
    \]
    If $\logord_p(I[\ell+1]) = \infty$, set $k = \ell$, and define $Q$ to be the monomial ideal of $\sO_{Y,p}$ that lifts the monomial ideal $\sM_{V(x_1,\dotsc,x_\ell)}(I[\ell+1])_p \subset \sO_{V(x_1,\dotsc,x_\ell),p}$. 
    
    If not, set $b_{\ell+1} := \logord_p(I[\ell+1]) \in \NN_{>0}$, let $x_{\ell+1}$ be a lift of $\overline{x}_{\ell+1}$ to $\sO_{Y,p}$, where $\overline{x}_{\ell+1}$ is a maximal contact element of $I[\ell+1]$ at $p$, and proceed to the $(\ell+1)^{\textrm{st}}$ \emph{inductive step}.
\end{enumerate}
\end{xpar}

\begin{remark}\label{R:log-invariant}
The invariant $\inv$ satisfies the following properties:
\begin{enumerate}
    \item\label{R:li-1} Both $\inv_p(I)$ and the monomial ideal $Q$ are well defined, \textit{i.e.} independent of the choice of $(x_1,x_2,\dotsc,x_k)$; see \cite[Lemma 6.1]{quek-weighted-log-resolution}.
    \item \label{R:li-2} $\inv_p(I)$ is upper semi-continuous on $Y$; see \cite[Lemma 6.3(ii)]{quek-weighted-log-resolution}.
    \item\label{R:li-3} If $f \colon \widetilde{Y} \to Y$ is a logarithmically smooth morphism between strict toroidal $\kk$-schemes and $f$ maps $\widetilde{p} \in \widetilde{Y}$ to $p \in Y$, then $\inv_{\widetilde{p}}(f^{-1}I \cdot \sO_{\widetilde{Y}}) = \inv_p(I)$; see \cite[Lemma 6.3(iii)]{quek-weighted-log-resolution}. If $f$ is moreover surjective, then $\max\inv(f^{-1}I \cdot \sO_{\widetilde{Y}}) = \max\inv(I)$.
    \item\label{R:li-4} Finally, if $V(I)$ inherits its logarithmic structure from $Y$, then $V(I)$ is smooth and toroidal at $p$ if and only if $\inv_p(I) = (1,1,\dotsc,1)$ of length equal to the height of $I_p$; see \cite[Lemma 5.1.2]{abramovich-temkin-wlodarczyk-toroidal-destackification-kummer-blow-ups}.
\end{enumerate}
\end{remark}

\begin{xpar}[Local description of associated center]\label{X:center}
Let $p \in Y$ be such that $\inv_p(I) = \max\inv(I)$, and let $(x_1,x_2,\dotsc,x_k)$, $(b_1,b_2,\dotsc,b_k)$, $(a_1,a_2,\dotsc,a_k)$ and $Q$ be associated to $(I,p)$ as in Section~\ref{X:log-invariant}\eqref{4.1.10.a}--\eqref{4.1.10.c}. 

Then the \emph{center associated to $I$ at $p$} is the following integrally closed ($\NN$-graded) Rees algebra: \[
    \sJ(I,p)_\bullet := \left(x_1^{a_1},x_2^{a_2},\dotsc,x_k^{a_k},Q^{1/d}\right) \; \subset \; \sO_{Y,p}[t], 
\]
where \[
    d := \prod_{i=1}^k{(b_i-1)!} \; \textrm{ ($:= 1$ if $k = 0$)}
\]
and by convention, $x_1^{a_1} := 1$ if $k=1$ and $a_1 = 0$. Writing each $a_i = \frac{u_i}{v_i}$ for $u_i,v_i \in \NN$, $\left(x_1^{a_1},x_2^{a_2},\dotsc,x_k^{a_k},Q^{1/d}\right)$ refers to the integral closure in $\sO_{Y,p}[t]$ of the $\sO_{Y,p}$-subalgebra generated by $\lbrace x_i^{u_i} \cdot t^{v_i} \colon 1 \leq i \leq k \rbrace$ and $Q \cdot t^d := \lbrace \alpha \cdot t^d \colon \alpha \in Q \rbrace$. This definition is independent of the choices of $u_i$ and $v_i$ since we are passing to the integral closure in $\sO_{Y,p}[t]$.
\end{xpar}

\begin{xpar}[Convention]\label{X:convention-4.1}
The notation above follows \cite[Notation 2.12]{quek-weighted-log-resolution}. For the remainder of this paper, we will adopt the notation therein without further mention. Namely, if $J_\bullet \subset \sO_Y[t]$ is an integrally closed ($\NN$-graded) Rees algebra on a scheme $Y$ and $q = \frac{u}{v} \in \QQ_{>0}$ (for $u,v \in \NN_{>0}$), then $J_\bullet^q$ is the integral closure in $\sO_Y[t]$ of the Rees algebra $J_{\lceil u\bullet/v \rceil} = \bigoplus_{m \in \NN}{J_{\lceil um/v \rceil} \cdot t^m}$; \textit{i.e.} the $v$\textsuperscript{th} Veronese subalgebra of $J_\bullet^q$ is $J_{u\bullet}$. %, see \cite[Example 3.1.9]{quek-rydh-weighted-blow-up} 
Explicitly, if $J_\bullet = \left(g_1^{a_1},g_2^{a_2},\dotsc,g_k^{a_k}\right)$ $\subset \sO_Y[t]$ for global sections $g_i$ of $\sO_Y$ and $a_i \in \QQ_{>0}$, then $J_\bullet^q$ is $\left(g_1^{a_1q},g_2^{a_2q}\dotsc,g_k^{a_kq}\right)$.
\end{xpar}

\begin{xpar}\label{X:dth-Veronese-subalgebra}
The $d^{\textrm{th}}$ Veronese subalgebra $\sJ(I,p)_{d\bullet}$ is generated in degree $1$. For later purposes, we set $J(I,p) := \sJ(I,p)_d =$ the integral closure of $\left(x_1^{a_1d},x_2^{a_2d},\dotsc,x_k^{a_kd},Q\right)$.
\end{xpar}

\begin{remark}\label{R:center}
  The object   $\sJ(I,p)_\bullet$ does not depend on the choice of sequence of ordinary parameters $(x_1,x_2,\dotsc,x_k)$ associated to $I$ at $p$ and is \emph{$I_p$-admissible} -- that is, it contains the Rees algebra of $I_p$. In fact, a ``\emph{unique admissibility property}'' holds for $\sJ(I,p)_\bullet$. For details, see \cite[Theorem 6.5]{quek-weighted-log-resolution}.
\end{remark}

\begin{xpar}[Associated center]\label{X:associated-center}
By \cite[Theorem 6.9]{quek-weighted-log-resolution}, there is an ($\NN$-graded) Rees algebra $\sJ(I)_\bullet \subset \sO_Y[t]$ on $Y$ such that for every $p \in Y$, we have \[
    \left(\sJ(I)_\bullet\right)_p = \begin{cases}
        \sJ(I,p)_\bullet \quad &\textrm{if } \inv_p(I) = \max\inv(I), \\
        \sO_{Y,p}[t] \quad &\textrm{if } \inv_p(I) < \max\inv(I).
    \end{cases}
\]
We call $\sJ(I)_\bullet$ the \emph{center associated to $I$}. Then: \begin{enumerate}
    \item $\sJ(I)_\bullet$ is \emph{$I$-admissible}; \textit{i.e.} it contains the Rees algebra of $I$.
    \item Setting $d := \prod_{i=1}^k{(b_i-1)!}$ as in Section~\ref{X:center}, the $d^{\textrm{th}}$ Veronese subalgebra $\sJ(I)_{d\bullet}$ is generated in degree $1$. For later purposes, we set \[
        J(I) := \sJ(I)_d.
    \]
    \item Finally, if $f \colon \widetilde{Y} \to Y$ is a logarithmically smooth, surjective morphism between strict toroidal $\kk$-schemes, then $\sJ(f^{-1}I \cdot \sO_{\widetilde{Y}})_\bullet = \sJ(I)_\bullet$; see \cite[Lemma 6.12]{quek-weighted-log-resolution}.
\end{enumerate} 
\end{xpar}

\begin{xpar}[Associated reduced center]\label{X:reduced-center}
Let $(a_1,a_2,\dotsc,a_k)$ and $d$ be as in Section~\ref{X:center}. Setting \[
    \ell := \lcm(a_id \colon 1 \leq i \leq k) \; \textrm{ $(:= 1$ if $k=0)$}, 
\]
the \emph{reduced center associated to $I$} is \[
    \underline{J}(I)_\bullet := J(I)^{1/\ell} = \sJ(I)_\bullet^{d/\ell}. 
\]
For $p \in Y$ such that $\inv_p(I) = \max\inv(I)$, the stalk of $\underline{J}(I)_\bullet$ at $p$ is \[
    \underline{J}(I,p)_\bullet := J(I,p)^{1/\ell} = \sJ(I,p)_\bullet^{d/\ell} = \left(x_1^{1/w_1},x_2^{1/w_2},\dotsc,x_k^{1/w_k},Q^{1/\ell}\right)
\]
where $(x_1,x_2,\dotsc,x_k)$ and $Q$ are associated to $(I,p)$ as in Section~\ref{X:log-invariant}(a)--(c), $w_i := \frac{\ell}{a_id}$ for $1 \leq i \leq k$, and $\gcd(w_i \colon 1 \leq i \leq k) = 1$; see \cite[Definition 3.4]{quek-weighted-log-resolution}. 
\end{xpar}

\subsection{Multi-weighted blow-up along the associated center}\label{4.2}

In this section let $Y$ be as in Section~\ref{4.1}, and in addition assume $Y$ is smooth over $\kk$. Motivated by Section~\ref{3.2}, we make the following definition. 

\begin{definition}\label{D:multi-weighted-blow-up-along-center}
The \emph{multi-weighted blow-up $\Bl_{J(I)}Y$ of $Y$ along $\underline{J}(I)_\bullet$}  is the composition \[
    \pi_{J(I)} \colon \Bl_{J(I)}Y \lra \tBl_{\underline{J}(I)_\bullet}Y \xrightarrow{\varpi_{\underline{J}(I)_\bullet}} Y,
\]
where \[
    \tBl_{\underline{J}(I)_\bullet}Y := \SProj_Y\left(\underline{J}(I)_\bullet\right) = \left[\GSpec_Y\left(\underline{J}(I)_\bullet\right) \smallsetminus V\left(\underline{J}(I)_+\right) \q \Gm\right] \xrightarrow{\varpi_{\underline{J}(I)_\bullet}} Y
\]
is the weighted toroidal blow-up of $Y$ along $\underline{J}(I)_\bullet$, see \cite[Section 4]{quek-weighted-log-resolution} and \[
    \zeta \colon \Bl_{J(I)}Y \lra \tBl_{\underline{J}(I)_\bullet}Y
\]
is the canonical smooth, toroidal Artin stack, in the sense of Section~\ref{X:satriano-construction} (\textit{i.e.} \cite[Section 3]{satriano-canonical-artin-stacks-for-log-smooth-schemes}), over the toroidal Deligne--Mumford stack $\tBl_{\underline{J}(I)_\bullet}Y$.
\end{definition}

\begin{remark}\label{R:isomorphism-above-most-singular-locus}
Since $\varpi_{\underline{J}(I)_\bullet}$ is an isomorphism away from the closed locus of points $p \in Y$ such that $\inv_p(I) = \max\inv(I)$, the same holds for $\pi_{J(I)}$.
\end{remark}

While $\tBl_{\underline{J}(I)_\bullet}Y$ is a global quotient stack, we warn that $\Bl_{J(I)}Y$ is typically not. Nevertheless, $\Bl_{J(I)}Y$ is locally a quotient stack; see Sections~\ref{X:local-description} and \ref{X:rees-algebra-description} below.

\begin{xpar}[Local description of the multi-weighted blow-up in Definition~\ref{D:multi-weighted-blow-up-along-center}]\label{X:local-description}
Fix $p \in Y$ such that $\inv_p(I) = \max\inv(I)$. Extend a sequence of ordinary parameters $(x_1,x_2,\dotsc,x_k)$
associated to $I$ at $p$ (see Section~\ref{X:log-invariant}) to a system of ordinary coordinates $(x_1,x_2,\dotsc,x_{n-r})$ at $p$ (with $n-r \geq k$) and a system of mononial parameters $(\underline{x}_{n-r+1},\dotsc,\underline{x}_n)$ at $p$ (see Section~\ref{X:log-coordinates}). This local system of logarithmic coordinates at $p$ then induces an \'etale, strict morphism \[
    \pmb{x} = (x_1,x_2,\dotsc,x_{n-r},\underline{x}_{n-r+1},\dotsc,\underline{x}_n) \colon U \lra \AA^{n;r}
\]
for some Zariski open $U$ containing $p$ such that the stalk $\underline{J}(I,p)_\bullet$ extends over $U$ to $\underline{J}(I)_\bullet|_U$. Write \begin{align*}
    J(I,p) &= \textrm{integral closure of } \left(x_1^{a_1d},\dotsc,x_k^{a_kd},Q\right), \\
    \underline{J}(I,p)_\bullet &= \left(x_1^{1/w_1},x_2^{1/w_2},\dotsc,x_k^{1/w_k},Q^{1/\ell}\right)
\end{align*}
as in Sections~\ref{X:dth-Veronese-subalgebra} and \ref{X:reduced-center}, and set \begin{align*}
    \fj &:= J(I,p) \cap \sO_{\AA^{n;r}} = \textrm{integral closure of } \left(x_1^{a_1d},\dotsc,x_k^{a_kd},\fa\right), \\
    \underline{\fj}_\bullet &:= \underline{J}(I,p)_\bullet \cap \sO_{\AA^{n;r}}[t] = \left(x_1^{1/w_1},x_2^{1/w_2},\dotsc,x_k^{1/w_k},\fa^{1/\ell}\right),
\end{align*}
where $\fa := Q \cap \kk[\underline{x}_{n-r+1},\dotsc,\underline{x}_n]$ is a monomial ideal on $\AA^{n;r}$ that generates $Q$. Moreover, $\underline{\fj}_\bullet = \fj^{1/\ell}$ with $\ell = \lcm(a_id \colon 1 \leq i \leq k)$ ($:= 1$ if $k = 0$); see Section~\ref{3.2}. Then we have the commutative diagram with Cartesian squares \begin{equation}\label{E:local-description}
    \begin{tikzcd}
        \Bl_{J(I)}U \arrow[to=2-1] \arrow[to=1-3, "\textrm{\'etale, strict}"] \arrow[to=3-1, bend right=70, swap, "\pi_{J(I)}"] & &  \Bl_\fj\AA^{n;r} \arrow[to=2-3] \arrow[to=3-3, bend left=70, "\pi_\fj"] \\
        \tBl_{\underline{J}(I)_\bullet}U \arrow[to=3-1, "\varpi_{\underline{J}(I)_\bullet}"] \arrow[to=2-3, "\textrm{\'etale, strict}"] & & \tBl_{\underline{\fj}_\bullet}\AA^{n;r} \arrow[to=3-3, "\varpi_{\underline{\fj}_\bullet}"] \\
        U \arrow[to=3-3, "\pmb{x}"] & & \AA^{n;r}
    \end{tikzcd}
\end{equation}
Let us explicate the diagram. Firstly, since $\pmb{x}^\ast\underline{\fj}_\bullet = \underline{J}(I)_\bullet$, we have \[
    \tBl_{\underline{J}(I)_\bullet}U = U \times_{\AA^{n;r}} \tBl_{\underline{\fj}_\bullet}\AA^{n;r}.
\]
Next, by Proposition~\ref{P:canonicity-of-multi-weighted-blow-up-II}, $\Bl_\fj\AA^{n;r}$ is the canonical smooth, toroidal Artin stack over $\tBl_{\underline{\fj}_\bullet}\AA^{n;r}$. Therefore, by Remark~\ref{R:satriano-construction}, \[
    \Bl_{J(I)}U = \tBl_{\underline{J}(I)_\bullet}U \times_{\tBl_{\underline{\fj}_\bullet}\AA^{n;r}} \Bl_\fj\AA^{n;r}.
\]
Finally, we make a few remarks:
\begin{enumerate}
    \item\label{rem1} Since $\Bl_\fj\AA^{n;r} \to \tBl_{\underline{\fj}_\bullet}\AA^{n;r}$ is logarithmically smooth, birational, universally closed, surjective, and small (see Remark~\ref{R:II-properties-of-factored-morphism}), so is $\Bl_{J(I)}U \to \tBl_{\underline{J}(I)_\bullet}U$.
    \item\label{rem2} Since $\pi_{\underline{\fj}_\bullet}$ is birational, surjective, and universally closed (see Remark~\ref{R:multi-weighted-blow-up-properties}), so is $\pi_{J(I)}$.
    \item\label{rem3} If $k=0$, then $\varpi_{\underline{J}(I)_\bullet}$ is logarithmically smooth, see \cite[Section 4]{quek-weighted-log-resolution}, and thus by~\eqref{rem1}, so is $\pi_{J(I)}$. This is not true if $k \geq 1$.
    \item\label{rem4} On the other hand, if $Q = 0$, then $\tBl_{\underline{J}(I)_\bullet}U$ is already smooth and $\Bl_{J(I)}U = \tBl_{\underline{J}(I)_\bullet}U$ (see beginning of the proof of Proposition~\ref{P:canonicity-of-multi-weighted-blow-up-II}).
    \item\label{rem5} $\Bl_{J(I)}U$ admits a good moduli space, and it coincides with the coarse moduli space of $\tBl_{\underline{J}(I)_\bullet}U$, which is equal to the schematic blow-up $\bl_{J(I)}U$. %\cite[Proposition 1.6.1(iii)]{quek-rydh-weighted-blow-up} 
    Indeed, \[
        \qquad \qquad \tBl_{\underline{J}(I)_\bullet}U = \bl_{J(I)}U \times_{\bl_{\fj}\AA^n} \tBl_{\underline{\fj}_\bullet}\AA^n
    \]
    because the bottom square of \eqref{E:local-description} is Cartesian. Thus, \[
        \qquad \qquad \Bl_{J(I)}U = \bl_{J(I)}U \times_{\bl_{\fj}\AA^n} \Bl_\fj\AA^n.
    \]
    Since $\bl_{\fj}\AA^n$ is the good moduli space of $\Bl_\fj\AA^n$ (see Remark~\ref{R:II-properties-of-factored-morphism}), it follows from \cite[Proposition 4.7(i)]{alper-good-moduli-spaces} that $\bl_{J(I)}U$ is the good moduli space of $\Bl_{J(I)}U$.
\end{enumerate}
\end{xpar}

\begin{remark}\label{R:gms-of-multi-weighted-blow-up-along-center}
Because of Remark~\ref{R:isomorphism-above-most-singular-locus}, the remarks in Section~\ref{X:local-description}\eqref{rem1}--\eqref{rem5} globalize immediately. For example, \eqref{rem5} implies that $\Bl_{J(I)}Y$ admits a good moduli space, and it coincides with the coarse moduli space $\bl_{J(I)}Y$ of $\tBl_{\underline{J}(I)_\bullet}Y$; see \cite[Proposition 4.7(ii)]{alper-good-moduli-spaces}.
\end{remark}

\begin{xpar}[Local description via multi-graded Rees algebras]\label{X:rees-algebra-description}
Let $p \in U \subset Y$ be as in~\eqref{E:local-description}. As in Section~\ref{X:multi-graded-rees-algebras}, one may express $\Bl_{J(I)}U$ as \[
     \left[\GSpec_U\left(\sR^U_\bullet\right) \smallsetminus V\left(J_{\Sigma_\fj}\right) \q \Gm^{\bE(\fj)}\right], 
\]
where $\sR^U_\bullet$ is the $\ZZ^{\bE(\fj)}$-graded Rees algebra \[
    \sO_U\left[t_\rho^{-1} \colon \rho \in \bE(\fj)\right]\left[x_i \cdot \prod_{\rho \in \bE(\fj)}{t_\rho^{\uu_{\rho,i}}} \colon 1 \leq i \leq n\right] \; \subset \; \sO_U\left[t_\rho^{\pm} \colon \rho \in \bE(\fj)\right]
\]
and the $\Gm^{\bE(\fj)}$-action is induced by the $\ZZ^{\bE(\fj)}$-grading on $\sR^U_\bullet$. As in \eqref{E:substitution}, we set $x_i' := x_i \cdot \prod_{\rho \in \bE(\fj)}{t_\rho^{\uu_{\rho,i}}}$ for $1 \leq i \leq n$ and $x_\rho' := t_\rho^{-1}$ for $\rho \in \bE(\fj)$. Moreover, if $Q \neq 0$, the morphism \[
    \begin{tikzcd}
    \Bl_{J(I)}U = \left[\GSpec_U\left(\sR^U_\bullet\right) \smallsetminus V\left(J_{\Sigma_{\fj}}\right) \q \Gm^{\bE(\fj)}\right] \arrow[to=2-1] \\ \tBl_{\underline{J}(I)_\bullet}U = \left[\GSpec_U(\underline{J}(I)_\bullet) \smallsetminus V\left(\underline{J}(I)_+\right) \q \Gm\right]
    \end{tikzcd}
\]
is then induced by \begin{align}\label{E:factored-morphism-formula-for-center}
    x_i' &\longmapsto x_i \cdot \prod_{\rho \in \bE(\fj)}{\left(x_\rho'\right)^{\uu_{\rho,i}}} \nonumber, \\
    t^{-1} &\longmapsto \prod_{\rho \in \bE^+(\fj)}{\left(x_\rho'\right)^{\frac{N_\rho(\fj)}{\ell}}} = \begin{cases}
         \prod_{\rho \in \bE^+(\fa)}{\left(x_\rho'\right)^{N_\rho(\fa)}}\quad &\textrm{if } k = 0, \\
        \prod_{\rho \in \bE(\fj)}{\left(x_\rho'\right)^{\frac{N_\rho(\fj)}{\ell}}} \quad &\textrm{if } k \geq 1;
    \end{cases}
\end{align}
see \eqref{E:factored-morphism-formula} and \eqref{E:II-fM-to-sM}.
\end{xpar}

\hfill

We next turn our attention to the various transforms of $I$ under the multi-weighted blow-up of $Y$ along $\underline{J}(I)_\bullet$. We first recall the following. 

\begin{xpar}\label{X:weak-transform-under-weighted-toroidal-blow-up}
Let $d$ and $\ell$ be as in Section~\ref{X:local-description}. By \cite[Proposition 4.4]{quek-weighted-log-resolution}, \begin{equation}\label{E:weak-transform-under-weighted-toroidal-blow-up}
    \varpi_{\underline{J}(I)_\bullet}^{-1} I \cdot \sO_{\tBl_{\underline{J}(I)_\bullet}Y} \; = \; \fI \cdot (t^{-1})^{\ell/d}
\end{equation} for some ideal $\fI$ on $\tBl_{\underline{J}(I)_\bullet}Y$ since $\underline{J}(I)_\bullet^{\ell/d} = \underline{J}(I)_{(\ell/d)\bullet}$ is precisely $\sJ(I)_\bullet$, which is $I$-admissible (see Section~\ref{X:associated-center}). Moreover, by \cite[Theorem 6.5(i)]{quek-weighted-log-resolution}, $\fI$ is the \emph{weak transform} of $I$ under $\varpi_{\underline{J}(I)_\bullet} \colon \tBl_{\underline{J}(I)_\bullet}Y \to Y$, see \cite[Definition 4.5]{quek-weighted-log-resolution}; \textit{i.e.} this is tantamount to saying $\fI \not\subset (t^{-1})$. 
\end{xpar}

\begin{definition}\label{D:weak-transform-of-I}
Set $\sO := \sO_{\Bl_{J(I)}Y}$. We define the \emph{weak transform} of $I$ under $\pi_{J(I)} \colon \Bl_{J(I)}Y \to Y$ as \[
    \left(\pi_{J(I)}\right)_\ast^{-1}I := \zeta^{-1}\fI \cdot \sO, 
\]
where $\fI$ is the weak transform of $I$ under $\varpi_{\underline{J}(I)_\bullet}$ and $\zeta$ denotes the morphism $\Bl_{J(I)}Y \to \tBl_{\underline{J}(I)_\bullet}Y$ in Definition~\ref{D:multi-weighted-blow-up-along-center}.
\end{definition}

The next proposition shows that Definition~\ref{D:weak-transform-of-I} agrees with Definition~\ref{D:weak-transform}.

\begin{proposition}[Local description of weak transform of $I$]\label{P:weak-transform-justification}
Let $p \in U \subset Y$ be as in \eqref{E:local-description}, and set $\sO := \sO_{\Bl_{J(I)}U}$. Then the restriction of\, $(\pi_{J(I)})_\ast^{-1}I$ to $\Bl_{J(I)}U$ equals \[
    \left(\pi_{J(I)}^{-1}I \cdot \sO\right) \cdot \prod_{\rho \in \bE^+(\fj)}{\left(x_\rho'\right)^{-N_\rho}}, 
\]
where for each $\rho \in \bE^+(\fj)$, $N_\rho$ is the largest natural number $n_\rho$ such that the fractional ideal $(\pi_{J(I)}^{-1}I \cdot \sO) \cdot (x_\rho')^{-n_\rho}$ is an ideal on $\Bl_{J(I)}U$.
\end{proposition}

\begin{proof}[Proof of Proposition~\ref{P:weak-transform-justification}]
If $Q = 0$, then $\Bl_{J(I)}U = \tBl_{\underline{J}(I)_\bullet}U$, and there is nothing to show. Henceforth, assume $Q \neq 0$. Let $\sR^U_\bullet$ be defined as in Section~\ref{X:rees-algebra-description}. Under the correspondence in Section~\ref{X:idealistic-exponents}, $\sR^U_\bullet$ corresponds to a tuple $\gamma^U := (\gamma^{[\rho]} \colon \rho \in \bE(\fj))$ of $\#\bE(\fj)$ idealistic exponents over $U$, where each $\gamma^{[\rho]}$ is the idealistic exponent over $U$ associated to the following integrally closed, $\NN$-graded Rees algebra on $U$: \[
    \sR^{[\rho]}_\bullet := \left(x_i^{\frac{1}{\uu_{\rho,i}}} \colon i \in [1,n],\ \uu_{\rho,i} \neq 0\right); 
\]
see Section~\ref{X:associated-idealistic-exponents}. For each $i \in [1,n]$, we also set $\gamma^{[i]}$ to be the idealistic exponent over $U$ associated to the ideal $(x_i)$ on $U$. Finally, $\gamma_I$ denotes the idealistic exponent over $U$ associated to $I|_U$. Then we have the following. 

\begin{proposition}[Local description of weak transform of $I$, explicated]\label{P:local-description-weak-transform}
With the above hypotheses and notation, we have, for every $\rho \in \bE^+(\fj)$,  \[
    N_\rho = K_\rho = \frac{N_\rho(\fj)}{d},
\]
where $K_\rho$ is the largest natural number $k_\rho$ such that $\gamma_I \geq k_\rho \cdot \gamma^{[\rho]}$. If $k \geq 1$, this number is also equal to $a_i \cdot \uu_{\rho,i}$ for every $1 \leq i \leq k$.
\end{proposition}

\begin{proof}
The equality $N_\rho = K_\rho$ can be shown by the same methods as in the proofs of Lemma~\ref{L:inclusion-vs-inequality} and Proposition~\ref{P:admissibility-criterion}. 

For the equality $N_\rho = \frac{N_\rho(\fj)}{d}$, we prove the cases $k=0$ and $k \geq 1$ separately. If $k=0$, we make the canonical identification $\bE^+(\fj) = \bE^+(\fa)$;  for every $\rho$ in that set, note that $N_\rho(\fj) = N_\rho(\fa)$. In this case, $\underline{J}(I)_\bullet$ is the integral closure in $\sO_Y[t]$ of the Rees algebra of $Q = \sM_Y(I)$. By Proposition~\ref{P:total-transform-of-monomial-ideal}\eqref{P:ttomi-2}, \[
    \pi^{-1}_{\underline{J}(I)_\bullet}Q \cdot \sO = \prod_{\rho \in \bE^+(\fa)}{\left(x_\rho'\right)^{N_\rho(\fa)}}.
\]
Since $Q \supset I$, the left-hand side contains $\pi^{-1}_{\underline{J}(I)_\bullet}I \cdot \sO$, so that $N_\rho \geq N_\rho(\fa)$ for every $\rho \in \bE^+(\fa)$. Conversely, by the definition of $N_\rho$, we have $\pi^{-1}_{\underline{J}(I)_\bullet}I \cdot \sO \subset \prod_{\rho \in \bE^+(\fj)}{(x_\rho')^{N_\rho}}$. Taking monomial saturation $\sM_U(-)$, we get \[
    \pi^{-1}_{\underline{J}(I)_\bullet}Q \cdot \sO = \sM_U\left(\pi^{-1}_{\underline{J}(I)_\bullet}I \cdot \sO\right) \subset \prod_{\rho \in \bE^+(\fj)}{\left(x_\rho'\right)^{N_\rho}}.
\]
where the  equality follows from \cite[Corollary 3.3.12]{abramovich-temkin-wlodarczyk-principalization-of-ideals-on-toroidal-orbifolds} since $\pi_{J(I)}$ is logarithmically smooth if $k=0$ (see Section~\ref{X:local-description}\eqref{rem3}). That inclusion shows that $N_\rho(\fa) \geq N_\rho$ for every $\rho \in \bE^+(\fa)$, as desired.

If $k \geq 1$, we show instead that $K_\rho = \frac{N_\rho(\fj)}{d}$. By \cite[Lemma 3.7]{quek-weighted-log-resolution}, \[
    K_\rho = \max\left\lbrace k_\rho \in \NN_{>0} \colon \left(\sR_\bullet^{[\rho]}\right)^{k_\rho} \textrm{ is $I$-admissible} \right\rbrace.
\]
By Lemma~\ref{L:numerical-data-for-j}, $N_\rho(\fj) = (a_id) \cdot \uu_{\rho,i}$ for every $1 \leq i \leq k$ and $\uu_{\rho,i} = 0$ for all $k < i \leq n-r$. Therefore, we have \[
    \left(\sR^{[\rho]}_\bullet\right)^{\frac{N_\rho(\fj)}{d}} = \left(x_1^{a_1},x_2^{a_2},\dotsc,x_k^{a_k},\left(\underline{x}_i^{\frac{N_\rho(\fj)}{\uu_{\rho,i}}} \colon n-r < i \leq n,\ \uu_{\rho,i} \neq 0\right)^{\frac{1}{d}}\right).
    \]
{Letting $\rho \mapsto \overline{\rho}$ be the one-to-one correspondence $\bE^+(\fj) \xrightarrow{\lowsimeq} \bE^+(\fa)$ in Remark~\ref{R:numerical-data-for-j}, the same remark says that $\left(\sR^{[\rho]}_\bullet\right)^{\frac{N_\rho(\fj)}{d}}$ is equal to \[
     \left(x_1^{a_1},x_2^{a_2},\dotsc,x_k^{a_k},\left(\underline{x}_i^{\frac{N_\rho(\fa)}{\uu_{\overline{\rho},i}}} \colon n-r < i \leq n,\ \uu_{\rho,i} \neq 0\right)^{\frac{1}{d}}\right), 
\]
and hence it contains $\sJ(I,p)_\bullet = \left(x_1^{a_1},x_2^{a_2},\dotsc,x_k^{a_k},Q^{1/d}\right)$, by Remark~\ref{R:admissibility-criterion-for-monomial-ideals}.} Since the latter is $I$-admissible, so is the former, whence $K_\rho \geq \frac{N_\rho(\fj)}{d}$. On the other hand, $a_1 = \logord_p(I)$, whence \cite[Corollary 4.10]{quek-weighted-log-resolution} implies that for any $k_\rho > \frac{N_\rho(\fj)}{d}$, $(\sR^{[\rho]}_\bullet)^{k_\rho}$ cannot be $I$-admissible, as desired.
\end{proof}

We return  to the proof of Proposition~\ref{P:weak-transform-justification}. Applying $\zeta^{-1}(-) \cdot \sO$ to \eqref{E:weak-transform-under-weighted-toroidal-blow-up} and then applying \eqref{E:factored-morphism-formula-for-center}, we obtain \[
    \pi_{J(I)}^{-1}I \cdot \sO = \left(\zeta^{-1}\fI \cdot \sO\right) \cdot \prod_{\rho \in \bE^+(\fj)}{\left(x_\rho'\right)^{\frac{N_\rho(\fj)}{\ell} \cdot \frac{\ell}{d}}}. 
\]
It remains to note that $\zeta^{-1}\fI \cdot \sO =: (\pi_{J(I)})_\ast^{-1}I$ and $\frac{N_\rho(\fj)}{\ell} \cdot \frac{\ell}{d}= \frac{N_\rho(\fj)}{d} = N_\rho$ for every $\rho \in \bE^+(\fj)$.
\end{proof}

The next proposition shows that the  equality in Definition~\ref{D:weak-transform-of-I} holds as well for proper transforms. 

\begin{proposition}\label{P:proper-transform-of-I}
Let $I'$ $($resp.\ $\fI')$ denote the proper transform of $I$ under $\pi_{J(I)} \colon \Bl_{J(I)}Y \to Y$ $($resp.\ $\varpi_{\underline{J}(I)_\bullet} \colon \tBl_{\underline{J}(I)_\bullet}Y \to Y)$. Then we have \[
    I' = \zeta^{-1}\fI' \cdot \sO, 
\]
where $\sO := \sO_{\Bl_{J(I)}Y}$ and $\zeta$ denotes the morphism $\Bl_{J(I)}Y \to \tBl_{\underline{J}(I)_\bullet}Y$ in Definition~\ref{D:multi-weighted-blow-up-along-center}.
\end{proposition}

\begin{proof}
To see this, recall that $V(I')$ (resp.\ $V(\fI')$) is the smallest closed substack of $\Bl_{J(I)}Y$ (resp.\ $\tBl_{\underline{J}(I)_\bullet}Y$) containing $V(\pi_{J(I)}^{-1}I \cdot \sO) \smallsetminus V(\pi_{J(I)}^{-1}J(I) \cdot \sO)$ (resp.\ $V(\varpi_{\underline{J}(I)_\bullet}^{-1}I \cdot \sO_{\tBl_{\underline{J}(I)_\bullet}Y}) \smallsetminus V(\varpi_{\underline{J}(I)_\bullet}^{-1}J(I) \cdot \sO_{\tBl_{\underline{J}(I)_\bullet}Y})$). The proposition then follows from the following equalities: \begin{align*}
    \zeta^{-1}\left(V\left(\varpi_{\underline{J}(I)_\bullet}^{-1}I \cdot \sO_{\tBl_{\underline{J}(I)_\bullet}Y}\right)\right) &= V\left(\pi_{J(I)}^{-1}I \cdot \sO\right),  \\
    \zeta^{-1}\left(V\left(\varpi_{\underline{J}(I)_\bullet}^{-1}J(I) \cdot \sO_{\tBl_{\underline{J}(I)_\bullet}Y}\right)\right) &= V\left(\pi_{J(I)}^{-1}J(I) \cdot \sO\right) 
\end{align*}
and the fact that $\zeta$ is closed (see Section~\ref{X:local-description}\eqref{rem1}).
\end{proof}

We can finally deduce the following. 

\begin{theorem}[Invariant drops in a well-ordered set]\label{T:invariant-drops}
Let $I \subset \sO_Y$ be a proper, nowhere zero ideal, and let $I'$ be its proper transform under $\pi_{J(I)}$. Then \[
    \max\inv(I') \leq \max\inv\left((\pi_{J(I)})_\ast^{-1}I\right) < \max\inv(I), 
\]
and all three maximum invariants are contained in the well-ordered set $\NN_\infty^{\leq \dim(Y),!}$; see Section~\ref{X:well-ordered-set}.
\end{theorem}

\begin{proof}
We adopt the notation in Definition~\ref{D:weak-transform-of-I} and Proposition~\ref{P:proper-transform-of-I}. Since $\zeta$ is logarithmically smooth and surjective (see Section~\ref{X:local-description}\eqref{rem1}), we have \[
    \max\inv\left(\left(\pi_{J(I)}\right)_\ast^{-1}I\right) = \max\inv\left(\zeta^{-1}\fI \cdot \sO\right) = \max\inv(\fI) < \max\inv(I), 
\]
where the middle equality is given by Remark~\ref{R:log-invariant}\eqref{R:li-3} and the strict inequality is given by \cite[Theorem~7.2]{quek-weighted-log-resolution}. Recall from Section~\ref{X:log-invariant} that the lengths of $\max\inv(I)$ and $\max\inv(\fI)$ are bounded above by $\dim(Y) = \dim(\tBl_{\underline{J}(I)_\bullet}Y)$, and hence so is the length of $\max\inv\left((\pi_{J(I)})_\ast^{-1}I\right)$. Moreover, since $I' \supset (\pi_{J(I)})_\ast^{-1}I$ (see Remark~\ref{R:transforms-properties}), we also have $\max\inv(I') \leq \max\inv\left((\pi_{J(I)})_\ast^{-1}I\right)$. Finally, the proposition together with Remark~\ref{R:log-invariant}\eqref{R:li-3} imply that \[
    \max\inv\left(I'\right) = \max\inv\left(\zeta^{-1}\fI' \cdot \sO\right) = \max\inv\left(\fI'\right).
\]
Since the length of $\max\inv(\fI')$ is also bounded above by $\dim(\tBl_{\underline{J}(I)_\bullet}Y) = \dim(Y)$ (see Section~\ref{X:log-invariant}), so is the length of $\max\inv(I')$.
\end{proof}

\begin{xpar}[Functoriality]\label{X:functoriality}
{Given a strict, smooth, and surjective morphism $f$ from another smooth, strict toroidal $\kk$-scheme $\widetilde{Y}$ to $Y$, we have $\widetilde{Y} \times_Y \Bl_{J(I)}Y = \Bl_{J(f^{-1}I \cdot \sO_{\widetilde{Y}})}\widetilde{Y}$. Here, the fibre product is taken in the category of fs logarithmic Artin stacks.

Indeed, since $\sJ(f^{-1}I \cdot \sO_{\widetilde{Y}})_\bullet = f^{-1}\sJ(I)_\bullet \cdot \sO_{\widetilde{Y}}$ (see Section~\ref{X:associated-center}), we have $\widetilde{Y} \times_Y \tBl_{\underline{J}(I)_\bullet}Y = \tBl_{\underline{J}(f^{-1}I \cdot \sO_{\widetilde{Y}})_\bullet}\widetilde{Y}$. As a consequence, \begin{align*}
    \widetilde{Y} \times_Y \Bl_{J(I)}Y &= \left(\widetilde{Y} \times_Y \tBl_{\underline{J}(I)_\bullet}Y\right) \times_{\tBl_{\underline{J}(I)_\bullet}Y} \Bl_{J(I)}Y \\
    &= \tBl_{\underline{J}(f^{-1}I \cdot \sO_{\widetilde{Y}})_\bullet} \times_{\tBl_{\underline{J}(I)_\bullet}Y} \Bl_{J(I)}Y.
\end{align*}
Since $\Bl_{J(I)}Y$ is by definition the canonical smooth, toroidal Artin stack over $\tBl_{\underline{J}(I)_\bullet}Y$, Remark~\ref{R:satriano-construction} implies that $\widetilde{Y} \times_Y \Bl_{J(I)}Y$
is then the canonical smooth, toroidal Artin stack over $\tBl_{\underline{J}(f^{-1}I \cdot \sO_{\widetilde{Y}})_\bullet}\widetilde{Y}$ and is therefore by definition $\Bl_{J(f^{-1}I \cdot \sO_{\widetilde{Y}})}\widetilde{Y}$.}
\end{xpar}

\subsection{Proofs of the main results}\label{4.3}

Before we prove the results in the introduction, the next remark is necessary; it will be implicit in the proofs in this section. 

\begin{remark}\label{R:case-of-artin-stacks}
{Recall that $\inv_p(I)$ and the formation of $\sJ(I,p)_\bullet$ (resp.\ $\max\inv(I)$, the formation of $\sJ(I)_\bullet$, and the formation of $\Bl_{J(I)_\bullet}Y$) are functorial with respect to strict, smooth (resp.\ strict, smooth, surjective) morphisms of smooth, strict toroidal $\kk$-schemes (see Remarks~\ref{R:log-invariant} and~\ref{R:center}, and Sections~\ref{X:associated-center} and \ref{X:functoriality}).} Thus, by descent, the discussions and constructions in Section~\ref{4.1} (resp.\ in Section~\ref{4.2}) extend immediately to the case where $Y$ is a toroidal (resp.\ smooth, toroidal) Artin stack over $\kk$ (see Section~\ref{X:ambient-spaces}). Indeed, one can work on an atlas $Y_1 \rightrightarrows Y_0$ of $Y$ by strict toroidal (resp.\ smooth, strict toroidal) $\kk$-schemes, where the arrows are strict, smooth, and surjective.
\end{remark}

\begin{proof}[Proof of Theorem~\ref{T:invariant-drops-intro}]
We may assume $X \neq \emptyset$ or $Y$. Let $I$ be the underlying ideal of $X \subset Y$. We set $\pi \colon Y' \to Y$ to be $\pi_{J(I)} \colon \Bl_{J(I)}Y \to Y$. Then part~\eqref{T:idi-1} is immediate, part~\eqref{T:idi-2} follows from Theorem~\ref{T:invariant-drops}, part~\eqref{T:idi-3}  follows from Remark~\ref{R:isomorphism-above-most-singular-locus}, and part~\eqref{T:idi-4} follows from parts ~\eqref{T:idi-2} and~\eqref{T:idi-4} of Section~\ref{X:local-description} (see Remark~\ref{R:gms-of-multi-weighted-blow-up-along-center}). Finally, functoriality with respect to strict, smooth, and surjective morphisms of pairs follows from Section~\ref{X:functoriality}.
\end{proof}

\begin{proof}[Proof of Theorem~\ref{T:logarithmic-embedded-resolution}]
We may assume that $Y$ is a smooth, strict toroidal $\kk$-scheme and that $X \neq \emptyset$ or $Y$. We define $\Pi$ inductively. After the $k^{\textrm{th}}$ step of the algorithm (\textit{i.e.} we have defined $Y_k \xrightarrow{\pi_k} Y_{k-1} \xrightarrow{\pi_{k-1}} \dotsb \xrightarrow{\pi_1} Y_0 = Y$ with proper transforms $X_i \subset Y_i$ of $X$), we undertake the following steps for the $(k+1)^{\textrm{st}}$ step: \begin{enumerate}
    \item\label{step1} If $\max\inv(X_k \subset Y_k) = (1,1,\dotsc,1)$ of some length $c$, we \emph{\underline{claim}} that the locus $C_k$ consisting of points $p \in \abs{Y_k}$ such that $\inv_p(X_k \subset Y_k) = \max\inv(X_k \subset Y_k) = (1,1,\dotsc,1)$ (of length $c$) is both open and closed in $X_k$ and hence is a smooth connected component of $X_k$. We admit this \emph{\underline{claim}} for now, and postpone its proof.
    
    If $C_k = X_k$, we then stop at the $k^{\textrm{th}}$ step. If $C_k \neq X_k$ and $\max\inv(X_k \smallsetminus C_k \subset Y_k) = (1,1,\dotsc,1)$ of some length $c' > c$, we repeat step~\eqref{step1}  with $X_k \subset Y_k$ replaced by $X_k \smallsetminus C_k \subset Y_k$. Otherwise, we proceed to step~\eqref{step2}  with $X_k \subset Y_k$ replaced by $X_k \smallsetminus C_k \subset Y_k$.
    
    \item\label{step2}  If $\max\inv(X_k \subset Y_k) \neq (1,1,\dotsc,1)$ of any length $c$, we apply Theorem~\ref{T:invariant-drops-intro} to $X_k \subset Y_k$, which gives us $\pi_{k+1} \colon Y_{k+1} \to Y_k$ and a proper transform $X_{k+1} \subset Y_{k+1}$ of $X_k$ which satisfies $\max\inv(X_{k+1} \subset Y_{k+1}) < \max\inv(X_k \subset Y_k)$.
\end{enumerate}
Under this procedure, observe that at every point $p$ of $X$, the invariant of proper transforms $X_i \subset Y_i$ at points $p'$ above $p$ must eventually drop to $(1,1,\dotsc,1)$ of some length and, moreover, cannot drop to $(0)$ without first dropping to $(1,1,\dotsc,1)$ of some length. This is because $X$ is reduced and generically toroidal, and therefore so are the proper transforms $X_i$ of $X$ (see Remark~\ref{R:transforms-properties}). Since the lengths of these invariants are bounded above by $\dim(Y)$\footnote{If $Y$ is more generally an Artin stack, note that the lengths of these invariants are only bounded above by $\dim(Y_0)$ for any atlas $Y_1 \rightrightarrows Y_0$ of $Y$ by smooth, strict toroidal $\kk$-schemes.} (see the last statement in Theorem~\ref{T:invariant-drops}), this procedure eventually terminates to the desired $\Pi$.

{Finally, if $f \colon \widetilde{Y} \to Y$ is a strict, smooth morphism of smooth, toroidal Artin stacks over $\kk$} and the logarithmic embedded resolution of $X \subset Y$ is $\Pi \colon Y_N \xrightarrow{\pi_N} Y_{N-1} \xrightarrow{\pi_{N-1}} \dotsb \xrightarrow{\pi_1} Y_0 = Y$, then it follows from the functoriality in Theorem~\ref{T:invariant-drops-intro} that the logarithmic embedded resolution of $X \times_Y \widetilde{Y} \subset \widetilde{Y}$ agrees step-by-step with the pull-back of $\Pi$ along $f \colon \widetilde{Y} \to Y$: \[
    Y_N \times_Y \widetilde{Y} \xrightarrow{f^\ast\pi_N} Y_{N-1} \times_Y \widetilde{Y} \xrightarrow{f^\ast\pi_{N-1}} \dotsb \xrightarrow{f^\ast\pi_1} \widetilde{Y}
\]
after removing any $f^\ast\pi_i$ which are empty blow-ups, which may occur whenever $f$ is not surjective.
\end{proof}

\begin{proof}[Proof of the \underline{claim}]
We may assume that $Y$ is a smooth, strict toroidal $\kk$-scheme. Let $p \in C_k$, and let $x_1,\dotsc,x_c$ be ordinary parameters associated to $I$ at $p$, which are defined on some open $U \subset Y_k$. Then $\sJ(I,p)_\bullet$ is simply {the} Rees algebra associated to the ideal $J(I,p) = (x_1,\dotsc,x_c)$, and we have \begin{enumerate}[label=(\alph*),ref=\alph*]
    \item\label{pf-a} $I_p \subset (x_1,\dotsc,x_c)$ (see Remark~\ref{R:center}).
    \item\label{pf-b}  By the description in Section~\ref{X:well-ordered-set}, note that for $p' \in \abs{U} \cap \abs{X_k}$, we have $\inv_{p'}(I) = (a_1,\dotsc,a_\ell) < \max\inv(I) = (1,1,\dotsc,1)$ (of length $c$) if and only if \[
        \inv_{p'}(I) = (\overbrace{1,1,\dotsc,1}^{\textup{length $c$}},a_{c+1},\dotsc,a_\ell) \quad \textrm{ with } \ell > c.
    \]
    By Section~\ref{X:associated-data}\eqref{X:ad-2}, that happens if and only if $I_{p'}|_{V(x_1,\dotsc,x_c)} \neq 0$, \textit{i.e.} $I_{p'} \not\subset (x_1,\dotsc,x_c)$.
\end{enumerate}
Set $U' := (U \cap X_k) \smallsetminus V\left((x_1,\dotsc,x_c):I\right)$. Then $U'$ is open in $X_k$, contains the point $p$ (by \eqref{pf-a}), and is moreover contained in $C_k$ (by \eqref{pf-b}). Since $p \in \abs{C_k}$ was arbitrary, we conclude that $C_k$ is open in $X_k$. But $C_k$ is also closed in $X_k$, by the upper semi-continuity of $\inv$ (Remark~\ref{R:log-invariant}).
\end{proof}

\begin{remark}
Note that the proof of Theorem~\ref{T:logarithmic-embedded-resolution} simplifies if one assumes $X \subset Y$ is of pure codimension~$c$. In that case, one iterates Theorem~\ref{T:invariant-drops-intro} till $\max\inv(X_k \subset Y_k) = (1,1,\dotsc,1)$ of length $c$, and the procedure terminates. Indeed, $C_k = X_k$ in \eqref{step1} of the proof of Theorem~\ref{T:logarithmic-embedded-resolution} since they both contain the dense open $X^{\logsm}$ and are both of pure codimension $c$ in $Y_k$.
\end{remark}

\begin{proof}[Proof of Theorem~\ref{T:re-embedding-principle}]
We may assume that $Y$ is a smooth, strict toroidal $\kk$-scheme. Note that $I_1 = (x_0) + I$ and $\sD^{\leq 1}_Y(I_1) = (1)$ with maximal contact element $x_0$ everywhere, so that $I[2] = I_1|_{V(x_0)=Y} = I$. Thus, part~\eqref{T:rep-1} follows by the definition of $\inv$. For part~\eqref{T:rep-2}, part\eqref{T:rep-1} implies that there is a bijection \[
    \lbrace p \in Y \colon \inv_p(I) = \max\inv(I) \rbrace \xrightarrow{\;\lowsimeq\;} \lbrace p_1 \in Y_1 \colon \inv_{p_1}(I_1) = \max\inv(I_1) \rbrace
\]
which sends $p$ to $(p,0)$. Moreover, for any such $p \in Y$, if \begin{align*}
    \sJ(I,p)_\bullet &= \left(x_1^{a_1},x_2^{a_2},\dotsc,x_k^{a_k},Q^{1/d}\right), \\
    J(I,p) &= \textrm{integral closure of } \left(x_1^{a_1d},x_2^{a_2d},\dotsc,x_k^{a_kd},Q\right)
\end{align*}
as in Sections~\ref{X:center} and \ref{X:dth-Veronese-subalgebra}, then \begin{align*}
    \sJ\left(I_1,(p,0)\right)_\bullet &= \left(x_0,x_1^{a_1},\dotsc,x_k^{a_k},Q^{1/d}\right), \\
    J\left(I_1,(p,0)\right) &= \textrm{integral closure of } \left(x_0^d,x_1^{a_1d},x_2^{a_2d},\dotsc,x_k^{a_kd},Q\right).
\end{align*}
Thus, Corollary~\ref{C:proper-transform-of-first-coordinate-hyperplane} implies that $Y'$ is canonically identified with the proper transform $V(x_0')$ of $Y = V(x_0) \subset Y_1$ in $Y_1'$. Moreover, if $I'$ (resp.\ $I_1'$) denotes the underlying ideal of $X' \subset Y'$ (resp.\ $X_1' \subset Y_1'$), then $I_1' = (x_0') + I'$, so $X' = X_1'$.
\end{proof}

\begin{proof}[Proof of Theorem~\ref{T:logarithmic-resolution}]
We embed $X$ smooth locally in a smooth $\kk$-scheme $Y$ \emph{in pure codimension} and apply Theorem~\ref{T:logarithmic-embedded-resolution} to the pair $(X \subset Y)$, where $X \subset Y$ are given the trivial logarithmic structures, to obtain a local logarithmic resolution of $X$. It remains to show these local logarithmic resolutions of $X$ are compatible, in the sense that they do not depend on the choice of local embeddings, and that they glue.

For this, it suffices to show that given two closed embeddings of $X$ into two smooth, pure-dimensional Artin stacks $Y_i$ ($i=1,2$) over $\kk$, the logarithmic resolutions of $X$ obtained from $\sF_{\ler}(X \subset Y_i)$ ($i=1,2$) coincide. Firstly, if $\dim(Y_1) = \dim(Y_2)$, then the two closed embeddings $X \inj Y_i$ are smooth locally isomorphic, so by the functoriality in Theorem~\ref{T:logarithmic-embedded-resolution}, the resulting logarithmic resolutions of $X$ obtained from $\sF_{\ler}(X \subset Y_i)$ coincide. In general, one reduces to the earlier case, by repeatedly applying Theorem~\ref{T:re-embedding-principle}.
\end{proof}

\section{Examples and further remarks}\label{C:further-remarks}

\subsection{Examples}\label{5.1}

Throughout this section, we freely adopt the notation introduced in Section~\ref{3.2} and throughout Section~\ref{C:log-resolution-via-multi-weighted-blow-ups}.

\begin{example}\label{EX:newton-non-degenerate}
Let $Y = \AA^{3;3} = \Spec(\NN^3 \to \kk[\underline{x},\underline{y},\underline{z}])$, and consider the following hypersurface: \[
    X = V(I) = V(f) := V\left(\underline{x}^2+\underline{y}^2\underline{z}+\underline{z}^3\right) \subset Y.
\]
Then $\max\inv(I) = (\infty)$, and $J(I)$ is the integral closure of $\sM_Y(I) = (\underline{x}^2,\underline{y}^2\underline{z},\underline{z}^3)$. Let $\pi \colon Y' := \Bl_{J(I)}Y \to Y$, which was explicated in Example~\ref{EX:second-example}. By the equations therein, the total transform of $I$ is \begin{equation}\label{E:newton-non-degenerate}
    \pi^{-1}I \cdot \sO_{Y'} = \underline{u}_1^6\underline{u}_2^2 \cdot \underbrace{\left(\underline{x}'^2+\underline{y}'^2\underline{z}'+\underline{z}'^3\underline{u}_2^4\right)}_{\textrm{proper transform $I'$}}.
\end{equation}
Finally, $\sD^{\leq 1}(I') = (\underline{x}'^2,\underline{y}'^2\underline{z}',\underline{z}'^3\underline{u}_2^4)$, which is the unit ideal on the $\underline{x}'$-chart, $\underline{y}'\underline{z}'$-chart, and $\underline{z}'\underline{u}_2$-chart of $Y'$. Therefore, $\max\inv(I') = (1) < (\infty) = \max\inv(I)$, and we get logarithmic resolution of singularities in one step. 
\end{example}

The above is an example of a \emph{Newton non-degenerate} polynomial; see \cite[Section 2]{kouchnirenko-newton-polytopes}. We recall this notion here. Let $f = \sum_{\ba \in \NN^n}{c_\ba \cdot \pmb{x}^\ba} \in \kk[x_1,\dotsc,x_n]$ be a non-constant polynomial such that $f(0) = 0$. For every face $\tau$ of the Newton polyhedron $P_f$ of the ideal $(f)$ (see Definition~\ref{D:newton-polyhedron-of-ideal}), we set $f_\tau := \sum_{\ba \in \NN^n \cap \tau}{c_\ba \cdot \pmb{x}^\ba}$. We say $f$ is Newton non-degenerate if for every face\footnote{This includes $P_f$, the $n$-dimensional face of itself.} $\tau$ of $P_f$, $V(f_\tau) \subset \AA^n$ is non-singular in the torus $\Gm^n \subset \AA^n$. This condition guarantees that the singularity theory of $V(f) \subset \AA^n$ is, to a certain extent, governed by its Newton polyhedron. Indeed, Example~\ref{EX:newton-non-degenerate} is manifested by a general phenomenon that was earlier observed in \cite[Proposition 8.31]{bultot-nicaise-log-smooth-models}.\footnote{We thank Johannes Nicaise for bringing this to our attention.} 

\begin{theorem}\label{T:newton-non-degenerate}
Let $f \in \kk[x_1,\dotsc,x_n]$ be a Newton non-degenerate polynomial, and assume $x_i$ does not divide $f$ for every $1 \leq i \leq n$. Then the multi-weighted blow-up of $\AA^n$ along the monomial saturation $\fa_{(f)}$ of\, $(f)$ $($see Definition~\ref{D:newton-polyhedron-of-ideal}\,$)$ is a logarithmic embedded resolution of singularities for $V(f) \subset \AA^{n;n} = \Spec(\NN^n \to \kk[\underline{x}_1,\dotsc,\underline{x}_n])$. 

In other words, our logarithmic embedded resolution algorithm in  Theorem~\ref{T:logarithmic-embedded-resolution}, applied to the pair $V(f) \subset \AA^{n;n} = \Spec(\NN^n \to \kk[\underline{x}_1,\dotsc,\underline{x}_n])$, terminates after one step.
\end{theorem}

{We remark that we impose the condition that $x_i$ does not divide $f$ for every $1 \leq i \leq n$, so that $V(f) \subset \AA^{n;n}$ is generically toroidal and hence satisfies the hypotheses of Theorem~\ref{T:logarithmic-embedded-resolution}. The same proof below, with some minor modifications, continues to work if one drops that condition.}

\begin{proof}
{Let $\fa := \fa_{(f)}$, and $\pi_\fa \colon \Bl_\fa\AA^n = \left[X_{\widehat{\Sigma}_\fa} \q \Gm^{\bE(\fa)}\right] \to \AA^n$. Let $\overline{\sigma}$ be an arbitrary cone in $\widehat{\Sigma}_\fa$, and let $\sigma$ denote its image under the morphism $\beta \colon \ZZ^{\Sigma_\fa(1)} \to \ZZ^n$ which sends $\be_\rho$ to $\bu_\rho$ for every $\rho \in \Sigma_\fa(1)$ (see Definition~\ref{D:multi-weighted-blow-up}). By the definition of $\widehat{\Sigma}_\fa$, there is a \emph{smallest} cone $\sigma'$ in $\Sigma_\fa$ such that $\sigma$ is a sub-cone of $\sigma'$. Let $\tau$ be the face of $P_\fa = P_f$ corresponding to $\sigma'$ (see Section~\ref{X:conventions-2.1}). If $O(\sigma)$ denotes the $\Gm^{\widehat{\Sigma}_\fa(1)}$-orbit of $X_{\widehat{\Sigma}_\fa}$ corresponding to $\overline{\sigma}$, we claim the proper transform of $V(f) \subset \AA^n$ under $\pi$ is non-singular on the $\left(\Gm^{\widehat{\Sigma}_\fa(1)} \q \Gm^{\bE(\fa)}\right)$-orbit $\left[O(\sigma) \q \Gm^{\bE(\fa)}\right] \subset \left[X_{\widehat{\Sigma}_\fa} \q \Gm^{\bE(\fa)}\right] = \Bl_\fa\AA^n$. This claim proves the proposition since $X_{\widehat{\Sigma}_\fa} = \bigsqcup_{\overline{\sigma} \in \widehat{\Sigma}_\fa}{O(\sigma)}$. We prove the claim in three steps.

\begin{xpar}[Step 1]\label{X:newton-step-1}
Let $U_\sigma$ denote the affine toric variety associated to the cone $\overline{\sigma}$ in $\widehat{\Sigma}_\fa$. By \eqref{E:presentation-of-multi-weighted-blow-up} and \eqref{E:multi-graded-rees-algebra}, $D_+(\sigma) := \left[U_{\sigma} \q \Gm^{\bE(\fa)}\right] \subset \Bl_\fa\AA^n$ is \[
    \left[\GSpec_{\AA^n}\left(\frac{\sO_{\AA^n}[x_1',\dotsc,x_n']\left[x_\rho' \colon \rho \in \bE(\fa)\right]\left[(x_\sigma')^{-1}\right]}{\left(x_i' \cdot \prod_{\rho \in \bE(\fa)}{\left(x_\rho'\right)^{\uu_{\rho,i}}} - x_i \colon 1 \leq i \leq n\right)}\right) \q \Gm^{\bE(\fa)} \right], 
\]
where \begin{equation}\label{E:newton-step-1a}
    x_\sigma' \; := \prod_{\rho \in \Sigma_\fa(1) \smallsetminus \sigma(1)}{x_\rho'} \quad \textrm{with }  \sigma(1) := \lbrace \rho \in \Sigma_\fa(1) \colon \rho \subset \sigma \rbrace.
\end{equation}
Next, by {Lemma~\ref{L:slicing}}, the assignment $x_\rho' \mapsto 1$ for every $\rho \in \bE(\fa) \smallsetminus \sigma(1)$ identifies $D_+(\sigma)$ with \begin{equation}\label{E:newton-step-1b}
    \left[\GSpec_{\AA^n}\left(\frac{\sO_{\AA^n}[x_1',\dotsc,x_n']\left[x_\rho' \colon \rho \in \bE(\fa) \cap \sigma(1) \right]\left[(x_\sigma')^{-1}\right]}{\left(x_i' \cdot \prod_{\rho \in \bE(\fa) \cap \sigma(1)}{\left(x_\rho'\right)^{\uu_{\rho,i}}} - x_i \colon 1 \leq i \leq n\right)}\right) \q \Gm^{\bE(\fa) \cap \sigma(1)} \right], 
\end{equation}
where we redefine $x_\sigma'$ as \[
    x_\sigma' := \prod_{i \in [1,n] \smallsetminus \sigma(1)}{x_i'}.
\]
\end{xpar}

\begin{xpar}[Step 2]\label{X:newton-step-2}
Write $f = \sum_{\ba \in \NN^n}{c_\ba \cdot \pmb{x}^\ba}$. By Proposition~\ref{P:total-transform-of-monomial-ideal}\eqref{P:ttomi-1}, the total transform of $f$ on \eqref{E:newton-step-1b} is \[
    f = \prod_{\rho \in \bE(\fa) \cap \sigma(1)}{\left(x_\rho'\right)^{N_\rho(\fa)}} \cdot \underbrace{\left(\sum_{\ba \in \NN^n}{c_\ba \cdot (\pmb{x}')^\ba \cdot \prod_{\rho \in \bE(\fa) \cap \sigma(1)}{\left(x_\rho'\right)^{(\ba \cdot \bu_\rho) - N_\rho(\fa)}}}\right)}_{\textrm{proper transform $f'$}}.
\]
Let us record two essential observations about $f'$: \begin{enumerate}
    \item\label{X:ns2-1} If $\ba \in \NN^n \cap \tau$, then for every $\rho \in \sigma(1)$, we have $\ba \in \tau \subset H_\rho$, \textit{i.e.} $\ba \cdot \bu_\rho = N_\rho(\fa)$. In particular, if $\rho = i \in [1,n] \cap \sigma(1)$, we note separately that this means $a_i = 0$.
    \item\label{X:ns2-2} If $\ba \in \NN^n \smallsetminus \tau$, there exists  a $\rho \in \sigma(1)$ such that $\ba \in \NN^n \smallsetminus H_\rho$,  \textit{i.e.} $\ba \cdot \bu_\rho > N_\rho(\fa)$. This is because $\tau = \bigcap_{\rho \in \sigma'(1)}{H_\rho} = \bigcap_{\rho \in \sigma(1)}{H_\rho}$.
\end{enumerate}
\end{xpar}

\begin{xpar}[Step 3]\label{X:newton-step-3}
Finally, \[
    \left[O(\sigma) \q \Gm^{\bE(\fa)}\right] = V\left(x_\rho' \colon \rho \in \sigma(1)\right) \xhookrightarrow{\textrm{closed}} D_+(\sigma).
\]
Combining the above with \eqref{E:newton-step-1b}, we get the identification \begin{equation}\label{E:newton-step-3}
    \left[O(\sigma) \q \Gm^{\bE(\fa)}\right] = \left[\Spec\left(\kk\left[\left(x_i'\right)^\pm \colon i \in [1,n] \smallsetminus \sigma(1)\right]\right) \q \Gm^{\bE(\fa) \cap \sigma(1)}\right].
\end{equation}
Moreover, by Section~\ref{X:newton-step-2}\eqref{X:ns2-1}--\eqref{X:ns2-2}, the restriction of $f'$ to \eqref{E:newton-step-3} is \[
    \sum_{\ba \in \NN^n \cap \tau}{c_\ba \cdot (\pmb{x}')^\ba}.
\]
Since the above expression matches that of $f_\tau$, the claim follows.
\end{xpar}}
\end{proof}

The next three examples (Examples~\ref{EX:newton-non-degenerate-modified-1}, \ref{EX:newton-non-degenerate-modified-2}, \ref{EX:newton-non-degenerate-modified-3}) will revisit the same hypersurface $X = V(I) = V(f) := V(x^2+y^2z+z^3) \subset \AA^3$ from before, but we explore what happens if we vary the toroidal logarithmic structure on $\AA^3$.

\begin{example}\label{EX:newton-non-degenerate-modified-1}
Consider $Y = \AA^{3;2} = \Spec(\NN^2 \to \kk[x,\underline{y},\underline{z}])$. Then we have $\max\inv(I) = (2,\infty)$ and $J(I) = (x^2,\underline{y}^2\underline{z},\underline{z}^3)$. The multi-weighted blow-up $\pi \colon Y' := \Bl_{J(I)}Y \to Y$ is schematically the same as the one in Example~\ref{EX:newton-non-degenerate}, and we still have \eqref{E:newton-non-degenerate} (but $x'$ is no longer underlined) and logarithmic resolution of singularities in one step.
\end{example}

\begin{example}\label{EX:newton-non-degenerate-modified-2}
Next, consider $Y = \AA^{3;0}$ (trivial logarithmic structure). Then $\max\inv(I) = (2,3,3)$, and $\underline{J}(I)_\bullet = (x^{1/3},y^{1/2},z^{1/2})$. The multi-weighted blow-up $\pi \colon Y' := \Bl_{J(I)}Y \to Y$ is the weighted blow-up of Example~\ref{EX:weighted-blow-ups}. By the equations therein, the total transform of $I$ is \[
        \qquad \pi^{-1}I \cdot \sO_{Y'} = \underline{u}^6 \cdot \underbrace{\left(x'^2+y'^2z'+z'^3\right)}_{\textrm{proper transform $I'$}}.
    \]
    We have $\sD^{\leq 1}(I') = (x',y'z',y'^2+3z'^2,z'^3)$, which is the unit ideal on the $x'$-chart, $y'$-chart, and $z'$-chart of~$Y'$. Thus, we have $\max\inv(I') = (1) < (\infty) = \max\inv(I)$, \textit{i.e.} logarithmic resolution of singularities in one step. 
\end{example} 

\begin{remark}\label{R:newton-non-degenerate-modified-2}
Example~\ref{EX:newton-non-degenerate-modified-2} is also part of a more general phenomenon: namely, $X = V(I)$ has a $(3,2,2)$-weighted homogeneous isolated singularity at $0 \in \AA^3$, and hence its singularities are resolved after the $(3,2,2)$-weighted blow-up of $\AA^3$ in $0$. From the viewpoint of the \emph{monodromy conjecture} of Denef--Loeser, see \cite{denef-loeser-zeta-functions}, this resolution is ``more minimal'' than the one in Example~\ref{EX:newton-non-degenerate} since it has one less exceptional divisor, namely the one corresponding to the \emph{$B_1$-facet}, see \cite[Definition 3]{lemahieu-van-proeyen-nondegenerate-surface-singularities}, of the Newton polyhedron of~$(f)$. This begs the question of whether in general and to what extent Theorem~\ref{T:newton-non-degenerate} can be refined in this direction. This is pursued in \cite{quek-motivic-monodromy-conjecture}.
\end{remark}
    
\begin{example}\label{EX:newton-non-degenerate-modified-3}
Finally, consider $Y = \AA^{3;1} = \Spec(\NN \to \kk[x,y,\underline{z}])$. Then $\max\inv(I) = (2,\infty)$ and $J(I) = (x^2,\underline{z})$. Then $\pi \colon Y' := \Bl_{J(I)}Y \to Y$ is \[
    \left[\GSpec_{\AA^3}\left(\frac{\sO_{\AA^3}[x',\underline{z'},\underline{u}]}{(x'u-x,\underline{z}'\underline{u}^2-\underline{z})}\right) \smallsetminus V(x',z') \q \Gm\right] \lra \AA^{3;1},
\]
so the total transform of $I$ under $\pi$ is \[
    \pi^{-1}I \cdot \sO_{Y'} = \underline{u}^2 \cdot \underbrace{\left(x'^2+y^2\underline{z}'+\underline{z}'^3\underline{u}^4\right)}_{\textrm{proper transform $I'$}}.
\]
Note that $V(I') \subset Y'$ is non-singular in every chart except the $\underline{z}'$-chart of $Y'$. Nevertheless, we have $\max\inv(I') = (2,2,\infty) < (2,\infty) = \max\inv(I)$, and $J(I)$ is the integral closure of $(x'^2,y^2,\underline{u}^4)$. The composition $\pi' \colon Y'' := \Bl_{\underline{J}(I')}Y' \to Y' \xrightarrow{\pi} Y$ is \[
    \text{\small $\left[\GSpec_{\AA^3}\left(\frac{\sO_{\AA^3}[x'',y',\underline{z}',\underline{u}',\underline{v}]}{(x''\underline{u}'v^3 - x, y'\underline{v}^2 - y, z'\underline{u}'^2\underline{v}^2 - z)}\right) \smallsetminus V\left(x''\underline{v},\underline{z}' (x'',y',\underline{u}')\right) \q \Gm^2\right] \lra \AA^{3;1}$}, 
\]
and the total transform of $I$ under $\pi'$ is \[
    \pi'^{-1}I \cdot \sO_{Y''} = \underline{u}'^2\underline{v}^6 \cdot \underbrace{\left(x''^2 + y'^2\underline{z}' + \underline{z}'^3\underline{u}'^4\right)}_{\textrm{proper transform $I''$}}.
\]
We have $\sD^{\leq 1}(I'') = (x'',y'\underline{z}',\underline{z}'^3\underline{u}_2^4)$, which is the unit ideal on every chart of $Y''$. Thus, $\max\inv(I'') = (1) < (2,2,\infty) = \max\inv(I')$, and we get logarithmic resolution of singularities in two steps.
\end{example}

\begin{remark}\label{R:newton-non-degenerate-modified-3}
Note that the Newton polyhedron of the first center $J(I) = (x^2,\underline{z}) \subset \kk[x,y,\underline{z}]$ in Example~\ref{EX:newton-non-degenerate-modified-3} contains the $B_1$-facet of the Newton polyhedron of $(f)$. As mentioned in the preceding remark, $B_1$-facets are known to be ``problematic'' from the viewpoint of the monodromy conjecture;  see \cite{lemahieu-van-proeyen-nondegenerate-surface-singularities}. Indeed, we saw above that the first multi-weighted blow-up in Example~\ref{EX:newton-non-degenerate-modified-3} did not completely resolve the singularities of $X = V(I) \subset Y$. 
\end{remark}

%It is only fitting to conclude this section by desingularizing a Newton degenerate polynomial:

%\begin{example}\label{EX:newton-degenerate}
%Let $Y = \AA^{3;0} = \Spec(\kk[x,y,z])$, and consider the following hypersurface, see \cite[Exercise 6.45]{kollar-smith-corti-rational-varieties}: \[
    %X = V(I) = V(f) := V\left((x^2+y^2+z^2)^2+x^5\right) \subset Y.
%\]
%\end{example}

\subsection{Reduction of stabilizers and destackification}\label{5.2}

In this section we sketch how one can refine Theorem~\ref{T:logarithmic-resolution} further to obtain a smooth $\kk$-scheme at the end, instead of a smooth Artin stack over $\kk$. The first ingredient is a special case of Edidin--Rydh's \cite[Theorem 2.11]{edidin-rydh-reduction-of-stabilizers}. For the definitions of saturated blow-ups and strong morphisms, see \cite[Definition 3.2]{edidin-rydh-reduction-of-stabilizers} and \cite[Definition 6.8]{edidin-rydh-reduction-of-stabilizers}.

\begin{theorem}[Reduction of stabilizers: Smooth, toroidal case]\label{T:reduction-of-stabilizers}
Let $X$ be a smooth Artin stack over $\kk$ that admits a good moduli space $\bX$, has affine diagonal, and has no generic stackiness. Let $E \subset X$ be a snc divisor. Then there exists a canonical sequence of saturated blow-ups of Artin stacks $\Phi \colon X_N \xrightarrow{\phi_N} X_{N-1} \xrightarrow{\phi_{N-1}} \dotsb \xrightarrow{\phi_1} X_0 = X$ along smooth, closed substacks $C_i \subset X_i$, together with snc divisors $E_i \subset X_i$ with $E_0 = E$, such that the following hold: \begin{enumerate}
    \item Each $X_i$ is a smooth Artin stack over $\kk$ admitting a good moduli space $X_i \xrightarrow{\varphi_i} \bX_i$.
    \item Each $\abs{C_i}$ is the locus in $X_i$ of points of maximum dimensional stabilizer.
    \item Each $\phi_i$ restricts to an isomorphism $X_i \smallsetminus \phi_i^{-1}(C_{i-1}) \xrightarrow{\lowsimeq} X_{i-1} \smallsetminus \varphi_{i-1}^{-1}(\varphi_{i-1}(C_{i-1}))$.
    \item Each $E_i$ is the inverse image of\, $C_{i-1} \cup E_{i-1}$ under $\phi_i$.
    \item The maximum dimension of the stabilizers of points of\, $X_i$ is strictly smaller than that of the stabilizers of points of\, $X_{i-1}$. 
    \item The final stack $X_N$ has finite inertia, with coarse moduli space $X_N \xrightarrow{\varphi_N} \bX_N$.
    \item Each $\phi_i$ induces a schematic blow-up of good moduli spaces $\bX_i \to \bX_{i-1}$, which is an isomorphism over $\bX_{i-1} \smallsetminus \varphi_{i-1}(C_{i-1})$.
\end{enumerate}
The sequence $\Phi$ does not depend on $E$. This procedure, denoted by $\sF_{\ros} \colon X \mapsto X_N$, is functorial with respect to strong morphisms.
\end{theorem}

The second ingredient is due to Bergh--Rydh; see \cite[Theorem B]{bergh-rydh-destackification}. For the definition of stacky blow-ups, see \cite[Section 3.5; in particular, Remark 3.7]{bergh-rydh-destackification}. 

\begin{theorem}[Destackification]\label{T:destackification}
Let $X$ be a smooth Artin stack over $\kk$ with finite inertia, and let $E \subset X$ be a snc divisor. Then there exists a sequence of stacky blow-ups $\Psi \colon X_N \xrightarrow{\psi_N} X_{N-1} \xrightarrow{\psi_{N-1}} \dotsb \xrightarrow{\psi_1} X_0 = X$ along smooth weighted centers $(Z_i,r_i)$, together with snc divisors $E_i \subset X_i$, such that the following hold: \begin{enumerate}
    \item Each $X_i$ is a smooth Artin stack over $\kk$, admitting a coarse moduli space $X_i \xrightarrow{\varphi} \bX_i$. 
    \item Each $E_i$ is the inverse image of\, $Z_{i-1} \cup E_i$ under $\psi_i$.
    \item If\, $\bE_N \subset \bX_N$ denotes the coarse moduli space of\, $E_N \subset X_N$, then $\bE_N$ is a snc divisor on $\bX_N$.
    \item Each $\psi_i$ induces a schematic blow-up of coarse moduli spaces $\bX_i \to \bX_{i-1}$, which is an isomorphism over $\bX_{i-1} \smallsetminus \varphi_{i-1}(Z_{i-1})$.
    \item The final stack
      $X_N$ admits a rigidification $X_N \to (X_N)_\rig$ such that the canonical morphism $(X_N)_\rig \to \bX_N$ is an iterated root stack in $\bE_N$.
\end{enumerate}
This procedure, denoted by $\sF_{\destack} \colon (X,E) \mapsto (X_N,E_N)$, is functorial with respect to smooth morphisms of pairs that are either stabilizer-preserving or tame gerbes.
\end{theorem}

Applying the above procedures after Theorem~\ref{T:logarithmic-resolution}, we recover Hironaka's celebrated theorem in \cite{hironaka-resolution}. 

\begin{theorem}[``Coarse'' logarithmic resolution]\label{T:coarse-logarithmic-resolution}
Given a reduced, pure-dimensional scheme $X$ of finite type over $\kk$, there exists a birational and projective morphism $\Pi \colon X' \to X$ such that: \begin{enumerate}
    \item $X'$ is a smooth $\kk$-scheme; 
    \item $\Pi$ is an isomorphism over the smooth locus $X^\sm$ of $X$; 
    \item $\Pi^{-1}(X \smallsetminus X^\sm)$ is a snc divisor on $X'$.
\end{enumerate}
This procedure, denoted by $\sF_\clr \colon X \mapsto X'$, is functorial with respect to smooth morphisms.
\end{theorem}

\subsection{Remarks on hypotheses}

We first revisit the hypothesis in Theorem~\ref{T:logarithmic-embedded-resolution} that ``$X$ is reduced and generically toroidal''. We remark that: \begin{enumerate}
    \item\label{Rh-1} the first condition can be discarded by applying the  procedure for $X_\red$ to $X$, and replacing $X_N$ in part~\eqref{T:ler-1} of Theorem~\ref{T:logarithmic-embedded-resolution} with $(X_N)_{\red}$; 
    \item\label{Rh-2} one however cannot do without the second condition, or else the final smooth stack $X_N$ is not toroidal. 
\end{enumerate}
In spite of \eqref{Rh-2}, a mindless iteration of Theorem~\ref{T:invariant-drops} yields the following in general. 

\begin{theorem}[Logarithmic principalization]\label{T:principalization}
Given a closed substack $X$ of a smooth, $($strict$)$ toroidal Artin stack $Y$ over $\kk$, there exists a canonical sequence of multi-weighted blow-ups $\Pi \colon Y_N \xrightarrow{\pi_N} Y_{N-1} \xrightarrow{\pi_{N-1}} \dotsb \xrightarrow{\pi_1} Y_0 = Y$, together with proper transforms $X_i \subset Y_i$ of $X$, such that the following hold: \begin{enumerate}
    \item Each $Y_i$ is a smooth, (strict) toroidal Artin stack over $\kk$.
    \item For each $1 \leq i \leq N$, $\max\inv(X_i \subset Y_i) < \max\inv(X_{i-1} \subset Y_{i-1})$. Moreover, $\max\inv(X_N \subset Y_N) = (0)$; \textit{i.e.} $X_N = \emptyset$.
    \item $\Pi^{-1}(X)$ is a snc divisor on $Y_N$.
    \item Each $\pi_i$ is birational, surjective, universally closed, and factors as $Y_i \to \bY_i \to Y_{i-1}$, where $Y_i \to \bY_i$ is a good moduli space relative to $Y_{i-1}$, $\bY_i$ is normal, and $\bY_i \to Y_{i-1}$ is a schematic blow-up $($whence birational and projective$)$.
\end{enumerate}
Moreover, $\Pi$ is functorial with respect to logarithmically smooth morphisms of such pairs $X \subset Y$ $($whether or not surjective$)$.
\end{theorem}

Finally, as pointed out in \cite[Section 8.1]{abramovich-temkin-wlodarczyk-weighted-resolution}, the hypothesis in Theorem~\ref{T:logarithmic-resolution} that ``$X$ is pure-dimensional'' can of course be dropped; \textit{i.e.} one can definitely reduce the general case to that special case, although functoriality is not \textit{a priori} granted. We refrain from discussing this here because the methods are standard. We however emphasize that we do not assume in Theorem~\ref{T:logarithmic-embedded-resolution} that $X$ or $Y$ is pure-dimensional or that $X \subset Y$ is of pure codimension.

\subsection{{Remarks on functoriality}}

One can ask if functoriality in the procedure of Theorem~\ref{T:invariant-drops-intro} holds more generally for a logarithmically smooth, surjective morphism $f \colon \widetilde{Y} \to Y$ of smooth, toroidal Artin stacks over $\kk$. Let $I \subset \sO_Y$ denote the ideal of $X \subset Y$. If $f$ happens to be strict (and hence smooth), we saw in Section~\ref{X:functoriality} that functoriality would then follow from \begin{enumerate}
    \item functoriality of the associated center $\sJ(I)_\bullet$ with respect to a logarithmically smooth, surjective morphisms $f \colon \widetilde{Y} \to Y$ (see Section~\ref{X:associated-center}),
    \item and functoriality of Satriano's construction (\textit{i.e.} the formation of the canonical smooth, toroidal Artin stack over a toroidal Artin stack) with respect to strict morphisms of toroidal Artin stacks $\mathfrak{f} \colon \widetilde{\fY} \to \fY$ (see Remark~\ref{R:satriano-construction}), e.g. the morphism $\tBl_{\underline{J}(f^{-1}I \cdot \sO_{\widetilde{Y}})_\bullet}\widetilde{Y} = \widetilde{Y} \times_Y \tBl_{\underline{J}(I)_\bullet}Y \to \tBl_{\underline{J}(I)_\bullet}Y$ obtained from $f$ by pull-back.
\end{enumerate} 
However, the next example demonstrates that Satriano's construction is \emph{not} necessarily functorial with respect to logarithmically smooth but non-strict morphisms, and hence the answer to the earlier question is no.

\begin{example}
Let $\Upgamma$ denote the sub-monoid in $\NN^3$ generated by $(1,0,0)$, $(0,1,0)$, $(1,0,1)$, and $(0,1,1)$. Let $\fY$ denote the toric $\kk$-variety $\Spec(\Upgamma \inj \kk[\Upgamma])$. The dual cone $C(\Upgamma)^\vee$ of $C(\Upgamma)$ has extremal rays $\be_1$, $\be_2$, $\be_3$, and $\be_1+\be_2-\be_3$. Thus, the monoid homomorphism $\iota \colon \Upgamma \inj \NN^4$, which maps every $\ba = (\ua_1,\ua_2,\ua_3) \in \Upgamma$ to $(\ua_1,\ua_2,\ua_3,\ua_1+\ua_2-\ua_3)$, then induces the canonical smooth, toroidal Artin stack over $Y$: \[
    \sY = \left[\Spec\left(\NN^4 \longinj \kk\left[\NN^4\right]\right) \q D(\Coker(\iota^\gp))\right] \lra \fY.
\]
Next, consider the diagrams \begin{equation}\label{E:functoriality-counterexample}
    \begin{tikzcd}
    \AA^{1;1} \arrow[to=2-1, equal] \arrow[to=1-2, dotted, "?"] & \sY \arrow[to=2-2] \\ 
    \widetilde{\fY} = \AA^{1;1} \arrow[to=2-2, "\mathfrak{f}"] & \fY\rlap{,}
    \end{tikzcd} \qquad \qquad \qquad \begin{tikzcd}
    \NN \arrow[to=2-1, equal] & \NN^4 \arrow[to=1-1, dotted, swap, "?"] \\
    \NN & \Upgamma\rlap{,} \arrow[to=2-1, swap, "\varphi"] \arrow[to=1-2, hookrightarrow, swap, "\iota"]
    \end{tikzcd}
\end{equation}
where: \begin{enumerate}
    \item $\AA^{1;1} := \Spec(\NN \inj \kk[\NN])$; 
    \item $\mathfrak{f} \colon \AA^{1;1} \to \fY$ is induced by the monoid homomorphism $\varphi \colon \Upgamma \to \NN$ which maps $\ba = (\ua_1,\ua_2,\ua_3) \in \Upgamma$ to $\ua_1+\ua_2$; \textit{i.e.} maps each generator of $\Upgamma$ to $1$; 
    \item and the vertical arrows are formations of smooth, toroidal Artin stacks over the respective toric $\kk$-schemes.
\end{enumerate}
Here, there are two toric morphisms $\AA^{1;1} \to \sY$ that could fill the dotted arrow in \eqref{E:functoriality-counterexample}! They are induced by \begin{enumerate}
    \item[(a)] $\NN \xleftarrow{\nu_1} \NN^4$, which maps $\bv = (\uv_i)_{i=1}^4 \in \NN^4$ to $\uv_1+\uv_2$,
    \item[(b)] $\NN \xleftarrow{\nu_2} \NN^4$, which maps $\bv = (\uv_i)_{i=1}^4 \in \NN^4$ to $\uv_3+\uv_4$.
\end{enumerate}
\end{example}

Nevertheless, we can still show the following. 

\begin{lemma}\label{L:functoriality-log-smooth}
The procedure of\, Theorem~\ref{T:invariant-drops-intro} $($and hence of\, Theorem~\ref{T:logarithmic-embedded-resolution}$)$ is functorial with respect to logarithmically smooth, equidimensional morphisms of smooth, toroidal Artin stacks over $\kk$.
\end{lemma}

\begin{proof}
As noted earlier, it remains to show that Satriano's construction is functorial with respect to logarithmically smooth, equidimensional morphisms $\mathfrak{f} \colon \widetilde{\fY} \to \fY$ of toroidal Artin stacks over $\kk$. Let us first consider the local setting where $\widetilde{\fY} = \Spec(Q \inj \kk[Q])$, $\fY = \Spec(P \inj \kk[P])$, and $\mathfrak{f}$ is induced by a monoid homomorphism $\varphi \colon P \to Q$. Consider the diagrams \begin{equation}\label{E:functoriality-log-smooth}
    \begin{tikzcd}
    \widetilde{\sY} \arrow[to=2-1] \arrow[to=1-2, dotted, "\exists!"] & \sY \arrow[to=2-2] \\ 
    \widetilde{\fY} \arrow[to=2-2, "\mathfrak{f}"] & \fY\rlap{,}
    \end{tikzcd} \qquad \qquad \qquad \begin{tikzcd}
    \NN^{S_Q} & \NN^{S_P} \arrow[to=1-1, dotted, swap, "\exists!"] \\
    Q \arrow[to=1-1, hookrightarrow, "\iota_Q"] & P\rlap{,} \arrow[to=2-1, swap, "\varphi"] \arrow[to=1-2, hookrightarrow, swap, "\iota_P"]
    \end{tikzcd}
\end{equation}
where: \begin{enumerate}
    \item the vertical arrows are formations of smooth, toroidal Artin stacks $\widetilde{\sY}$ and $\sY$ over the respective strict toroidal $\kk$-schemes $\widetilde{\fY}$ and $\fY$; 
    \item $S_P$ (resp.\ $S_Q$) is the set of extremal rays of $C(P)^\vee$ (resp.\ $C(Q)^\vee$); 
    \item and $\iota_P$ (resp.\ $\iota_Q$) is the morphism that maps every $\ba \in P$ (resp.\ $\ba \in Q$) to $\sum_{\rho \in S_P}{(\ba \cdot \bu_\rho) \cdot \be_\rho}$ (resp.\ $\sum_{\theta \in S_Q}{(\ba \cdot \bu_{\theta}) \cdot \be_\theta}$).
\end{enumerate}
We now claim that there is one and only one toric morphism $\widetilde{\sY} \to \sY$ filling the dotted arrow in \eqref{E:functoriality-log-smooth}. Indeed, a monoid homomorphism $\nu = (c_{\theta,\rho})_{\rho \in S_P,\ \theta \in S_Q} \colon \NN^{S_P} \to \NN^{S_Q}$ fills up the dotted arrow in \eqref{E:functoriality-log-smooth} if and only if \[
    \varphi^\vee(\bu_\theta) = \sum_{\rho \in S_P}{c_{\theta,\rho} \bu_\rho} \quad \textrm{for every } \theta \in S_Q,
\]
where $\varphi^\vee \colon C(Q)^\vee \to C(P)^\vee$ is the dual morphism to $\varphi \colon C(P) \to C(Q)$. In light of this observation, it suffices to show that if $f$ is equidimensional, then for every $\theta \in S_Q$, there is one and only one way of expressing $\varphi^\vee(\bu_\theta)$ as an $\NN$-linear combination of the vectors $\bu_\rho$ for $\rho \in S_P$. This is a mere consequence of \cite[Lemma 4.1]{abramovich-karu-weak-semistable-reduction}: if $f$ is equidimensional, the image of every extremal ray in $C(Q)^\vee$ under $\varphi^\vee$ is an extremal ray in $C(P)^\vee$. 

Finally, using Kato's criterion for logarithmic smoothness, see \cite[Section 8]{kato-toric-singularities}, the above argument can be patched up to settle the case where $\widetilde{\fY}$ and $\fY$ are any strict toroidal $\kk$-schemes. From this and Remark~\ref{R:satriano-construction}, the lemma can be deduced in the general case where $\widetilde{\fY}$ and $\fY$ are toroidal Artin stacks.
\end{proof}

\begin{remark} Note that unless $\widetilde{\fY} \to \fY$ is strict (see Remark~\ref{R:satriano-construction}), the induced morphism $\widetilde{\sY} \to \widetilde{\fY} \times_\fY \sY$ is typically not an isomorphism. This already happens when $\fY$ is a simplicial toric variety and $\widetilde{\fY}$ is not. Instead, the above morphism exhibits $\widetilde{\sY}$ as Satriano's construction applied to the fibre product $\widetilde{\fY} \times_\fY \sY$ in the category of fs logarithmic Artin stacks.
\end{remark}

\renewcommand\thesection{\Alph{section}}
\setcounter{section}{0}

\section*{Appendix. A lemma on quotient stacks}
\addcontentsline{toc}{section}{Appendix. A lemma on quotient stacks}
\refstepcounter{section}
\setcounter{subsection}{1}

\begin{applemma}[\textrm{$=$ \cite[Lemma 1.3.1]{quek-rydh-weighted-blow-up}}]\label{L:slicing}
Let $A$ be a finitely generated abelian group, with corresponding
diagonalizable algebraic group $D(A)$. Let $R = \bigoplus_{\alpha \in
  A}{R_\alpha}$ be an $A$-graded algebra, and let $r \in R$ be a
homogeneous element of degree $a \in A$. Then $R/(r-1)$ is an
$A/\langle a \rangle$-graded algebra, and the $A/\langle a \rangle$-graded homomorphism $R
\to R/(r-1)$ induces a morphism of algebraic stacks
\[
    \left[\left(\Spec(R/(r-1)\right) \q D\left(A/\langle a \rangle\right)\right] \xrightarrow{\;\lowsimeq\;} [\Spec(R) \q D(A)].
\]
This is an isomorphism if $r$ is invertible and $a$ has infinite order.
\end{applemma}

Note that as $A$-graded modules $R \simeq R/(r-1)[r,r^{-1}]$, but the algebra
structures do not coincide. Similarly, $R/(r-1) \simeq \bigoplus_{[\alpha] \in A/\langle a \rangle} R_\alpha$
but only as $A/\langle a \rangle$-graded modules. 

\begin{proof}
We need to prove that the natural $D(A)$-equivariant map
\[
\Spec\left(R/(r-1)\right) \times^{D(A/\langle a \rangle)} D(A) \lra \Spec(R)
\]
is an isomorphism. Let us elaborate on the left-hand side. We have two commuting actions on ${\Spec\left(R/(r-1)\right) \times D(A)} = \Spec(R/(r-1)[v^A]) := \Spec(R/(r-1)[v^\alpha \colon \alpha \in A])$:
\begin{enumerate}
\item the diagonal $D(A/\langle a \rangle)$-action, given by $(y,t) \cdot s=(ys,s^{-1}t)$, where in the first factor the action corresponds to the induced $A/\langle a \rangle$-grading on $R/(r-1)$, and
\item the $D(A)$-action on the second factor given by $(y,t) \cdot s=(y,ts)$.
\end{enumerate}
The $D(A/\langle a \rangle)$-action is free with quotient $\Spec\left(R/(r-1)\right) \times^{D(A/\langle a \rangle)} D(A) = \Spec(R^\circ)$, where $R^\circ$ is the degree~$0$ part of $R/(r-1)[v^A]$ with the $A/\langle a \rangle$-grading. The $D(A)$-action endows $R^\circ$ with the following $A$-grading: 
\[
R^\circ = \bigoplus_{\alpha \in A} \left(R/(r-1)\right)_{[\alpha]} v^\alpha.
%% = \left(\bigoplus_{[\alpha] \in A/\langle a \rangle}
%% \left(R/(r-1)\right)_\alpha v^\alpha\right)\left[v^{\langle a \rangle}\right]
\]
The natural $A$-graded algebra homomorphism $R \to R^\circ$ is thus an isomorphism.
\end{proof}

%%%%%%%%%%%%%%%%%%%%%
% References
%%%%%%%%%%%%%%%%%%%%%

\newcommand{\etalchar}[1]{$^{#1}$}

\end{document}